\documentclass[a4paper,10pt]{article}
\usepackage[margin=0.80in,bottom=0.90in,top=0.85in]{geometry}

\usepackage[english]{babel}
\usepackage{chngcntr}
\usepackage[utf8]{inputenc}
\usepackage{amsmath,mathrsfs}
\usepackage{indentfirst}
\usepackage{fancyhdr}
\usepackage{amssymb}
\usepackage{amsthm}
\usepackage[dvips]{graphicx}
\usepackage{color}
\usepackage{xcolor}
\usepackage{cancel}
\usepackage{subcaption}
\usepackage[dvipsnames]{xcolor}
\usepackage{eucal}
\usepackage{latexsym}
\usepackage{chngcntr}
\usepackage{faktor}
\usepackage{microtype}
\usepackage{newlfont}
\usepackage{enumitem}
\usepackage{geometry}
\usepackage{nicematrix}
\usepackage{tikz}
\usepackage{todonotes}

\usepackage{amsmath, amsfonts, amsthm,amssymb,  bbm}

\usepackage{tikz-cd}
\usepackage{multicol}
\usepackage{cancel}
\usepackage{mathtools}
\usepackage{graphicx}

\definecolor{mytealblue}{rgb}{0.12,0.47,0.71}
\usepackage[colorlinks=true,
            citecolor=mytealblue,
            linkcolor=mytealblue,
            urlcolor=mytealblue]{hyperref}

\newcommand{\e}{\epsilon}
\newcommand{\cO}{\mathcal{O}}

\newcommand{\ck}{\mathtt{c}_{\kappa}}
\newcommand{\tth}{\mathtt{h}}
\newcommand{\ttf}{\mathtt{f}_\e}
\newcommand{\pa}{\partial}

\newcommand{\cA}{{\mathcal A}}
\newcommand{\cJ}{{\mathcal J}}

\newcommand{\N}{{\mathbb{N}}}
\newcommand{\R}{{\mathbb{R}}}
\newcommand{\Z}{{\mathbb{Z}}}
\newcommand{\T}{{\mathbb{T}}}

\newcommand{\vet}[2]{\begin{bmatrix}#1 \\ #2 \end{bmatrix}}
\newcommand{\cB}{\mathcal{B}}

\newcommand{\fR}{\mathfrak{R}}

\newcommand{\im}{\textup{i}}

\newcommand{\und}[1]{\underline{#1}}
\newcommand{\red}[1]{{\textcolor{black}{#1}}}

\newcommand{\cL}{\mathcal{L}}

\newcommand{\vertiii}[1]{{\left\vert\kern-0.25ex\left\vert\kern-0.25ex\left\vert #1 
    \right\vert\kern-0.25ex\right\vert\kern-0.25ex\right\vert}}
\newcommand{\opnorm}[1]{{\vert\kern-0.25ex\vert\kern-0.25ex\vert #1 
    \vert\kern-0.25ex\vert\kern-0.25ex\vert}}

\theoremstyle{plain}
\newtheorem{lemma}{Lemma}
\newtheorem{theorem}[lemma]{Theorem}

\newtheorem{remark}[lemma]{Remark}
\newtheorem{definition}[lemma]{Definition}

\theoremstyle{definition} 

\numberwithin{equation}{section}
\counterwithin{lemma}{section}

\title{Spectral structure of the Benjamin--Feir instability \\ in deep-water gravity–capillary Stokes waves}

\begin{document}

\author{Ting-Yang Hsiao\footnote{International School for Advanced Studies (SISSA), Via Bonomea 265, 34136, Trieste, Italy. (Corresponding authors)  \newline
	\textit{Emails:} \texttt{thsiao@sissa.it}, \texttt{xiwang@sissa.it}}, \  Xinyang Wang$^*$}

\maketitle

\begin{abstract}
We investigate the Benjamin--Feir instability of small-amplitude gravity-capillary Stokes waves in deep water for the full water-wave equations. While modulational instability has been classically predicted by formal asymptotic approaches, such as nonlinear Schrödinger approximations, a complete spectral description at the level of the Euler equations has remained open. We perform a rigorous Bloch--Floquet spectral analysis of the linearized operator and describe the splitting of the multiple eigenvalues at the origin. In the unstable regime, we identify a pair of eigenvalues with non-zero real part forming the characteristic ``figure-eight'' pattern in the complex plane. As a consequence, we recover sharp near-origin Benjamin--Feir instability and stability regions in terms of the surface tension parameter, thereby providing a fully rigorous justification of the classical predictions in the gravity-capillary setting.
\end{abstract}

\tableofcontents

\section{Introduction}\label{sec1}
In 1847, Stokes \cite{stokes} first studied two-dimensional periodic traveling waves for the pure gravity water-wave equations, namely solutions that propagate at constant speed without changing shape. Such solutions, now called \emph{Stokes waves}, represent one of the earliest examples of global-in-time solutions for dispersive PDEs. The extensive literature on this topic is surveyed in \cite{ebb}.

A fundamental question concerns their stability under long-wave perturbations. In 1967, Benjamin and Feir \cite{BF} discovered that small-amplitude Stokes waves in deep water are unstable with respect to long-wavelength perturbations, a phenomenon now known as \emph{Benjamin--Feir} (or modulational) instability. Similar conclusions were obtained independently by Lighthill \cite{Li} and Zakharov \cite{Zak1}.

Despite these early developments, a rigorous spectral analysis of this instability remained a major challenge for several decades. The first mathematical proof was obtained by Bridges--Mielke \cite{BrM} in finite depth, using spatial dynamics and center manifold reduction. More recently, alternative approaches based on Evans function techniques have been developed, see e.g. Hur--Yang \cite{HY}. In contrast, the infinite-depth case is substantially more delicate due to the lack of compactness, and a complete proof of modulational instability was obtained only recently by Nguyen--Strauss \cite{NS}. In the last few years, new spectral methods have led to a much more detailed understanding of the instability mechanism. In particular, Berti--Maspero--Ventura \cite{BMV1, BMV3,BMV_ed} provided a complete description of the $L^2(\mathbb{R})$-spectrum of the linearized operator near the origin, both in deep water and in finite depth, showing that the unstable eigenvalues trace a characteristic ``figure-eight'' pattern in the complex plane.

On the other hand, gravity-capillary waves are commonly observed in nature. 
In the 1970s, classical works in nonlinear wave theory by Djordjevic--Redekopp \cite{DjRe} and Ablowitz--Segur \cite{AbSe}, carried out in the finite-depth setting, showed that surface tension modifies the dispersion relation, in particular by increasing the group velocity $c_g$ (see \eqref{id.DjRe} for a comparison with our notation). 
As a consequence, the balance between dispersion and nonlinearity is altered, leading to significant changes in the onset and structure of modulational instability. More recently, rigorous results have been obtained by Hur--Yang \cite{HY2}, Sun--Wahlen \cite{SW}, and Hsiao--Maspero \cite{HM}, which confirm that, in the instability regions identified by Djordjevic--Redekopp, the linearized operator exhibits unstable spectrum near the origin. In finite depth, Hsiao--Maspero \cite{HM} further
established the complete splitting of the four eigenvalues
near the origin and described the geometry of the unstable
branches.
The present paper provides the corresponding complete
description of the Benjamin--Feir spectrum for
gravity-capillary Stokes waves in deep water.

In this paper, we prove that, for every nonresonant $\kappa\geq0$ with $\kappa\neq\kappa_c$, the four eigenvalues of the linearized water-wave problem near the origin undergo a complete splitting for sufficiently small amplitude $\epsilon$ and Floquet exponent $\mu$, yielding a full spectral description of the Benjamin--Feir instability. 
The instability is governed by the Whitham--Benjamin function $\mathsf{e}_{\mathrm{WB}}(\kappa)$ \eqref{Def:eWB}, which changes sign at $\kappa=\kappa_c:=2/\sqrt{3}-1$ and is singular at the resonant value $\kappa=1/2$, separating the unstable region $\mathcal{U}$ \eqref{def:cU} from the stable region $\mathcal{S}$ \eqref{def:cS}. 
In particular, for $\kappa\in\mathcal{U}$, the unstable eigenvalues trace a complete ``figure-eight'' in the complex plane (see, Figure~\ref{fig:eight_combined}), providing a rigorous spectral description of the Benjamin--Feir mechanism beyond the formal modulational regime. This reveals that surface tension fundamentally alters the modulational instability mechanism. 
In contrast to the pure gravity case, where small-amplitude deep-water Stokes waves are always modulationally unstable near the origin, the presence of surface tension introduces a nontrivial dependence on the capillarity parameter $\kappa$. In particular, the critical value $\kappa = \kappa_c$ marks the transition between stable and unstable regimes, while the resonant value $\kappa = 1/2$ corresponds to a singular configuration where the spectral structure degenerates. 
These thresholds characterize the change in the qualitative behavior of the spectrum and determine the onset and geometry of the Benjamin--Feir instability. Let us now present precisely our results.
 
\paragraph{Benjamin--Feir instability for gravity-capillary water-waves.} We consider the two-dimensional water-wave equations governing an inviscid fluid of infinite depth under the influence of gravity \(g\) (which we normalize to \(g=1\)) and surface tension \(\kappa \geq 0\).  
We study the linear stability of a \(2\pi\)-periodic Stokes wave of small amplitude \(0<\epsilon\ll 1\) and speed $c_\epsilon = \ck + \mathcal{O}(\epsilon^2)$, where $\ck = \sqrt{1+\kappa}$. Linearizing the water-wave equations around this Stokes wave and passing to the co-moving frame with speed \(c_\epsilon\), we obtain a linear, time-independent system of the form $h_t = \mathcal{L}_\epsilon h$, where \(\mathcal{L}_\epsilon = \mathcal{L}_\epsilon(\kappa)\) is a linear operator with \(2\pi\)-periodic coefficients (see \eqref{mathcal L e}).  
This operator is obtained by conjugating the linearized system in Zakharov variables via the good unknown of Alinhac \eqref{alinhac} and the invertible Levi--Civita transformation \eqref{LC}. The operator $\mathcal{L}_\epsilon$ possesses a defective eigenvalue at $0$ of algebraic multiplicity four, arising from the symmetries of the water-wave equations.  
The problem is to determine whether the linear system admits exponentially growing solutions of the form $h(t,x) = \mathrm{Re}\big(e^{\lambda t} e^{\mathrm{i}\mu x} v(x)\big)$, where $v(x)$ is $2\pi$-periodic, $\mu \in \mathbb{R}$ is the Floquet exponent, and $\lambda$ has positive real part. By Bloch--Floquet theory, such $\lambda$ corresponds to an eigenvalue of the operator $\mathcal{L}_{\mu,\epsilon} := e^{-\mathrm{i}\mu x}\,\mathcal{L}_\epsilon\,e^{\mathrm{i}\mu x}$
acting on \(2\pi\)-periodic functions.

The main result of this paper establishes, for every nonresonant $\kappa\geq0$ with $\kappa\neq\kappa_c$, the complete splitting of the four eigenvalues of \(\mathcal{L}_{\mu,\epsilon} = \mathcal{L}_{\mu,\epsilon}(\kappa)\) near zero, for sufficiently small \(\epsilon\) and $\mu$; see Theorem \ref{Complete BF thm}.  
We first state Theorem \ref{thm:main}, which describes the ``figure-eight'' structure formed by the Benjamin--Feir unstable eigenvalues. This figure-eight arises precisely when $\kappa$ belongs to a certain open set $\mathcal{U}$ (defined in \eqref{def:cU}), previously identified by Djordjević--Redekopp \cite{DjRe} and Ablowitz--Segur \cite{AbSe} through formal modulational analysis with deep-water limit. Our results provide a rigorous justification of the Benjamin--Feir instability mechanism beyond the formal modulational regime.

Let us introduce the Whitham--Benjamin function of $\kappa$ as follows:
\begin{align} \label{Def:eWB}
    \mathsf{e}_{\mathrm{WB}}:=\mathsf{e}_{\mathrm{WB}}(\kappa)
    = \frac{6\,\kappa ^4+15\,\kappa ^3+28\,\kappa ^2+47\,\kappa -8}
    {8\,(2\,\kappa -1)\,(\kappa +1)^2}.
\end{align}
We observe that $\mathsf{e}_{\mathrm{WB}}$ is singular at $\kappa = 1/2$. 
To characterize such singular (resonant) values, we introduce the set
\begin{equation}\label{def:fR}
    \fR := \bigcup_{n \geq 2} \fR_n 
    := \left\{ \kappa \in \R_{\geq 0} \;:\; \kappa = \frac{1}{n} \right\}.
\end{equation}
In particular, 
\begin{equation}\label{denneq0}
    2\kappa - 1 \neq 0 \quad \text{whenever} \quad \kappa \notin \fR_2.
\end{equation}
We recall that resonance occurs when $\kappa = 1/n$, in which case distinct Fourier modes share the same dispersion relation, namely
\begin{align*}
    (1+\kappa k^2)k - \ck^2 k^2 = 0,
    \quad \text{for } k = 1 \text{ and } k = n,
\end{align*}
where $\ck$ denotes the phase speed of the linear gravity-capillary wave. 
At such resonant values, the linearized operator becomes degenerate and allows the coexistence of multiple harmonics; this phenomenon is known as \emph{Wilton ripples} (see, for example, \cite{reeder1981wiltonI, reeder1981wiltonII}). In the present work, we exclude all resonant values $\kappa \in \fR_n$ with $n \geq 2$. 
These non-resonance conditions are necessary for the construction of $2\pi$-periodic Stokes waves; see the discussion following Theorem~\ref{Thm: Stokes expansion}. We define the \emph{unstable} region
\begin{equation}\label{def:cU}
    \mathcal{U}
    := \left\{ \kappa \in \R_{\geq 0} \setminus \fR 
    \;:\; \kappa < \frac{2\sqrt{3}}{3}-1 
    \ \text{or} \ 
    \kappa > \frac{1}{2} \right\},
\end{equation}
and the \emph{stable} region
\begin{equation} \label{def:cS}
    \mathcal{S}
    := \left\{ \kappa \in \R_{\geq 0} \setminus \fR 
    \;:\; \frac{2\sqrt{3}}{3}-1 < \kappa < \frac{1}{2} \right\}.
\end{equation}

Our main result shows that, for $\kappa \in \mathcal{U}$, the spectrum of the linearized water wave operator $\cL_{\mu,\e}$ near the origin exhibits a complete ``figure-eight'' configuration. 
In contrast, for $\kappa \in \mathcal{S}$, the spectrum near the origin remains purely imaginary. In Figure~\ref{fig: eWB}, we plot the Whitham--Benjamin function $\mathsf{e}_{\mathrm{WB}}(\kappa)$. 
In particular, $\mathsf{e}_{\mathrm{WB}}(\kappa)$ vanishes at 
\[
\kappa_c := \frac{2\sqrt{3}}{3} - 1=0.1547005\ldots,
\]
and is singular at $\kappa = 1/2$. 
Consequently, it is negative for $\kappa_c < \kappa < \frac{1}{2}$, and positive for $0 \leq \kappa < \kappa_c$ or $\kappa > \frac{1}{2}$. These results are consistent with those obtained by Djordjevic--Redekopp \cite{DjRe} and Ablowitz--Segur \cite{AbSe} in the formal infinite-depth limit. 
Further discussion of this connection will be given after Theorem~\ref{thm:main}.

Finally, we introduce the following functions:
\begin{align} \label{e11e22e12}
    \mathsf{e}_{11} :=\mathsf{e}_{11}(\kappa):=\frac{2\kappa^2+\kappa+8}{8(1-2\kappa)\ck},\quad \mathsf{e}_{22}:=\mathsf{e}_{22}(\kappa):=\frac{-3\kappa^2-6\kappa+1}{\ck^3},\quad\mathsf{e}_{12}:=\mathsf{e}_{12}(\kappa):=\frac{1+3\kappa}{\ck}.
\end{align}
Both Djordjevic--Redekopp \cite{DjRe} and Ablowitz--Segur \cite{AbSe} formally established the linear instability of Stokes waves under long-wave perturbations via a modulational approximation leading to a nonlinear Schr\"odinger equation. 
To facilitate a precise comparison with our results, we relate the coefficients $\mathsf{e}_{11}(\kappa)$, $\mathsf{e}_{22}(\kappa)$, and $\mathsf{e}_{12}(\kappa)$ to those appearing in \cite{DjRe,AbSe}.
More specifically, let $ \nu(\kappa):=\lim_{\tth\rightarrow \infty}\nu(\tth,\kappa) $ and 
$\lambda(\kappa):=\lim_{\tth\rightarrow \infty}\lambda (\tth,\kappa)$, as defined in \cite[formula (2.17)]{DjRe}, and $c_g(\kappa):=\lim_{\tth\rightarrow \infty} c_g(\tth,\kappa)$ as in \cite[(2.7)]{DjRe}. Then the following identities hold:
\begin{equation}\label{id.DjRe}
\mathsf{e}_{11}(\kappa)=\frac{(1+\kappa)\nu(\kappa)}{2}   \ , \qquad \mathsf{e}_{22}(\kappa) = - 8 \lambda (\kappa)\ , \qquad \mathsf{e}_{12}(\kappa)=2 c_g(\kappa)\,.
\end{equation}
Here we have normalized the physical parameters by setting the gravitational constant $g=1$ and the wavenumber $k=1$, in agreement with the conventions adopted in this paper. The identities \eqref{id.DjRe} can be readily verified (e.g., using MATLAB R2024b).
   In particular, our  instability criterion in \eqref{def:cU}, namely 
   $\mathsf{e}_{\mathrm{WB}}(\kappa) > 0$, is equivalent to 
   $
   \nu(\kappa)\lambda (\kappa) < 0$ 
   which coincides with the condition derived in \cite[Section 3]{DjRe}. 

\begin{figure}[h!]
    \centering         
    \includegraphics[width=0.7\textwidth]{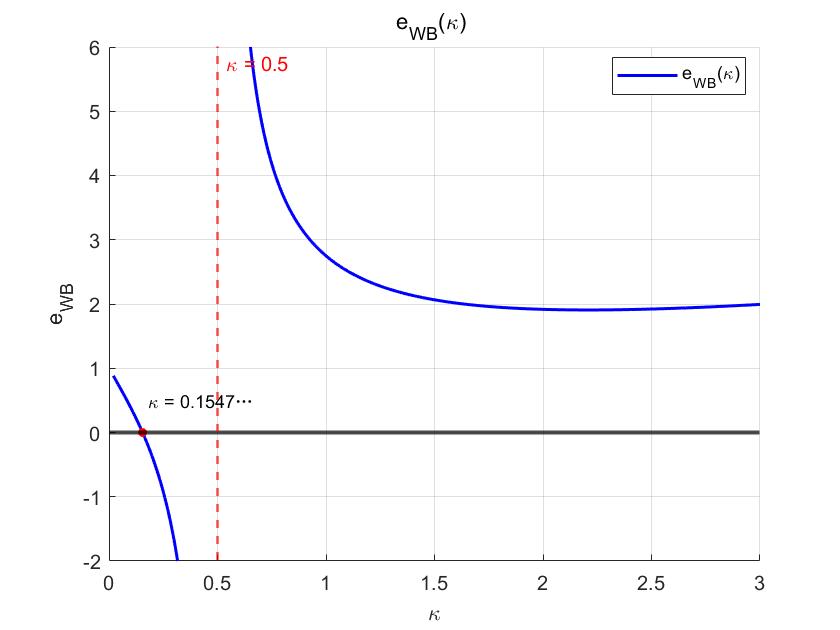} 
    \caption{Plot of the Whitham--Benjamin function $\mathsf{e}_{\mathrm{WB}}(\kappa)$. 
The sign of $\mathsf{e}_{\mathrm{WB}}$ changes at $\kappa=\kappa_c=0.1547005\ldots$ and $\kappa=1/2$, 
in agreement with the deep-water limit obtained by Djordjevic--Redekopp \cite[Section 3]{DjRe} via the NLS approximation.}\label{fig: eWB} 
\end{figure}

\begin{figure}[h!]
    \centering
    \begin{minipage}{0.9\textwidth}
        \centering
        \includegraphics[width=\linewidth]{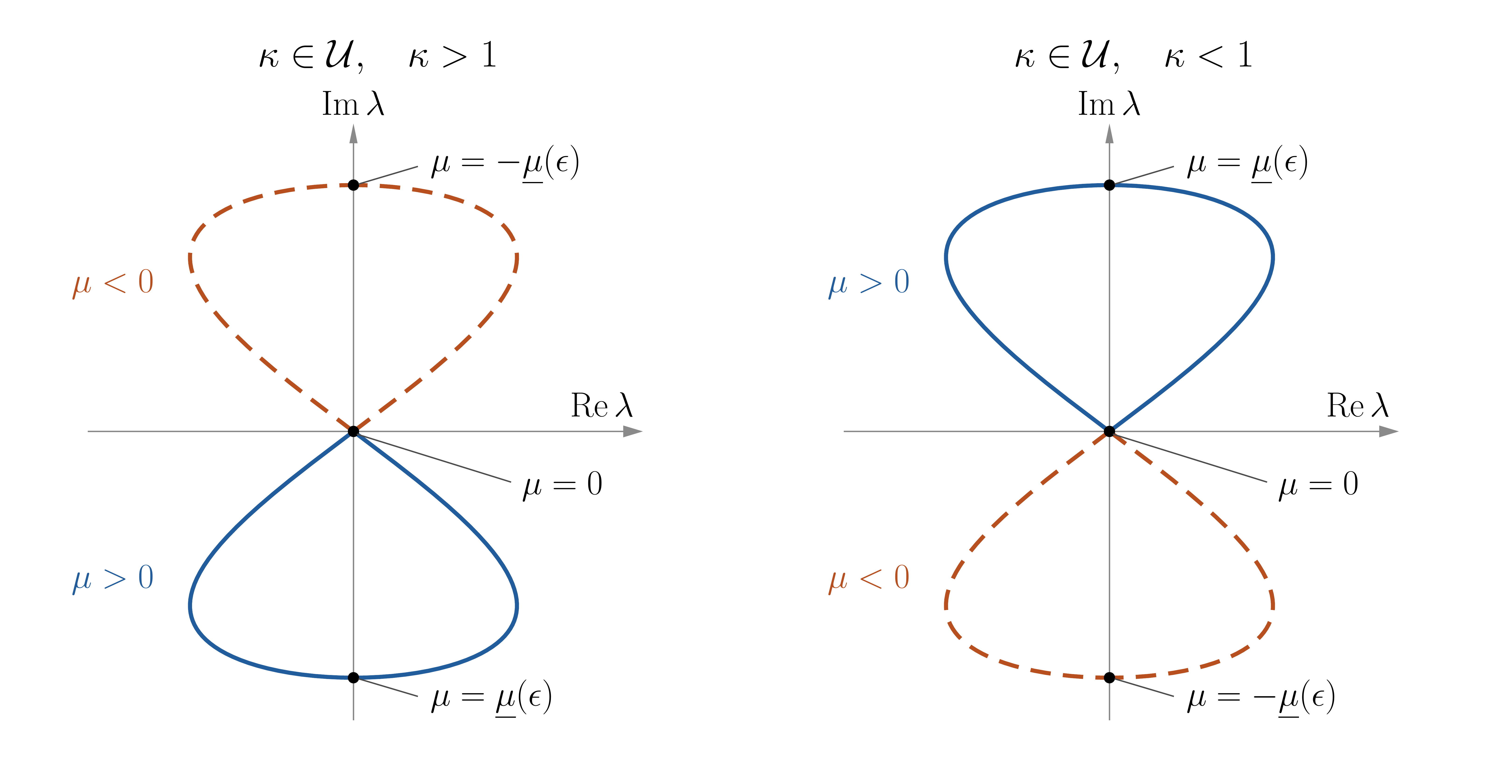}
    \end{minipage}
    
    \caption{Stability results shown in Theorem~\ref{thm:main}. Traces of the eigenvalues $\lambda_1^{\pm}(\mu,\e)$ in the complex $\lambda$-plane at fixed $|\e|\ll1$ as $\mu$ varies. If $\kappa\in\mathcal{U}$ and $\kappa > 1$ (Left), for $\mu \in (0,\underline{\mu}(\e))$ the eigenvalues fill the portion of the 8 in $\{\mathrm{Im}(\lambda) < 0\}$ and for $\mu \in (-\underline{\mu}(\e),0)$ the symmetric portion in $\{\mathrm{Im}(\lambda) > 0\}$. If $\kappa\in\mathcal{U}$ and $\kappa < 1$ (Right), for $\mu \in (0,\underline{\mu}(\e))$ the eigenvalues fill the portion of the 8 in $\{\mathrm{Im}(\lambda) > 0\}$ and for $\mu \in (-\underline{\mu}(\e),0)$ the symmetric portion in $\{\mathrm{Im}(\lambda) < 0\}$.}
\label{fig:eight_combined}
\end{figure}

Throughout the paper, we denote by
$r(\epsilon^{m_1}\mu^{n_1},\ldots,\epsilon^{m_p}\mu^{n_p})$
a real analytic function satisfying
\[
\left|
r(\epsilon^{m_1}\mu^{n_1},\ldots,\epsilon^{m_p}\mu^{n_p})
\right|
\leq C\sum_{j=1}^{p}
|\epsilon|^{m_j}|\mu|^{n_j}
\]
for sufficiently small $\epsilon$ and $\mu$.
The constant $C$ and the smallness thresholds may depend
on the fixed nonresonant parameter $\kappa$.
For $\kappa>0$, these estimates are uniform on compact
subsets of
\[
(0,\infty)\setminus
\bigl(\mathfrak{R}\cup\{\kappa_c\}\bigr),
\qquad
\kappa_c=\frac{2\sqrt{3}}{3}-1.
\]
The smallness thresholds may shrink as $\kappa$
approaches the resonant set or $\kappa_c$.
The estimates at $\kappa=0$ are understood separately,
without asserting uniformity as $\kappa\to0^+$.

\begin{theorem}[Benjamin--Feir unstable eigenvalues] \label{thm:main} 
Let  $\kappa\in \mathcal{U} $, cf. \eqref{def:cU}.
There exist $\epsilon_1, \mu_0 > 0$ and an analytic function $\und{\mu} : [0, \epsilon_1) \to [0, \mu_0)$,
of the form
\begin{align} \label{und mu}
\und{\mu}(\epsilon) =  \epsilon (1 + r(\epsilon))\sqrt{\frac{8\,\mathsf{e}_{11}(\kappa)}{\mathsf{e}_{22}(\kappa)}}, 
\end{align}
such that, for any $\epsilon \in [0, \epsilon_1)$, the operator $\mathcal{L}_{\mu, \epsilon}$ has two eigenvalues $\lambda^{\pm}_1(\mu, \epsilon)$ of the form
\begin{equation}\label{eigelemu}
\begin{cases}
\im\frac{1}{2}\breve{\mathtt{c}}_{\kappa}\mu + \im r_{2}(\mu \epsilon^{2}, \mu^{2}\epsilon, \mu^{3}) 
\pm \frac{1}{8} \mu (1+r(\epsilon,\mu)) \, \sqrt{\Delta_{\mathrm{BF}}(\kappa,\mu,\epsilon)}, 
& \forall \mu \in [0, \und{\mu}(\epsilon)), \\[1em]
\im\frac{1}{2}\breve{\mathtt{c}}_{\kappa}\mu + \im r_{2}(\mu \epsilon^{2}, \mu^{2}\epsilon, \mu^{3}), 
& \mu = \und{\mu}(\epsilon), \\[1em]
\im\frac{1}{2}\breve{\mathtt{c}}_{\kappa}\mu + \im r_{2}(\mu \epsilon^{2}, \mu^{2}\epsilon, \mu^{3}) 
\pm \im\frac{1}{8} \mu 
(1+r(\epsilon,\mu)) \, \sqrt{|\Delta_{\mathrm{BF}}(\kappa,\mu,\epsilon)|}, 
& \forall \mu \in (\und{\mu}(\epsilon), \mu_0),
\end{cases}
\end{equation}
where $\breve{\mathtt{c}}_{\kappa} := 2 \ck - \mathsf{e}_{12}(\kappa)$ and 
$\Delta_{\mathrm{BF}}(\kappa; \mu, \epsilon)$ is the \emph{Benjamin--Feir discriminant function}
\begin{align}\label{BFDF}
    \Delta_{\mathrm{BF}}(\kappa;\mu,\epsilon):=8\,\mathsf{e}_{\mathrm{WB}}(\kappa)\e^2+r_1(\e^3,\mu\e^4)-\mathsf{e}^2_{22}(\kappa)\mu^2\left(1+r''_1(\e,\mu)\right),
\end{align}
where $\mathsf{e}_{\mathrm{WB}}(\kappa)$, $\mathsf{e}_{22}(\kappa)$, and $\mathsf{e}_{12}(\kappa)$ in \eqref{e11e22e12}.
Note that, since  $\kappa\in \mathcal{U} $, for any $0 < \epsilon < \epsilon_1$ (depending on $\kappa$), 
the function $\Delta_{\mathrm{BF}}(\kappa; \mu, \epsilon)>0$ is positive, respectively $<0$, 
provided $0 < \mu < \und{\mu}(\epsilon)$, respectively $\mu > \und{\mu}(\epsilon)$.
\end{theorem}

Let $\kappa\in\mathcal{S}$, where $\mathcal{S}$
is defined in \eqref{def:cS}.
There exist $\epsilon_0,\mu_0>0$, depending on $\kappa$,
such that, for every $0\leq\epsilon<\epsilon_0$
and $|\mu|<\mu_0$, all four eigenvalues of
$\mathcal{L}_{\mu,\epsilon}$ bifurcating from the origin
are purely imaginary.
Indeed, since $\mathsf{e}_{\mathrm{WB}}(\kappa)<0$
and $\mathsf{e}_{22}(\kappa)\neq0$, \eqref{BFDF} gives
\[
\Delta_{\mathrm{BF}}(\kappa;\mu,\epsilon)<0
\qquad
\text{for }0<\mu<\mu_0,\quad
0\leq\epsilon<\epsilon_0,
\]
after reducing $\epsilon_0$ and $\mu_0$ if necessary.
The assertion for $\mu>0$ then follows from
Theorem~\ref{Complete BF thm}.
The cases $\mu=0$ and $\mu<0$ follow by continuity
and the symmetry
$\sigma(\mathcal{L}_{-\mu,\epsilon})
=\overline{\sigma(\mathcal{L}_{\mu,\epsilon})}$,
respectively.

Let us make some comments:

\begin{enumerate}
\item {\sc Benjamin--Feir eigenvalues.}
Whenever $\kappa \in \mathcal{U}$, for Floquet parameters $0<\mu<\underline{\mu}(\epsilon)$, the eigenvalues 
$\lambda^\pm_1(\mu,\epsilon)$ have opposite non-zero real parts, according to \eqref{eigelemu}. The corresponding ``figure-eight'' is realized in the upper or lower half-plane provided that 
$\breve{\mathtt{c}}_{\kappa}\neq 0$ (i.e. $\kappa\neq 1$). When $0\leq\kappa<1$, the imaginary part in \eqref{eigelemu} is positive for $\mu>0$, and the eigenvalues $\lambda^\pm_1(\mu,\epsilon)$ lie in the upper half-plane $\mathrm{Im}(\lambda)>0$. As $\mu \to \underline{\mu}(\epsilon)$, the two eigenvalues $\lambda^\pm_1(\mu,\epsilon)$ collide on the imaginary axis away from the origin, and for $\mu>\underline{\mu}(\epsilon)$ they remain purely imaginary and move along it. In particular, for $\mu \in [0,\underline{\mu}(\epsilon)]$, one obtains the upper branch of the ``figure-eight'', see Figure \ref{fig:eight_combined}, which is well approximated by the curves
\begin{equation}\label{appdr}
\mu \mapsto \Big( \pm\frac{\mu}{8}  \sqrt{8\epsilon^2\mathsf{e}_{\mathrm{WB}}(\kappa)-\mu^2\mathsf{e}^2_{22}(\kappa)},  \ 
\frac{\mu}{2} \breve{\mathtt c}_{\kappa}  \Big) \,.
\end{equation}
For $\mu<0$, the operator $\mathcal{L}_{\mu,\epsilon}$ possesses the symmetric eigenvalues 
$\overline{\lambda_1^{\pm}(-\mu,\epsilon)}$ in the lower half-plane $\mathrm{Im}(\lambda)<0$, which generate the lower branch of the ``figure-eight''. Viceversa, in the region $\kappa>1$ where the function $\breve{\mathtt{c}}_{\kappa}<0$,   the imaginary part \eqref{appdr} is negative as $\mu>0$ and the ``figure-eight'' forms in the lower semiplane as 
$ \mu \in [0, \underline \mu(\e)]$  (of course the upper ``figure-eight'' is formed when $\mu <0$ by  symmetry). At the special value $\kappa=1$, we have
\[
\breve{\mathtt{c}}_{\kappa}
=2\ck-\mathsf{e}_{12}(\kappa)
=\frac{1-\kappa}{\ck}=0.
\]
Since $1\in\mathcal{U}$, the Benjamin--Feir instability
still occurs for $0<\mu<\underline{\mu}(\epsilon)$.
However, the leading imaginary drift in \eqref{eigelemu}
vanishes, and in this interval
\[
\operatorname{Im}\lambda_1^{\pm}(\mu,\epsilon)
=r_2(\mu\epsilon^2,\mu^2\epsilon,\mu^3).
\]
Thus the leading-order curves in \eqref{appdr} lie on
the real axis, while the imaginary parts of the exact
eigenvalues depend on higher-order terms.
The upper/lower half-plane description based on the
sign of $\breve{\mathtt{c}}_{\kappa}$ therefore applies
only when $\kappa\neq1$.
 
\item {\sc Relation with \cite{HM}.}
Our convention differs from that of \cite{HM}, since
\[
\mathsf{e}_{\mathrm{WB}}(\kappa)
=\mathsf{e}_{11}(\kappa)\mathsf{e}_{22}(\kappa).
\]
Denoting the Benjamin--Feir function in \cite{HM}
by $\widetilde{\mathsf{e}}_{\mathrm{WB}}(\kappa,\tth)$,
we have, for each fixed nonresonant $\kappa$,
\[
\lim_{\tth\to\infty}
\widetilde{\mathsf{e}}_{\mathrm{WB}}(\kappa,\tth)
=\mathsf{e}_{11}(\kappa).
\]
 
\item {\sc Comparison with previous results.}
At $\kappa=0$, our result recovers the deep-water
pure-gravity result of \cite{BMV1}:
$\mathsf{e}_{\mathrm{WB}}(0)=1$, and sufficiently small
nontrivial Stokes waves are Benjamin--Feir unstable.
In finite depth, \cite{BMV3} shows that pure-gravity
Stokes waves are modulationally stable below the
critical depth approximately equal to $1.36\ldots$ and
unstable above it, whereas \cite{HM} describes the
joint dependence of the stability regions on depth
and surface tension.
Our result determines the effect of surface tension
in deep water, yielding both the unstable region
$\mathcal{U}$ and the stable region $\mathcal{S}$. The complete spectrum in Theorem~\ref{Complete BF thm}
also contains two purely imaginary eigenvalues of
order $\mathcal{O}(\sqrt{\mu})$, in addition to the
two Benjamin--Feir eigenvalues of order $\mathcal{O}(\mu)$.
In the finite-depth problems of \cite{BMV3,HM},
all four eigenvalues are of order $\mathcal{O}(\mu)$.
Thus, although the instability coefficients are related
through the deep-water limit, the small-$\mu$ spectral
expansions are not uniform as the depth tends to infinity.
The present spectral description therefore requires
a direct analysis of the deep-water problem.

\end{enumerate}

Before closing the introduction, we place our results within the broader picture of the spectral and dynamical theory for water waves. Concerning the global structure of the spectrum, \cite{HY2, SW} show that, for suitable choices of depth and capillarity, a first isola of unstable spectrum may occur away from the origin. In addition, \cite{SW} proves that the spectrum of the linearized operator is purely imaginary outside a sufficiently large ball. It would be of significant interest to combine these results with the analysis of isolae developed in \cite{BCMV}, in order to obtain a complete description of the full spectrum of the linearized gravity–capillary water-wave operator.

We also recall several complementary directions in the literature. Modulational instability under transverse perturbations has been investigated in \cite{CNS, CNS2, JRSY, HTW}. High-frequency instabilities for two-dimensional Stokes waves in the pure gravity case have been studied in \cite{HY, BMV4, BCMV}, together with numerical investigations in \cite{Mc1, Mc2, CD, BDS, CDT, DT}. Nonlinear aspects of modulational instability have been addressed in \cite{ChenSu}, while instability mechanisms driven by energy cascades have been analyzed in \cite{LMMT, MM_ARMA}. Finally, various shallow-water and approximate models have been considered in \cite{HP, HJ, BHJ, BJ, GH, HK, BBHZ, BHR, MR, BMP2025}.

\section{The full Benjamin--Feir spectrum of gravity-capillary waves}
In this section we present the complete spectral Theorem \ref{Complete BF thm}.
First let us introduce the gravity-capillary water-wave equations and the Stokes wave solutions.      \\
{\bf The water-wave equations.} 
We consider the Euler equations of hydrodynamics for a two-dimensional perfect and incompressible fluid  under the action of gravity and capillary forces at the free surface.
The fluid fills the region 
    \begin{align*}
       \mathcal{D}_\eta:=\{(x,y)\in \mathbb{R}\times \mathbb{R}: -\infty < y <\eta(t,x)\}\ ,
    \end{align*}
with infinite depth. The irrotational velocity
field is the gradient of a harmonic scalar potential  $\Phi = \Phi(t,x,y)$ determined by  its trace $\psi(t,x):= \Phi(t, x,\eta(t,x))$ at the free surface $y=\eta(t,x)$. 
Actually $\Phi$ is
the unique solution of the elliptic equation $\Delta \Phi  = 0$ in $\mathcal{D}_\eta$  with Dirichlet datum
$\Phi(t, x,\eta(t, x))=\psi(t, x)$ and $\Phi_y(t,x,y) = 0$ at $y = - \infty$. 

 Imposing that the fluid particles
at the free surface remain on it along the evolution (kinematic boundary condition) and that the pressure of
the fluid plus the capillary forces at the free surface is equal to the constant atmospheric pressure (dynamic boundary condition), the time evolution of the fluid is determined by the non-local quasi-linear equations \cite{Zak68, CS}
\begin{equation} \label{Craig-Sulem formula}
\left\{\begin{aligned}
    \eta_t&=G(\eta)\psi\\
    \psi_t&=-g\eta-\frac{\psi_x^2}{2}+\frac{1}{2(1+\eta^2_x)}\left(G(\eta)\psi+\eta_x\psi_x\right)^2+\kappa\left(\frac{\eta_x}{(1+\eta_x^2)^{\frac{1}{2}}}\right)_x,
\end{aligned}\right.
\end{equation}
where $g>0$ is the gravity constant and $G(\eta)$ denotes the Dirichlet-Neumann operator $[G(\eta)\psi](x):=\Phi_y(x,\eta(x))-\Phi_x(x,\eta(x))\eta_x(x)$. 
 In the sequel,
with no loss of generality, we set the gravity constant $g=1$.

System \eqref{Craig-Sulem formula} is Hamiltonian and can be written as
\begin{equation} \label{pa_t eta psi}
    \begin{aligned}
        \partial_t \begin{bmatrix}
\eta\\
\psi 
\end{bmatrix}=
\mathcal{J}\begin{bmatrix}
\nabla_\eta \mathcal{H}\\
\nabla_\psi \mathcal{H} 
\end{bmatrix}
    \end{aligned},~~\mathcal{J}:=
    \begin{bmatrix}
0& \mathrm{Id}\\
-\mathrm{Id} &0
\end{bmatrix},
\end{equation}
where $\nabla$ denote the $L^2$-gradient, and the Hamiltonian 
\begin{align} \label{mathcal H}
\mathcal{H}(\eta,\psi):=\int_{\mathbb{T}} \left(\frac{1}{2}\psi G(\eta)\psi+\frac{1}{2}g\eta^2+\kappa(\sqrt{1+\eta^2_x}-1)\right) dx    
\end{align}
is the sum of the kinetic and potential energy of the fluid. 
In addition,  the water waves system \eqref{Craig-Sulem formula} is reversible with respect to the involution
\begin{equation} \label{rho involution}
    \begin{aligned}
        \rho \begin{bmatrix}
            \eta(x)\\
            \psi(x)
        \end{bmatrix}:=
        \begin{bmatrix}
            \eta(-x)\\
            -\psi(-x)
        \end{bmatrix},~~\mathrm{i.e.}~\mathcal{H}\circ \rho=\mathcal{H},
    \end{aligned}
\end{equation}
and it is space invariant.

 \paragraph{Stokes waves.} In a moving reference frame with constant speed $c$ (and with normalized gravity $g=1)$, the water waves system \eqref{Craig-Sulem formula} becomes
 \begin{equation} \label{in the reference frame}
\left\{\begin{aligned}
    \eta_t&=c\eta_x+G(\eta)\psi\\
    \psi_t&=c\psi_x-\eta-\frac{\psi_x^2}{2}+\frac{1}{2(1+\eta^2_x)}\left(G(\eta)\psi+\eta_x\psi_x\right)^2+\kappa\left(\frac{\eta_x}{(1+\eta_x^2)^{\frac{1}{2}}}\right)_x.
\end{aligned}\right.
\end{equation}
We consider small-amplitude {\em Stokes wave solutions}, namely stationary solutions of \eqref{in the reference frame} which we further require to be  $2\pi$-periodic in space.
 The bifurcation of small-amplitude Stokes waves from the trivial solution was first studied for pure gravity water waves by Stokes \cite{stokes}, Levi-Civita \cite{LC}, Nekrasov \cite{Nek}, Struik \cite{Struik}. 
For $\kappa>0$, the existence and analyticity of the
bifurcating branch follow from the analytic
Crandall--Rabinowitz theorem, by adapting the argument
of \cite[Appendix A]{HLZZ} to zero bending rigidity.
Here the traveling-wave map takes values in
$H^{\sigma,s-1}_{\mathtt{odd}}(\T)
\times H^{\sigma,s-2}_{\mathtt{ev}}(\T)$.
The nonresonance condition ensures a one-dimensional
kernel, and the transversality condition follows
by differentiation with respect to the wave speed.
The case $\kappa=0$ is covered by
\cite[Theorem 1.3]{BMV2}. We denote $\mathbb{T}:=\mathbb{R}/(2\pi\mathbb{Z})$ and $B(r):=\{x\in\mathbb{R}: |x|<r\}$ the open ball of radius $r$ centered at zero.
 \begin{theorem} [Stokes waves] \label{Thm: Stokes expansion} 
 Let $\kappa \in \R_{\geq 0}  \setminus \fR$, with  $\fR$  in \eqref{def:fR}. There 
 exist $\e_*:=\e_*(\kappa) >0$ and a unique family  of real analytic 
 solutions $(\eta_\e(x), \psi_\e(x), c_\e)$, parameterized by the amplitude $|\e| \leq \e_*$, of 
\begin{equation}\label{travelingWWstokes}
c \, \eta_x+G(\eta)\psi = 0 \, , \quad 
c \, \psi_x -  \eta - \dfrac{\psi_x^2}{2} + 
\dfrac{1}{2(1+\eta_x^2)} \big( G(\eta) \psi + \eta_x \psi_x \big)^2  +
\kappa\left(\frac{\eta_x}{(1+\eta_x^2)^{\frac{1}{2}}}\right)_x
= 0 \, , 
\end{equation}
  such that
 $ \eta_\e (x), \psi_\e (x) $ are $2\pi$-periodic;  $\eta_\e (x) $ is even
and $\psi_\e (x) $ is odd, of the form 
 \begin{equation}\label{exp:Sto}
 \begin{aligned}
  & \eta_\e (x) = \e \cos(x)+\e^2\eta_2^{[2]}\cos(2x) +\cO(\e^3)  , \\ 
  & \psi_\e (x)  =  \e \ck \sin(x)+\e^2 \psi_2^{[2]}\sin(2x)+\cO(\e^3),  \\
  & c_\e = \ck  +\e^2 c_2+\cO(\e^3)\quad \text{where} \quad \ck := (1+\kappa)^{\frac{1}{2}}\, , 
   \end{aligned}
  \end{equation}
and 
\begin{align}\label{expcoef}
 &\eta_{2}^{[2]} := \frac{\ck^2}{2\left( 1 -2\,\kappa\right) } , \qquad \psi_{2}^{[2]} :=  \frac{\ck^3}{2\left( 1 -2\,\kappa\right) }, \qquad c_2 := \frac{2\kappa^2+\kappa+8}{16\ck\left( 1 -2\,\kappa\right)}.
\end{align}
More precisely for any  $ \sigma \geq  0 $ and $ {s > \frac72} $, there exists $ \e_*>0 $ such that
the map $\e \mapsto (\eta_\e, \psi_\e, c_\e)$ is analytic from $B(\e_*) \to H^{\sigma,s}_{\mathtt{ev}} (\T)\times H^{\sigma,s}_{\mathtt{odd}}(\T)\times \R$, where 
$ H^{\sigma,s}_{\mathtt{ev}}(\T) $, respectively $ H^{\sigma,s}_{\mathtt{odd}}(\T) $, denote the  space of even, respectively odd, 
 real valued $ 2 \pi $-periodic analytic functions
$ u(x) = \sum_{k \in \mathbb{Z}} u_k e^{\im k x} $
such that $ \| u \|_{\sigma,s}^2 := \sum_{k \in \mathbb{Z}} |u_k|^2 \langle k \rangle^{2s} 
e^{2 \sigma |k|} < + \infty$. 
\end{theorem}

The expansions in 
\eqref{exp:Sto}--\eqref{expcoef} are proved in Appendix \ref{sec:App2}. 
The condition $\kappa\not \in \fR$ is used in Lemma \ref{lem:B0} to ensure that the kernel of the linearized operator at the flat surface is one-dimensional. 
Indeed $\kappa \not \in \fR$ 
is equivalent to requiring that  the function
\begin{equation}\label{speedn}
n\mapsto \sqrt{\frac{1+\kappa n^2}{n}} \ , 
\end{equation}
which is the quotient between the dispersion relation and the wave number, is injective on $\N$, which is enough for constructing Stokes waves with the fixed spatial period of $2\pi$. Remark that when $\kappa \in \fR$,  higher-order resonances happen. Nevertheless, traveling waves -- going under the name of Wilton ripples-- can be constructed  \cite{Wilton,reeder}.

\paragraph{Linearization.}
We linearize system
\eqref{in the reference frame} at the Stokes waves $(\eta_\epsilon(x),\psi_\epsilon(x))$ given in Theorem \ref{Thm: Stokes expansion} and evaluate $c$ at $c_\epsilon$. 
Using the shape derivative formula \cite{LD, LD book} 
$
\mathrm{d}_\eta G(\eta)[\hat \eta][\psi] = - G(\eta)(B\hat \eta) - \pa_x( V \hat \eta), 
$
where
the functions $(V(x),B(x))$ are the horizontal and vertical components of the velocity field $(\Phi_x,\Phi_y)$ at the free surface and are given by
\begin{align} \label{espV}
   V:= V(x)&:=-B(\eta_\epsilon)_x+(\psi_\epsilon)_x,\\ \label{espB}
  B:=  B(x)&:=\frac{G(\eta_\epsilon)\psi_\epsilon+(\psi_\epsilon)_x(\eta_\epsilon)_x}{1+(\eta_\epsilon)^2_x}=\frac{(\psi_\epsilon)_x-c_\epsilon}{1+(\eta_\epsilon)^2_x} (\eta_\epsilon)_x,
\end{align}
one obtains the real, autonomous linearized  system
\small
\begin{equation} \label{first linear eq}
    \begin{aligned}
        \begin{bmatrix}
            \hat{\eta}_t\\
            \hat{\psi}_t
        \end{bmatrix}=
        \left[\begin{array}{c|c} 
	 -G(\eta_\epsilon)B-\partial_x\circ(V-c_\epsilon) & G(\eta_\epsilon)  \\ 
	\hline 
	-1+B(V-c_\epsilon)\partial_x-B\partial_x\circ(V-c_\epsilon)-BG(\eta_\epsilon)\circ B+\kappa \partial_x \circ l \circ\partial_x & -(V-c_\epsilon)\partial_x+B G(\eta_\epsilon)
\end{array}\right] 
        \begin{bmatrix}
            \hat{\eta}\\
            \hat{\psi}
        \end{bmatrix},
    \end{aligned}
\end{equation}
\normalsize
where
\begin{align} \label{esp l}
   l:= l(x):=\dfrac{1}{(1+(\eta_\e)_x^2)^{\frac{3}{2}}}.
\end{align}
The map  $\e \to (V, B, l)$ is analytic as a map $B(\e_0) \to H^{\sigma, s-1}_{\mathtt{ev}}(\T) \times H^{\sigma, s-1}_{\mathtt{odd}}(\T) \times H^{\sigma, s-1}_{\mathtt{ev}}(\T) $. 
The real system \eqref{first linear eq} is Hamiltonian, i.e. of the form
$\cJ \cA$ with $\cA$ symmetric, $\cA = \cA^\top$, where the transpose is taken with respect to the real scalar product of $L^2(\T, \R) \times L^2(\T, \R)$.
Moreover the linear operator in \eqref{first linear eq} is reversible, namely it anti-commutes with the involution $\rho$ in \eqref{rho involution}. \\
Next, we conjugate \eqref{first linear eq} by using the time-independent ``good unknown of Alinhac'' linear transformation 
\begin{align}\label{alinhac}
    \begin{bmatrix}
        \hat{\eta}\\
        \hat{\psi}
    \end{bmatrix}:= Z
    \begin{bmatrix}
        u\\
        v
    \end{bmatrix},~~Z=\begin{bmatrix}
        1&0\\
        B&1
    \end{bmatrix},
    ~~Z^{-1}=\begin{bmatrix}
        1&0\\
        -B&1
    \end{bmatrix},
\end{align}
yielding the linear system 
\begin{align} \label{second linear eq}
    \begin{bmatrix}
        u_t\\
        v_t
    \end{bmatrix}=  \widetilde{\cL}_\e 
    \begin{bmatrix}
        u\\
        v
    \end{bmatrix} \ , \qquad 
    \widetilde{\cL}_\e := \begin{bmatrix}
        -\partial_x\circ (V-c_\epsilon)& G(\eta_\epsilon)\\
        -1-(V-c_\epsilon)B_x+\kappa \partial_x\circ l\circ \partial_x ~~& -(V-c_\epsilon)\partial_x
    \end{bmatrix} \ , 
\end{align}
which  is Hamiltonian and reversible since the transformation $Z$ is symplectic, $\mathrm{i.e.}$ $Z^T\mathcal{J} Z=\mathcal{J}$, and satisfies $Z\circ \rho=\rho\circ Z$.

Next, 
we perform a conformal change of variables to flatten 
the water surface. 
By \cite[Appendix A]{BBHM}, 
 there exists a diffeomorphism of $\mathbb{T}$,
 $ x\mapsto x+\mathfrak{p}(x)$, with a small $2\pi$-periodic  function $\mathfrak{p}(x)$ such that, by defining the associated composition operator $ (\mathfrak{P}u)(x) := u(x+\mathfrak{p}(x))$, the Dirichlet-Neumann operator writes as \cite[Lemma A.5]{BBHM}
\begin{equation}\label{Gneta}
 G(\eta_\e) = \pa_x \circ \mathfrak{P}^{-1} \circ {\mathfrak H} 
 \circ \mathfrak{P} \, , 
\end{equation}
where $ {\mathfrak H} $ is the Hilbert transform, i.e. the  Fourier multiplier operator
$$
 \mathfrak{H}(e^{\im j x}):= - \im\, \textup{sign}(j) e^{\im j x} \, , 
 \quad  \forall j \in \Z \setminus \{0\} \, , 
 \quad \mathfrak{H}(1) := 0 \, . 
$$
The function $\mathfrak p(x)$ is determined as a fixed point of 
(see  \cite[formula (A.15)]{BBHM})
\begin{equation} \label{def:ttf}
\mathfrak{p}  = \mathfrak{H}[\eta_\e ( x + \mathfrak{p}(x))]  
  \end{equation}
  As proved in \cite{BMV3}, the map $\e \to \mathfrak{p}$ is analytic as a map $B(\e_0) \to H^s_{\mathtt{odd}}(\T) $. We give a short proof for completeness. Fix $\sigma>0$ and $s>7/2$. By
Theorem \ref{Thm: Stokes expansion}, the map
$\epsilon\mapsto\eta_\epsilon$ is analytic with values
in $H^{\sigma,s}_{\mathtt{ev}}(\T)$.
The spatial analyticity of $\eta_\epsilon$ and the
Banach algebra property of $H^s(\T)$ imply that
\[
\Phi(\epsilon,p)
:=p-\mathfrak H[\eta_\epsilon\circ(\mathrm{Id}+p)]
\]
is analytic near $(0,0)$ in
$\mathbb R\times H^s_{\mathtt{odd}}(\T)$,
with values in $H^s_{\mathtt{odd}}(\T)$.
Since $\eta_0=0$, we have
$\Phi(0,0)=0$ and $D_p\Phi(0,0)=\mathrm{Id}$. The analytic implicit function theorem therefore
yields an analytic map
$\epsilon\mapsto\mathfrak p_\epsilon
\in H^s_{\mathtt{odd}}(\T)$
solving \eqref{def:ttf}.
The spatial analyticity supplied by
Theorem \ref{Thm: Stokes expansion} is used here
to justify analytic composition; Sobolev regularity
suffices for the subsequent analysis.
In addition, in  Appendix \ref{sec:App2} we prove the expansion
\begin{equation}
 \label{expfe}
 \begin{aligned}
&  \mathfrak p(x)  = \e  \sin(x) +\e^2 \frac{2-\kappa}{2\,\left( 1 -2\,\kappa\right)} \sin(2x)+\cO(\e^3) \, ,
   \end{aligned}
 \end{equation}
Under the symplectic and reversibility-preserving change of variables 
\begin{align}\label{LC}
    h=\mathcal{P}\begin{bmatrix}
        u\\
        v
\end{bmatrix},~~\mathcal{P}=\begin{bmatrix}
    (1+\mathfrak{p}_x)\mathfrak{P} & 0\\
    0 & \mathfrak{P}
\end{bmatrix} \ , 
\end{align}
one  transforms the system \eqref{second linear eq} into the linear system $h_t=\mathcal{L}_\epsilon h$ where $\mathcal{L}_\epsilon$ is the Hamiltonian and reversible real operator
\begin{equation} \label{mathcal L e}
\begin{aligned}
    \mathcal{L}_\epsilon:= \mathcal{L}_{\epsilon}(\kappa) :=\mathcal{P} \widetilde{\cL}_\e \mathcal{P}^{-1} = &\begin{bmatrix}
\partial_x\circ(\ck+p_\epsilon(x)) &  |D| \\
        -(1+a_\epsilon(x))+\kappa \Sigma_\e & ~~(\ck+p_\epsilon(x))\partial_x
    \end{bmatrix}\\
    =&\begin{bmatrix}
        0& \mathrm{Id}\\
        -\mathrm{Id}& 0
    \end{bmatrix}\begin{bmatrix}
 (1+a_\epsilon(x))-\kappa \Sigma_\e& ~~-(\ck+p_\epsilon(x))\partial_x\\
 \partial_x\circ(\ck+p_\epsilon(x))& ~~|D| 
    \end{bmatrix},
\end{aligned}    
\end{equation}
where the functions $p_\e(x)$ and $a_\e(x)$ are given by
\begin{equation}\label{def:pa}
\ck+p_\e(x) :=  \displaystyle{\frac{ c_\e-V(x+\mathfrak{p}(x))}{ 1+\mathfrak{p}_x(x)}} \, , \quad 1+a_\e(x):=   \displaystyle{\frac{1+ (V(x + \mathfrak{p}(x)) - c_\e)
 B_x(x + \mathfrak{p}(x))  }{1+\mathfrak{p}_x(x)}} \,, 
\end{equation}
and the operator $\Sigma_\e$ is given by
\begin{equation}\label{def:Sigma g}
\begin{aligned}
  &   \Sigma_\e: =\frac{1}{1+\mathfrak{p}_x}\partial_x \circ g_\e(x)\circ \partial_x \circ \frac{1}{1+\mathfrak{p}_x}\,,\qquad g_\e(x): = \frac{l(x+\mathfrak{p}(x))}{1+\mathfrak{p}_x}\,.
\end{aligned}    
\end{equation}
By the analyticity result of the map $\e \mapsto (V, B, l)$ given above,  the map
$\e \to (p_\e, a_\e, g_\e)$ is analytic as a map $B(\e_0) \to H^s_{\mathtt{ev}}(\T)\times H^s_{\mathtt{ev}}(\T) \times H^s_{\mathtt{ev}}(\T)$.
In Appendix \ref{sec:App2} we provide their Taylor expansions, that we collect in the following lemma:
\begin{lemma}\label{lem:pa.exp}
The analytic functions $p_\e (x) $ and $a_\e (x) $  in \eqref{def:pa} 
are even in $ x $, and
\begin{equation}\label{SN1}
p_\e (x)  
= \e p_1 (x) + \e^2 p_2 (x)  + \cO(\e^3) \, , \qquad   
a_\e (x)  
= \e a_1(x) +\e^2 a_2 (x) + \cO(\e^3) \, , 
\end{equation}
where
\begin{align}\label{pino1fd}
     p_1(x)  =   p_1^{[1]} \cos(x)\, , \qquad p_2(x) =p_2^{[0]}+p_2^{[2]}\cos(2x)\, ,
\end{align}
\begin{align} \label{pino2fd}
     p_1^{[1]} := -2\ck\,, \qquad p_2^{[0]} := 
    \frac{-30\kappa^2-15\kappa+24
}{   16\ck (1-2\kappa)}\, , \qquad \quad p_2^{[2]} :=-\frac{2\ck^3}{1-2\kappa},
\end{align}
and 
 \begin{align} \label{aino1fd}
a_1(x)  = a_1^{[1]}\cos(x)\, , \qquad \qquad a_2(x)=a_2^{[0]}+a_2^{[2]}\cos(2x)\, ,
\end{align}
\begin{align}
\label{aino2fd} 
a_1^{[1]}:= -(2+\kappa)\, ,\qquad a_2^{[0]}:=\frac{4+3\kappa}{2}\, ,\qquad a_2^{[2]} := -\frac{10\kappa^2+11\kappa+4}{2(1-2\kappa)}.
\end{align}  
The function $g_\e(x)$ in \eqref{def:Sigma g} is even in $x$ and expands as
\begin{equation} \label{SN1 g}
    g_\e(x) = 1+ \e g_1(x)  + \e^2 g_2(x) + \cO(\e^3),
\end{equation}
where
\begin{align}\label{exp:g1}
    g_1(x) &= -\cos(x),\, \qquad  g_2(x)=-\frac{1}{4}-\frac{6\kappa+3}{4(1-2\kappa)}\cos(2x)\,.
\end{align}    
Finally the self-adjoint operator $\Sigma_\e$ in \eqref{def:Sigma g} expands as
\begin{equation} \label{Sigma epsilon}
    \Sigma_\e = \pa^2_{x} + \e \Sigma_1 + \e^2 \Sigma_2 + \cO(\e^3), 
\end{equation}
where
\begin{align}\label{Sigma 1}
    \Sigma_j:=& d_j(x)\,\pa^2_{x}+e_j(x)\pa_{x} + h_j(x), \quad j=1,2 
\end{align}
and
\begin{align}
\label{d1x}
    d_1(x)&=-3 \cos(x)\,, \quad 
     e_1(x) = 3 \sin(x) \,,\quad h_1(x)=\cos(x)\,,  \\
    d_2(x)&= d_2^{[0]} + d_2^{[2]} \cos(2x)  \,, \quad 
    d_2^{[0]}:= \frac{9}{4} \,, 
    \ \ 
    d_2^{[2]}:= -\frac{9+18\kappa}{4(1-2\kappa)}\,,
    \\
     e_2(x)& =  e_2^{[2]} \sin(2x)  \,,  \ \ 
     e_2^{[2]} := \frac{18\kappa+9}{2(1-2\kappa)}\,,
     \\
     h_2(x)& = h_2^{[0]} + h_2^{[2]} \cos(2x) \,, \ \ 
     h_2^{[0]}:= -\frac{1}{2} \,, 
    \ \ \label{h2 h02}
    h_2^{[2]}:=\frac{9+6\kappa}{2(1-2\kappa)}.
\end{align}
\end{lemma}

\paragraph{Bloch--Floquet expansions.} 
Since the operator $\mathcal{L}_\e$ in \eqref{mathcal L e} has $2\pi$-periodic coefficients, Bloch--Floquet theory  \cite{RS78} implies that $\lambda\in\mathbb{C}$ belongs to the $L^2(\mathbb{R})-$spectrum of $\mathcal{L}_\e$ if and only if there exists a nontrivial Bloch--Floquet mode $\tilde{h}(x)=e^{\im \mu x} v(x)$, where $v$ is $2\pi-$periodic and $\mu\in[-\frac{1}{2},\frac{1}{2})$, such that $\lambda \tilde{h}=\mathcal{L}_\e \tilde{h}$, or equivalently, $\lambda v=e^{-\im \mu x}\mathcal{L}_\e e^{\im \mu x}v$. Therefore, we obtain
\begin{align*}
    \sigma_{L^2(\mathbb{R})}(\mathcal{L}_\e)=\bigcup_{\mu\in[-\frac{1}{2},\frac{1}{2})} \sigma_{L^2(\mathbb{T})}(\mathcal{L}_{\mu,\e})~~\mbox{where}~~~\mathcal{L}_{\mu,\e}:=e^{-\im\mu x}\mathcal{L}_\e e^{\im\mu x}.
\end{align*}
In particular, if $\lambda$ is an eigenvalue of $\mathcal{L}_{\mu,\e}$ on $L^2(\mathbb{T},\mathbb{C}^2)$ with eigenvector $v(x)$, then $h(t,x)=e^{\lambda t}e^{\im \mu x}v(x)$ solves $h_t=\mathcal{L}_\e h$. We pause to remark that: \\
$(i)$ if $A=Op(a)$ is a pseudo-differential operator with symbol $a(x,\xi)$, which is $2\pi$-periodic in $x$, then $A_\mu:=e^{-\im\mu x}A e^{\im\mu x}=Op(a(x,\xi+\mu))$. \\
$(ii)$ If $A$ is a real operator then $\overline{A_\mu}=A_{-\mu}$. As a consequence the spectrum $\sigma(A_{-\mu})=\overline{\sigma(A_\mu)}$ and we can study $\sigma(A_\mu)$ just for $\mu>0$.\\
$(iii)$  $\sigma(A_\mu)$ is a $1$-periodic set with respect to $\mu$, so one can restrict to $\mu\in[0,\frac{1}{2})$.

By the previous remarks the Floquet operator associated with the real operator $\mathcal{L}_\e$ in \eqref{mathcal L e} is the complex Hamiltonian and reversible operator
\begin{equation} \label{mathcal L mu e}
\begin{aligned}
    \mathcal{L}_{\mu,\e}:=&\begin{bmatrix}
(\partial_x+\im\mu)\circ(\ck+p_\epsilon(x)) &  |D+\mu| \\
        -(1+a_\epsilon(x))+\kappa \Sigma_{\mu,\e} & ~~(\ck+p_\epsilon(x))(\partial_x+\im\mu)
    \end{bmatrix}\\
    =&\underbrace{\begin{bmatrix}
        0& \mathrm{Id}\\
        -\mathrm{Id}& 0
    \end{bmatrix}}_{=:\mathcal{J}}\underbrace{\begin{bmatrix}
 (1+a_\epsilon(x))-\kappa \Sigma_{\mu,\e}& ~~-(\ck+p_\epsilon(x))(\partial_x+\im\mu)\\
 (\partial_x+\im\mu)\circ(\ck+p_\epsilon(x))& ~~|D+\mu| 
    \end{bmatrix}}_{=:\mathfrak{B}_{\mu,\e}},
\end{aligned}    
\end{equation}
where $\Sigma_{\mu,\e}:=e^{-\im\mu x}\Sigma_\e e^{\im\mu x}$ is given by the selfadjoint operator
\begin{align}\label{Sigma}
    \Sigma_{\mu,\e}=\frac{1}{1+\mathfrak{p}_x}(\partial_x + \im \mu) \circ g_\e(x)\circ (\partial_x+\im\mu) \circ \frac{1}{1+\mathfrak{p}_x} \ . 
\end{align}

For the operator analysis below, we fix
$\kappa\in(0,\infty)\setminus\mathfrak{R}$.
The case $\kappa=0$ of our spectral results follows
from \cite{BMV1}, where the natural operator domain is
$H^1(\mathbb{T},\mathbb{C}^2)$. We regard $\mathcal{L}_{\mu,\e}$ as an operator with domain $Y:=H^2(\mathbb{T},\mathbb{C})\times H^1(\mathbb{T},\mathbb{C})$ and range $X:=L^2(\mathbb{T},\mathbb{C})\times L^2(\mathbb{T},\mathbb{C})$, equipped with the complex scalar product
\begin{align} \label{complex product}
    (f,g):=\frac{1}{2\pi}\int_0^{2\pi} \left(f_1\overline{g_1}+f_2\overline{g_2} \right)\,\mathrm{d}x,~~\forall~f=\vet{f_1}{f_2},~~g=\vet{g_1}{g_2}\in L^2(\mathbb{T},\mathbb{C}^2).
\end{align}
We also denote $\|f\|^2=(f,f)$.

The complex operator $\mathcal{L}_{\mu,\e}$ in \eqref{mathcal L mu e} is complex Hamiltonian and reversible. Recall that if $\mathcal{L} : Y \to X$ is a complex linear operator, we say that it is 
\begin{itemize}
    \item \textbf{Complex Hamiltonian:} if there exists a self-adjoint operator, namely $\mathcal{B}=\mathcal{B}^*$, where $\mathcal{B}^*$ (with domain $Y$) is the adjoint with respect to the complex scalar product \eqref{complex product} such that $\mathcal{L} = \mathcal{J} \mathcal{B}$.

    \item \textbf{Reversible:} if 
    \begin{align} \label{reversible}
    \mathcal{L} \circ \overline{\rho} = -\overline{\rho} \circ \mathcal{L}, \quad \text{where} \quad 
    \overline{\rho} \begin{bmatrix} \eta(x) \\ \psi(x) \end{bmatrix} := 
    \begin{bmatrix} \overline{\eta}(-x) \\ -\overline{\psi}(-x) \end{bmatrix} \ . 
    \end{align}
\end{itemize}

The property \eqref{reversible} for $\mathcal{L}_{\mu,\e}$ follows because $\mathcal{L}_\e$ is a real operator which is reversible with respect to the involution $\rho$ in \eqref{rho involution}. Equivalently, since $\mathcal{J}\circ \overline{\rho}=-\overline{\rho}\circ \mathcal{J}$, the self-adjoint operator $\mathfrak{B}_{\mu,\e}$ is reversibility-preserving, $\mathrm{i.e.}$
\begin{align}\label{B rho=rho B}
    \mathfrak{B}_{\mu,\e}\circ\overline{\rho}=\overline{\rho}\circ \mathfrak{B}_{\mu,\e}.
\end{align}
Here and in what follows, we always consider
$\mu\in[0,\frac12)$.
As in \cite[Section 5.1]{NS}, we use the identity
\begin{align} \label{D+mu}
    |D+\mu|=|D|+\mu(\mathrm{sgn}(D)+\Pi_0),~~{\forall \mu\in[0,\frac{1}{2})},
\end{align}
where $\mathrm{sgn}(D)$ is the Fourier multiplier operator, acting on $2\pi$-periodic functions, with symbol
\begin{equation} \label{sgn D}
\mathrm{sgn}(k):=1 \ \ \forall k >0 \ , \ \
\mathrm{sgn}(0):=0  , \ \ \ 
\mathrm{sgn}(k):= -1 \ \ \forall k < 0 \  , 
\end{equation}
and $\Pi_0$ is the projector operator on the zero mode, $\Pi_0 f(x):=\frac{1}{2\pi}\int_\mathbb{T} f(x) dx$.

\smallskip 

Our goal is to prove the existence of eigenvalues of $\mathcal{L}_{\mu,\e}$ in \eqref{mathcal L mu e} with non zero real part. We remark that the Hamiltonian structure of $\mathcal{L}_{\mu,\e}$ implies that eigenvalues with non zero real part may arise only from multiple eigenvalues of $\mathcal{L}_{\mu,0}$ (``Krein criterion''), because if $\lambda$ is an eigenvalue of $\mathcal{L}_{\mu,\e}$ then also $-\overline{\lambda}$ is, and the total algebraic multiplicity of the eigenvalues is conserved under small perturbation. We now describe the spectrum of $\mathcal{L}_{\mu,0}$.

\paragraph{The spectrum of $\mathcal{L}_{\mu,0}$.} The spectrum of the Fourier multiplier matrix operator 
\begin{equation} \label{mathcal L mu 0}
\begin{aligned}
\mathcal{L}_{\mu,0}:=&\begin{bmatrix}
\ck(\partial_x+\im\mu)&  |D+\mu|\\
        -1+\kappa (\pa_x + \im \mu)^2 & ~~\ck(\partial_x+\im\mu)
    \end{bmatrix}
\end{aligned}    
\end{equation}
consists of the purely imaginary eigenvalues $\{\lambda^{\pm}_k(\mu),~k\in\mathbb{Z}\}$, where
\begin{align} \label{eigenvalues of Lmu0}
    \lambda^{\pm}_k(\mu):=\im\left(\ck(\pm k+\mu)\mp \sqrt{\left(1+\kappa(k\pm\mu)^2\right)\, |k\pm \mu|}\right) \ . 
\end{align}
For $\kappa \in\R_{\geq 0}  \setminus \fR$ (see \eqref{def:fR}), at $\mu=0$ the real operator $\mathcal{L}_{0,0}$ possesses the eigenvalue $0$ with algebraic multiplicity $4$, 
\begin{align*}
    \lambda^+_0(0)=\lambda^-_0(0)=\lambda^+_1(0)=\lambda^-_1(0)=0,
\end{align*}
and geometric multiplicity $3$. A real basis of the kernel of $\mathcal{L}_{0,0}$ is 
\begin{align} \label{eigenfunc of mathcall L00}
    f^+_1:=\vet{\frac{1}{\sqrt{\ck}}\cos(x)}{\sqrt{\ck}\sin(x)}, ~~f^-_1:=\vet{-\frac{1}{\sqrt{\ck}}\sin(x)}{\sqrt{\ck}\cos(x)}, ~~f^-_0:=\vet{0}{1},
\end{align}
together with the generalized eigenvector
\begin{align} \label{eigenfunc f+0}
    f^+_0:=\vet{1}{0}, ~~\mathcal{L}_{0,0} f^+_0=-f^-_0.
\end{align}
Furthermore $0$ is an isolated eigenvalue for $\mathcal{L}_{0,0}$, namely the spectrum $\sigma(\mathcal{L}_{0,0})$ decomposes in two separated parts,
\begin{align} \label{sigma decomposition}
    \sigma(\mathcal{L}_{0,0})=\sigma'(\mathcal{L}_{0,0})\cup\sigma''(\mathcal{L}_{0,0}), ~~\mbox{where}~~\sigma'(\mathcal{L}_{0,0}):=\{0\},
\end{align}
and $\sigma''(\mathcal{L}_{0,0}):=\{\lambda^{\sigma}_k(0),~k\neq 0,1, ~\sigma=\pm\}$.

In addition, following the proof  in \cite[Lemma 4.1]{NS}, the operator $\mathcal{L}_{0,\e}$ possesses, for any sufficiently small $\e\neq 0$, the eigenvalue $0$ with a four dimensional generalized Kernel, spanned by $\e$-dependent vectors $U_1~,\Tilde{U}_2,~U_3,~U_4$ satisfying, for some real constant $\alpha_\e$, 
\begin{equation} \label{237}
    \begin{aligned}
        &\mathcal{L}_{0,\e} U_1=0, ~\mathcal{L}_{0,\e}\tilde{U}_2=0,~\mathcal{L}_{0,\e} U_3=\alpha_\e \tilde{U}_2,~\mathcal{L}_{0,\e} U_4=-U_1,~~U_1=\vet{0}{1}.
    \end{aligned}
\end{equation}
By Kato's perturbation theory for any $\mu,\e\neq 0$ sufficiently small, the perturbed spectrum $\sigma(\mathcal{L}_{\mu,\e})$ admits a disjoint decomposition as 
\begin{align} \label{disjoint decomposition of spectrum}
    \sigma(\mathcal{L}_{\mu,\e})=\sigma'(\mathcal{L}_{\mu,\e})\cup \sigma''(\mathcal{L}_{\mu,\e}),
\end{align}
where $\sigma'(\mathcal{L}_{\mu,\e})$ consists of $4$ eigenvalues close to $0$. We denote by $\mathcal{V}_{\mu,\e}$ the spectral subspace associated with $\sigma'(\mathcal{L}_{\mu,\e})$, which has dimension $4$ and it is invariant by $\mathcal{L}_{\mu,\e}$. Our goal is to prove that, for $\e$ small, for values of the Floquet exponent $\mu$ in an interval of order $\e$, the $4\times 4$ matrix which represents the operator $\mathcal{L}_{\mu,\e}:\mathcal{V}_{\mu,\e}\rightarrow \mathcal{V}_{\mu,\e}$ possesses a pair of eigenvalues close to zero with opposite non zero real parts. 

Before stating our main result, let us introduce a notation that we shall use throughout the paper.

\begin{itemize}
\item
\textbf{Notation:} we denote by $\mathcal{O}(\mu^{m_1} \epsilon^{n_1}, \dots, \mu^{m_p} \epsilon^{n_p})$, $m_j, n_j \in \mathbb{N}_0$ (for us $\mathbb{N}_0 := \{0,1,2, \dots \}$), analytic functions of $(\mu, \epsilon)$ with values in a Banach space $X$ which satisfy, for some $C > 0$, the bound
\[
\|\mathcal{O}(\mu^{m_j} \epsilon^{n_j})\|_X \leq C \sum_{j=1}^{p} |\mu|^{m_j} |\epsilon|^{n_j}
\]
for small values of $(\mu, \epsilon)$. Similarly we denote $r_k(\mu^{m_1} \epsilon^{n_1}, \dots, \mu^{m_p} \epsilon^{n_p})$ scalar functions $\mathcal{O}(\mu^{m_1} \epsilon^{n_1}, \dots, \mu^{m_p} \epsilon^{n_p})$ which are also \textit{real} analytic.
\end{itemize}

Our main spectral result is the following one (see also Figure~\ref{fig: four branches}):
\begin{theorem}[Complete Benjamin--Feir spectrum] \label{Complete BF thm}
Let $\kappa\in\mathbb R_{\ge0}\setminus
(\mathfrak R\cup\{\kappa_c\})$, where $\fR$ is defined in \eqref{def:fR}. There exist $\epsilon_{0}, \mu_{0} > 0$, such that, for any $0 < \mu < \mu_{0}$ and $0 \leq \epsilon < \epsilon_{0}$, the operator $\mathcal{L}_{\mu,\epsilon} : \mathcal{V}_{\mu,\epsilon} \to \mathcal{V}_{\mu,\epsilon}$ can be represented by a $4 \times 4$ matrix of the form
\begin{align} \label{U S diag}
\begin{pmatrix}
\mathsf{U} & 0 \\
0 & \mathsf{S}
\end{pmatrix},
\end{align}
where $U$ and $S$ are $2 \times 2$ matrices, with identical diagonal entries each, of the form
\begin{equation} \label{U S}
\begin{aligned}
\mathsf{U} &= 
\begin{pmatrix}
\im \left( \frac{1}{2}\breve{\mathtt{c}}_{\kappa} \mu + r_{2}(\mu \epsilon^{2}, \mu^{2}\epsilon, \mu^{3}) \right) 
& -\mathsf{e}_{22}\frac{\mu^2}{8}(1 + r_{5}(\epsilon,\mu)) \\
\mathsf{e}_{22}\frac{\mu^2}{8}(1+r''(\e,\mu))-\mathsf{e}_{11} \epsilon^2(1+r'(\e,\mu\e^2))  
& \im \left( \frac{1}{2}\breve{\mathtt{c}}_{\kappa} \mu + r_{2}(\mu \epsilon^{2}, \mu^{2}\epsilon, \mu^{3}) \right) 
\end{pmatrix},\\
\mathsf{S} &= 
\begin{pmatrix}
\im\ck\mu+\im r_{9}(\mu \epsilon^{2}, \mu^{2}\epsilon,\mu^3) & \mu + r_{10}(\mu^2 \epsilon,\mu^3) \\[1ex]
- (1+\kappa\mu^2) + r_{8}(\e^3,\mu \epsilon^{2}, \mu^2\e,\mu^{3} ) &\im\ck\mu +\im r_{9}(\mu \epsilon^{2}, \mu^{2}\epsilon,\mu^3)
\end{pmatrix},
\end{aligned}
\end{equation}
The eigenvalues of $\mathsf{U}$ have the form
\begin{align} \label{lambda 1}
\lambda_{1}^{\pm}(\mu,\epsilon) 
= \im\frac{1}{2}\breve{\mathtt{c}}_{\kappa}\mu + \im r_{2}(\mu \epsilon^{2}, \mu^{2}\epsilon, \mu^{3}) 
\pm \frac{1}{8} \mu \sqrt{1+r_{5}(\epsilon,\mu)} \, \sqrt{\Delta_{\mathrm{BF}}(\kappa,\mu,\epsilon)},
\end{align}
where $\Delta_\mathrm{BF}(\kappa;\mu,\epsilon)$  is the Benjamin--Feir discriminant function in \eqref{BFDF}. 
The eigenvalues  in \eqref{lambda 1} have non-zero real part if and only if $\Delta_{\mathrm{BF}}(\kappa;\mu,\epsilon)>0$.

The eigenvalues of the matrix $\mathsf{S}$ are a pair of purely imaginary eigenvalues of the form
\begin{align} \label{lambda zero}
\lambda_{0}^{\pm}(\mu,\epsilon) = \im  \mu \ck + \im r_{9}(\mu\epsilon^{2}, \mu^2 \epsilon,\mu^3) \mp \im \sqrt{\mu  (1 + r(\mu^2,\mu\e,\e^2))}.
\end{align}
For $\epsilon = 0$ the eigenvalues $\lambda_{1}^{\pm}(\mu,0), \lambda_{0}^{\pm}(\mu,0)$ coincide with those in \eqref{eigenvalues of Lmu0}.
\end{theorem}

\begin{remark}
At $\epsilon = 0$, the eigenvalues in \eqref{lambda 1} have the Taylor expansion
\[
\lambda_{1}^{\pm}(\mu,0) 
= \im \left( \ck - \frac{1}{2} \mathsf{e}_{12}(\kappa) \right) \mu 
\pm \im \frac{1}{8}|\mathsf{e}_{22}(\kappa)| \mu^{2} + \mathcal{O}(\mu^{3}),
\]
which coincides with the one of $\lambda_{1}^{\pm}(\mu)$ in \eqref{eigenvalues of Lmu0}, in view of the coefficients $\mathsf{e}_{12}(\kappa)$ and $\mathsf{e}_{22}(\kappa)$ defined in \eqref{e11e22e12}.
\end{remark}

\begin{figure}[tbp]
    \centering
    \includegraphics[width=\textwidth]{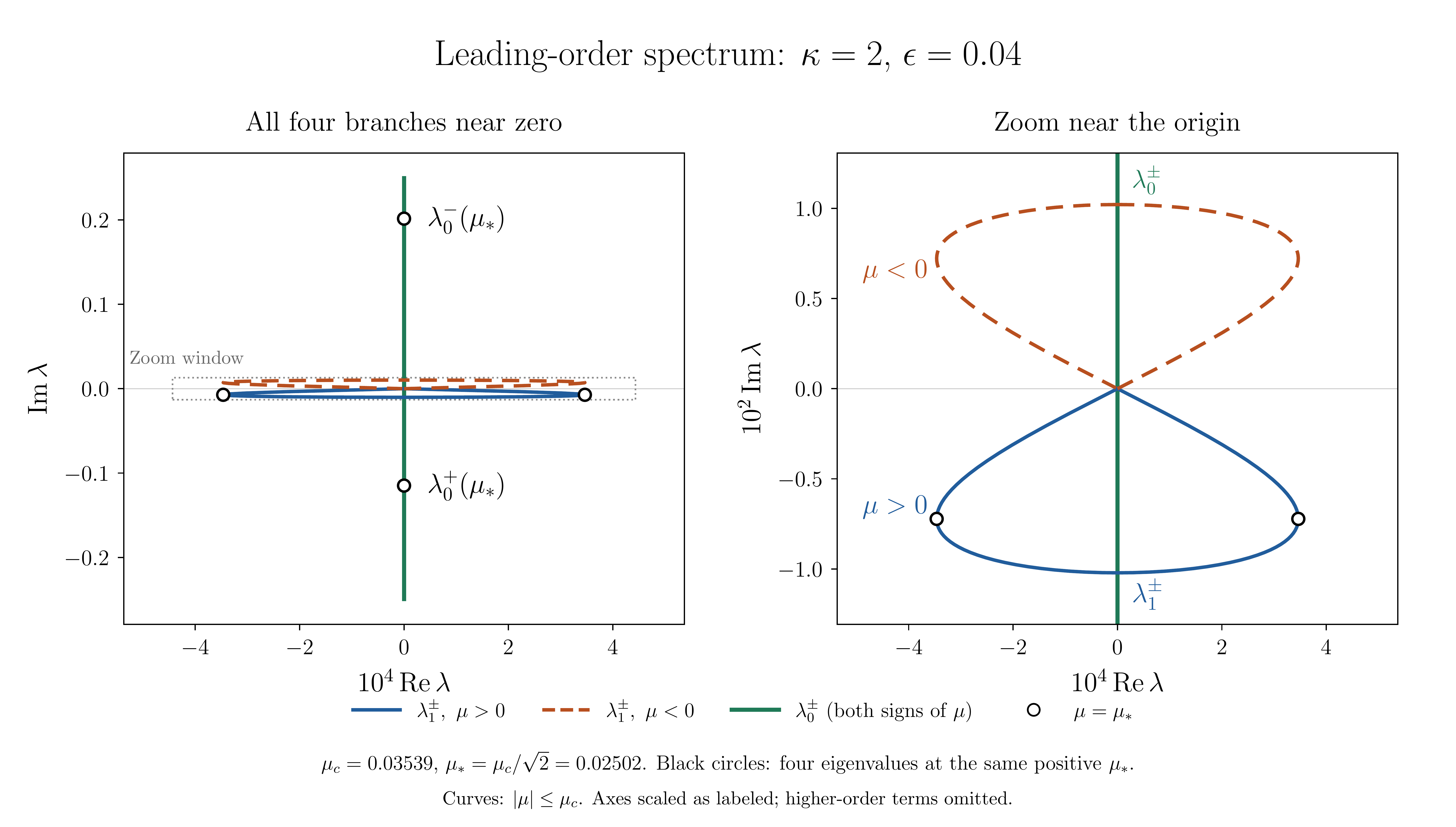}
    \caption{Leading-order illustration of the four eigenvalue branches
    in Theorem~\ref{Complete BF thm} for $\kappa=2$ and $\epsilon=0.04$,
    with $|\mu|\leq\mu_c$, where
    $\mu_c:=\epsilon\sqrt{8\mathsf e_{11}(\kappa)/\mathsf e_{22}(\kappa)}$
    is the leading-order collision threshold.
    The left panel gives an overview; the right panel magnifies
    the figure-eight near the origin.
    Blue solid and orange dashed curves show $\lambda_1^\pm$
    from \eqref{appdr} for positive and negative $\mu$, respectively.
    The green curve is the union of the $\lambda_0^\pm$ trajectories,
    approximated by $\mathrm{i}\ck\mu\mp\mathrm{i}\sqrt{\mu}$
    for $\mu\geq0$ and extended to negative $\mu$ by complex conjugation.
    Black circles mark the four eigenvalues at the same
    $\mu_*=\mu_c/\sqrt{2}$; the two $\lambda_0^\pm(\mu_*)$ lie
    outside the zoomed panel.
    Axes are scaled as labeled, and higher-order terms are omitted.}
    \label{fig: four branches}
\end{figure}

We conclude this section by describing our approach in detail.

\paragraph{Ideas and scheme of proof.}
The proof follows the strategy developed in \cite{BMV1,BMV3}, and relies on Kato's theory of similarity transformations together with a block-decoupling procedure. 
Using Kato’s theory, we analytically continue the unperturbed symplectic basis of the generalized kernel of the linearized operator at the flat surface into a symplectic basis 
\(\{f_k^{\sigma}(\mu,\epsilon)\}\) of the perturbed spectral subspace $\mathcal{V}_{\mu,\e}$, depending analytically on \((\mu,\epsilon)\). 
The expansion of this basis in \((\mu,\epsilon)\) is provided in Lemma~\ref{expansion of the basis F}. We then construct a new symplectic and reversible basis $\{g^{\sigma}_k(\mu,\e)\}$ of $\mathcal{V}_{\mu,\e}$, depending analytically on $(\mu,\e)$, with the additional property that $g^-_1(0,\e)$ has zero spatial average (see Lemma~\ref{new expansion of the basis G}). 
This choice allows us to represent the action of the operator \(\mathscr{L}_{\mu,\epsilon}\) on $\mathcal{V}_{\mu,\e}$ as a \(4\times4\) Hamiltonian and reversible matrix $\mathsf{L}_{\mu,\epsilon}$ (Lemma~\ref{matrix representation mathsf L}), which takes the form (see Lemma~\ref{B decomposition})
\begin{equation}\label{eq:L mu epsilon structure}
\mathsf{L}_{\mu,\epsilon}
= \mathsf{J}_4
\begin{pmatrix}
E & F \\[2pt]
F^{*} & G
\end{pmatrix}
=
\begin{pmatrix}
\mathsf{J}_2 E & \mathsf{J}_2 F \\[2pt]
\mathsf{J}_2 F^{*} & \mathsf{J}_2 G
\end{pmatrix},
\qquad
\mathsf{J}_4 =
\begin{pmatrix}
\mathsf{J}_2 & 0 \\[2pt]
0 & \mathsf{J}_2
\end{pmatrix},
\quad
\mathsf{J}_2 =
\begin{pmatrix}
0 & 1 \\[2pt]
-1 & 0
\end{pmatrix}.
\end{equation}

The blocks $E = E^*$, $F$, and $G = G^*$ are $2\times2$ matrices admitting the expansions \eqref{E}--\eqref{G 2nd}. 
In particular, the matrix $E$ reads
\begin{equation} \label{E first}
\begin{aligned} 
E &= 
\begin{pmatrix}
\mathsf{e}_{11}(\kappa) \epsilon^2(1+r'(\e,\mu\e^2))- \mathsf{e}_{22}(\kappa)\frac{\mu^2}{8}(1+r''(\e,\mu))  &
\im\left( \frac{1}{2}\mathsf{e}_{12}(\kappa) \mu + r_2(\mu \epsilon^2,\mu^2 \epsilon, \mu^3)\right) \\
- \im\left( \frac{1}{2}\mathsf{e}_{12}(\kappa) \mu + r_2(\mu \epsilon^2,\mu^2 \epsilon, \mu^3)\right) &
-\mathsf{e}_{22}(\kappa)\frac{\mu^2}{8}(1+r_5(\e,\mu))
\end{pmatrix},
\end{aligned}
\end{equation}
The coefficients $\mathsf{e}_{11}(\kappa)$, $\mathsf{e}_{12}(\kappa)$, and $\mathsf{e}_{22}(\kappa)$ are defined in \eqref{e11e22e12}. 
We note that $\mathsf{e}_{11}$ is singular at $\kappa=1/2$, while $\mathsf{e}_{22}$ vanishes at $\kappa=\kappa_c$. By analyzing the eigenvalues of the submatrix $\mathsf{J}_2 E$, we show that it possesses a pair of eigenvalues with nonzero real parts if and only if 
\[
\mathsf{e}_{\mathrm{WB}}(\kappa)
= \mathsf{e}_{11}(\kappa)\mathsf{e}_{22}(\kappa) > 0,
\]
provided $0<\mu<\underline{\mu}(\epsilon)\sim\epsilon$. 
To extend this instability to the full matrix $\mathsf{L}_{\mu,\epsilon}$, our goal is to eliminate the coupling term $\mathsf{J}_2 F$. 
This is achieved in Section~\ref{sec:BD} via a block-diagonalization procedure inspired by KAM theory. 

\medskip

We conclude with several remarks. The problem is of singular perturbation type, since the spectra of $\mathsf{J}_2 E$ and $\mathsf{J}_2 G$ both collapse to zero as $\mu \to 0$. 
For each fixed nonresonant $\kappa>0$ with $\kappa\neq\kappa_c$, the proof of
Theorem~\ref{Complete BF thm} follows the same structural
scheme as in the pure gravity case $\kappa=0$ treated
in \cite{BMV1}.
Indeed, the eigenvalue zero has algebraic multiplicity four
in both cases, and the reduced problem retains its
Hamiltonian and reversible structure, together with the
$O(\mu)$ and $O(\sqrt{\mu})$ spectral scales.
These features allow the same Kato reduction and
block-decoupling strategy, with coefficients and estimates
depending on $\kappa$.
Analytic perturbation theory is applied in $(\mu,\epsilon)$
at fixed $\kappa$.
Consequently, the accumulation of the resonant values
$\mathfrak{R}$ at zero does not obstruct this argument:
no estimates uniform as $\kappa\to0^+$ are required.

The singular character of the deep-water limit is
already visible in the dispersion relation.
For each fixed finite depth $\tth>0$,
\[
\sqrt{\mu\tanh(\tth\mu)}
\sim \sqrt{\tth}\,\mu
\qquad \text{as }\mu\to0^+,
\]
whereas, for each fixed $\mu>0$,
\[
\sqrt{\mu\tanh(\tth\mu)}
\longrightarrow \sqrt{\mu}
\qquad \text{as }\tth\to\infty.
\]
The small-$\mu$ expansion is therefore not uniform
in depth.
Consequently, the spectral description in deep water
cannot be obtained by simply passing to the limit
in the finite-depth expansions of \cite{HM}.

\section{Perturbative Approach to the Separated Eigenvalues}

In this section, we analyze the splitting of the eigenvalues of $\mathcal{L}_{\mu, \epsilon}$ close to 0 for small values of $\mu$ and $\epsilon$, using Kato’s similarity transformation theory \cite[I-§4-6, II-§4]{Kato1966} and \cite{BMV1, BMV3, BMV_ed}. To this end, it is convenient to rewrite the operator $\mathcal{L}_{\mu, \epsilon}$ in \eqref{mathcal L mu e} as
\begin{equation} \label{mathcal L= ichmu+mathscr L}
\mathcal{L}_{\mu, \epsilon} = \im \ck \mu + \mathscr{L}_{\mu, \epsilon}, \quad \mu \in (0,\frac{1}{2}),
\end{equation}
where, using also \eqref{D+mu}, $\mathscr{L}_{\mu, \epsilon}$ is the Hamiltonian operator
\begin{equation} \label{mathscr L mu e}
\mathscr{L}_{\mu, \epsilon} = \mathcal{J} \mathcal{B}_{\mu, \epsilon},
\end{equation}
with $\mathcal{B}_{\mu, \epsilon}$ the self-adjoint operator
\begin{equation} \label{mathcal B mu e}
\mathcal{B}_{\mu, \epsilon} := \begin{bmatrix} 1 + a_{\epsilon}(x)-\kappa\Sigma_{\mu,\e} & -(\ck + p_{\epsilon}(x)) \partial_x - \im \mu p_{\epsilon}(x) \\ \partial_x \circ(\ck + p_{\epsilon}(x)) + \im \mu p_{\epsilon}(x) & |D|+\mu(\mathrm{sgn}(D)+\Pi_0)  \end{bmatrix} \ ,
\quad \Sigma_{\mu,\e} \mbox{ in } \eqref{Sigma} \ . 
\end{equation}
In addition  $\mathscr{L}_{\mu, \epsilon}$   is also complex-reversible, namely it satisfies, by \eqref{reversible},
\begin{equation} \label{mathscr rho=-rho mathscr}
\mathscr{L}_{\mu, \epsilon} \circ \bar{\rho} = - \bar{\rho} \circ \mathscr{L}_{\mu, \epsilon},
\end{equation}
whereas $\mathcal{B}_{\mu, \epsilon}$ is reversibility-preserving, i.e. fulfills \eqref{B rho=rho B}. Note also that $\mathcal{B}_{0, \epsilon}$ is a real operator.

The scalar operator $\im \ck \mu \equiv \im \ck \mu \,\text{Id}$ just translates the spectrum of $\mathcal{L}_{\mu, \epsilon}$ along the imaginary axis of the quantity $\im \ck \mu$, that is, in view of \eqref{mathcal L= ichmu+mathscr L},
\begin{equation}
\sigma(\mathcal{L}_{\mu, \epsilon}) = \im \ck \mu + \sigma(\mathscr{L}_{\mu, \epsilon}).
\end{equation}
Thus in the sequel we focus on studying the spectrum of $\mathscr{L}_{\mu, \epsilon}$.

Note also that $\mathscr{L}_{0, \epsilon} = \mathcal{L}_{0, \epsilon}$ for any $\epsilon \geq 0$. In particular $\mathscr{L}_{0,0}$ has zero as an isolated eigenvalue with algebraic multiplicity 4, geometric multiplicity 3 and generalized kernel spanned by the vectors $\{f_1^{+}, f_1^{-}, f_0^{+}, f_0^{-}\}$ in \eqref{eigenfunc of mathcall L00}, \eqref{eigenfunc f+0}; furthermore, its spectrum is separated as in \eqref{sigma decomposition}. For any $\epsilon \neq 0$ small, $\mathscr{L}_{0, \epsilon}$ has zero as an isolated eigenvalue with geometric multiplicity 2, and two generalized eigenvectors satisfying \eqref{237}.

For the application of analytic perturbation theory,
we extend the family $\mathscr{L}_{\mu,\epsilon}$
from $\mu\geq0$ analytically to a neighborhood of
$\mu=0$ using the polynomial dependence on $\mu$
in \eqref{mathcal B mu e}.
All analyticity statements on
$B(\mu_0)\times B(\epsilon_0)$ refer to this extension.
The spectrum for negative physical Floquet exponents
is recovered from
$\sigma(\mathcal{L}_{-\mu,\epsilon})
=\overline{\sigma(\mathcal{L}_{\mu,\epsilon})}$. The operator $\mathscr{L}_{\mu, \epsilon}: Y \subset X \to X$ has domain $Y := H^2(\mathbb{T},\mathbb{C})\times H^1(\mathbb{T},\mathbb{C})$ and range $X := L^2(\mathbb{T},\mathbb{C})\times L^2(\mathbb{T},\mathbb{C})$.

\begin{lemma} \label{kato thm}
Let $\Gamma$ be a closed, counterclockwise-oriented curve around $0$ in the complex plane separating $\sigma' (\mathscr{L}_{0,0}) = \{0\}$ and the other part of the spectrum $\sigma'' (\mathscr{L}_{0,0})$ in \eqref{sigma decomposition}. There exist $\mu_0,\,\epsilon_0> 0$ such that for any $(\mu, \epsilon) \in B(\mu_0) \times B(\epsilon_0)$ the following statements hold:

\begin{enumerate}
    \item \textit{The curve $\Gamma$ belongs to the resolvent set of the operator $\mathscr{L}_{\mu,\epsilon} : Y \subset X \to X$ defined in \eqref{mathscr L mu e}.}
    \item \textit{The operators}
    \begin{equation} \label{Projection P mu e}
        P_{\mu,\epsilon} := - \frac{1}{2\pi \im} \oint_\Gamma (\mathscr{L}_{\mu,\epsilon} - \lambda)^{-1} \ d\lambda : X \to Y
    \end{equation}
    \textit{are well-defined projectors commuting with $\mathscr{L}_{\mu,\epsilon}$, i.e., $P_{\mu,\epsilon}^2 = P_{\mu,\epsilon}$ and $P_{\mu,\epsilon} \mathscr{L}_{\mu,\epsilon} = \mathscr{L}_{\mu,\epsilon} P_{\mu,\epsilon}$. The map $(\mu, \epsilon) \mapsto P_{\mu,\epsilon}$ is analytic from $B(\mu_0) \times B(\epsilon_0)$ to $\mathcal{L}(X,Y)$.}
    \item \textit{The domain $Y$ of the operator $\mathscr{L}_{\mu,\epsilon}$ decomposes as the direct sum}
    \begin{equation} \label{Y=V+ker P}
        Y = \mathcal{V}_{\mu,\epsilon} \oplus (Y\cap\ker(P_{\mu,\epsilon})), \quad \mathcal{V}_{\mu,\epsilon} := \operatorname{Rg}(P_{\mu,\epsilon}) = \ker(\operatorname{Id} - P_{\mu,\epsilon}),
    \end{equation}
    \textit{of closed invariant subspaces, namely $\mathscr{L}_{\mu,\epsilon} : \mathcal{V}_{\mu,\epsilon} \to \mathcal{V}_{\mu,\epsilon}$, $\mathscr{L}_{\mu,\epsilon} : Y\cap\ker(P_{\mu,\epsilon}) \to \ker(P_{\mu,\epsilon})$. Moreover}
    \begin{equation} \label{spectrum separated by Gamma}
    \begin{aligned}
        \sigma(\mathscr{L}_{\mu,\epsilon}) \cap \{z \in \mathbb{C} \text{ inside } \Gamma\} &= \sigma(\mathscr{L}_{\mu,\epsilon} |_{\mathcal{V}_{\mu,\epsilon}}) = \sigma'(\mathscr{L}_{\mu,\epsilon}), \\
        \sigma(\mathscr{L}_{\mu,\epsilon}) \cap \{z \in \mathbb{C} \text{ outside } \Gamma\} &= \sigma(\mathscr{L}_{\mu,\epsilon} |_{\ker(P_{\mu,\epsilon})}) = \sigma''(\mathscr{L}_{\mu,\epsilon}).
    \end{aligned}
    \end{equation}
    \item \textit{The projectors $P_{\mu,\epsilon}$ are similar to each other; the transformation operators}
    \begin{equation} \label{U transformation operators}
        U_{\mu,\epsilon} := (\operatorname{Id} - (P_{\mu,\epsilon} - P_{0,0})^2)^{-1/2} \big[P_{\mu,\epsilon} P_{0,0} + (\operatorname{Id} - P_{\mu,\epsilon})(\operatorname{Id} - P_{0,0}) \big]
    \end{equation}
    \textit{are bounded and invertible in $Y$ and in $X$, with inverse}
    \begin{equation} \label{U inverse}
        U_{\mu,\epsilon}^{-1} = \big[P_{0,0} P_{\mu,\epsilon} + (\operatorname{Id} - P_{0,0})(\operatorname{Id} - P_{\mu,\epsilon})\big](\operatorname{Id} - (P_{\mu,\epsilon} - P_{0,0})^2)^{-1/2},
    \end{equation}
    \textit{and $U_{\mu,\epsilon} P_{0,0} U_{\mu,\epsilon}^{-1} = P_{\mu,\epsilon}$ as well as $U_{\mu,\epsilon}^{-1} P_{\mu,\epsilon} U_{\mu,\epsilon} = P_{0,0}$. The map $(\mu, \epsilon) \mapsto U_{\mu,\epsilon}$ is analytic from $B(\mu_0) \times B(\epsilon_0)$ to $\mathcal{L}(Y)$.}
    \item \textit{The subspaces $\mathcal{V}_{\mu,\epsilon} = \operatorname{Rg}(P_{\mu,\epsilon})$ are isomorphic to each other: $\mathcal{V}_{\mu,\epsilon} = U_{\mu,\epsilon} \mathcal{V}_{0,0}$. In particular $\dim \mathcal{V}_{\mu,\epsilon} = \dim \mathcal{V}_{0,0} = 4$, for any $(\mu, \epsilon) \in B(\mu_0) \times B(\epsilon_0)$.}
\end{enumerate}
\end{lemma}

The proof of Lemma \ref{kato thm} is similar to the one of \cite[Lemma 3.1]{BMV1} and we skip it.  Recalling \eqref{mathscr L mu e}-\eqref{mathscr rho=-rho mathscr}, the Hamiltonian and reversible nature of the operator $\mathscr{L}_{\mu,\e}$ imply additional algebraic properties for spectral projectors $P_{\mu,\e}$ and the transformation operators $U_{\mu,\e}$ as follows.  

\begin{lemma} \label{properties of U and P}
For any $(\mu, \epsilon) \in B(\mu_0) \times B(\epsilon_0)$, the following holds true:

\begin{itemize}
    \item[(i)] The projectors $P_{\mu,\epsilon}$ defined in \eqref{Projection P mu e} are skew-Hamiltonian, namely $\mathcal{J} P_{\mu,\epsilon} = P_{\mu,\epsilon}^* \mathcal{J}$, and reversibility preserving, i.e. $\bar{\rho} P_{\mu,\epsilon} = P_{\mu,\epsilon} \bar{\rho}$.
    
    \item[(ii)] The transformation operators $U_{\mu,\epsilon}$ in \eqref{U transformation operators} are symplectic, namely $U_{\mu,\epsilon}^* \mathcal{J} U_{\mu,\epsilon} = \mathcal{J}$, and reversibility preserving.
    
    \item[(iii)] $P_{0,\epsilon}$ and $U_{0,\epsilon}$ are real operators, i.e. $\bar{P}_{0,\epsilon} = P_{0,\epsilon}$ and $\bar{U}_{0,\epsilon} = U_{0,\epsilon}$.
\end{itemize}
    
\end{lemma} 

See \cite[Lemma 3.2]{BMV1} for details. By the previous lemma, the linear involution $\bar{\rho}$ commutes with the spectral projectors $P_{\mu,\epsilon}$ and then $\bar{\rho}$ leaves invariant the subspace $\mathcal{V}_{\mu,\epsilon} = \mathrm{Rg}(P_{\mu,\epsilon})$.

\paragraph{Symplectic and reversible basis of $\mathcal{V}_{\mu,\epsilon}$.} It is convenient to represent the Hamiltonian and reversible operator $\mathscr{L}_{\mu,\epsilon}: \mathcal{V}_{\mu,\epsilon} \to \mathcal{V}_{\mu,\epsilon}$ in a basis which is symplectic and reversible, according to the following definition:

\begin{definition}[Symplectic and reversible basis]  \label{Symplectic and reversible basis}
    A basis $\mathsf{F} := \{ \mathsf{f}_1^+, \mathsf{f}_1^-, \mathsf{f}_0^+, \mathsf{f}_0^- \}$ of $\mathcal{V}_{\mu,\epsilon}$ is \textit{symplectic} if, for any $k, k' = 0,1$,
\begin{equation} \label{basis is symplectic}
    \begin{aligned}
    (\mathcal{J} \mathsf{f}_k^\mp, \mathsf{f}_k^\pm) = \pm 1, ~~(\mathcal{J} \mathsf{f}_k^\sigma, \mathsf{f}_k^\sigma) = 0, \quad \forall \,\sigma = \pm; \\
    \text{if } k \neq k', \text{ then } (\mathcal{J} \mathsf{f}_k^\sigma, \mathsf{f}_{k'}^{\sigma'}) = 0, \quad \forall\, \sigma, \sigma' = \pm.
\end{aligned}
\end{equation}   

This is \textit{reversible} if
\begin{equation} \label{basis is reversible}
    \begin{aligned}
    \bar{\rho} \mathsf{f}_1^+ &= \mathsf{f}_1^+,  ~\bar{\rho} \mathsf{f}_1^- = -\mathsf{f}_1^-,~\bar{\rho} \mathsf{f}_0^+ = \mathsf{f}_0^+,  ~\bar{\rho} \mathsf{f}_0^- = -\mathsf{f}_0^-, \\
    &\text{i.e. } \bar{\rho} \mathsf{f}_k^\sigma = \sigma \mathsf{f}_k^\sigma, \quad \forall\, \sigma = \pm, k = 0,1.
\end{aligned}
\end{equation}
\end{definition} 

We use the following notation along the paper: we denote by $\mathit{even}(x)$ a real $2\pi$-periodic function which is even in $x$, and by $\mathit{odd}(x)$ a real $2\pi$-periodic function which is odd in $x$.

\begin{remark}[Parity structure of the reversible basis \eqref{basis is reversible}]
The elements of a reversible basis \( \mathsf{F}=\{\mathsf{f}^+_1,\mathsf{f}^-_1,\mathsf{f}^+_0,\mathsf{f}^-_0\} \) enjoy specific parity properties. Specifically,
\begin{align} \label{Parity structure}
\mathsf{f}_k^+(x) = 
\begin{bmatrix}
\mathit{even}(x) + \im\,\mathit{odd}(x) \\
\mathit{odd}(x) + \im\,\mathit{even}(x)
\end{bmatrix}, \quad
\mathsf{f}_k^-(x) = 
\begin{bmatrix}
\mathit{odd}(x) + \im\,\mathit{even}(x) \\
\mathit{even}(x) + \im\,\mathit{odd}(x)
\end{bmatrix}.
\end{align}
This structure follows from the reversibility of the problem, specifically from the involution \( \overline{\rho} \) defined in equation \eqref{reversible}, which implies that the real and imaginary parts of each component satisfy definite parity conditions.
\end{remark}

\begin{remark}[Symplectic expansion using the basis in \eqref{basis is symplectic}]
We can express any vector \( \mathsf{f} \in \mathcal{V}_{\mu,\epsilon} \) as a linear combination of the symplectic basis:
\begin{align} \label{remark f expansion}
    \mathsf{f} = \alpha_1^+ \mathsf{f}_1^+ + \alpha_1^- \mathsf{f}_1^- + \alpha_0^+ \mathsf{f}_0^+ + \alpha_0^- \mathsf{f}_0^-,
\end{align}
for suitable coefficients \( \alpha_k^\sigma \in \mathbb{C} \). These coefficients are computed by applying the symplectic form \( \mathcal{J} \), taking \( L^2 \)-scalar products with the basis elements, and using the symplecticity \eqref{basis is symplectic}. Therefore, we may rewrite \eqref{remark f expansion} as
\begin{align} \label{expasion of f by using symplectic basis}
    \mathsf{f}=-(\mathcal{J}\mathsf{f},\mathsf{f}^-_1)\mathsf{f}^+_1+(\mathcal{J}\mathsf{f},\mathsf{f}^+_1)\mathsf{f}^-_1-(\mathcal{J}\mathsf{f},\mathsf{f}^-_0)\mathsf{f}^+_0+(\mathcal{J}\mathsf{f},\mathsf{f}^+_0)\mathsf{f}^-_0.
\end{align}
\end{remark}

We now compute the matrix representation of $\mathscr{L}_{\mu,\e}$ in a symplectic and reversible basis of $\mathcal{V}_{\mu,\e}$.

\begin{lemma} \label{matrix representation mathsf L}
The $4 \times 4$ matrix that represents the Hamiltonian and reversible operator $\mathscr{L}_{\mu,\epsilon} = \mathcal{J} \mathcal{B}_{\mu,\epsilon} : \mathcal{V}_{\mu,\epsilon} \to \mathcal{V}_{\mu,\epsilon}$ with respect to a symplectic and reversible basis $\mathsf{F} = \{ \mathsf{f}_1^{+}, \mathsf{f}_1^{-}, \mathsf{f}_0^{+}, \mathsf{f}_0^{-} \}$ of $\mathcal{V}_{\mu,\epsilon}$ is
\begin{align} \label{J4Bmuepsilon}
  \mathsf{L}_{\mu,\epsilon}:=  \mathsf{J}_4 \mathsf{B}_{\mu,\epsilon}, \quad \mathsf{J}_4 := 
\left(
\begin{array}{c|c}
\mathsf{J}_2 & 0 \\
\hline
0 & \mathsf{J}_2
\end{array}
\right), \quad
    \mathsf{J}_2 := 
    \begin{pmatrix}
        0 & 1 \\
        -1 & 0
    \end{pmatrix}, \quad \text{where} \quad \mathsf{B}_{\mu,\epsilon} = \mathsf{B}_{\mu,\epsilon}^{*}.
\end{align}
The self-adjoint matrix
\begin{align} \label{Bmuepsilon matrix representation}
    \mathsf{B}_{\mu,\epsilon} = \left( 
    \begin{array}{cccc}
        (\mathcal{B}_{\mu,\epsilon} \mathsf{f}_1^{+}, \mathsf{f}_1^{+}) & (\mathcal{B}_{\mu,\epsilon} \mathsf{f}_1^{-}, \mathsf{f}_1^{+}) & (\mathcal{B}_{\mu,\epsilon} \mathsf{f}_0^{+}, \mathsf{f}_1^{+}) & (\mathcal{B}_{\mu,\epsilon} \mathsf{f}_0^{-}, \mathsf{f}_1^{+}) \\
        (\mathcal{B}_{\mu,\epsilon} \mathsf{f}_1^{+}, \mathsf{f}_1^{-}) & (\mathcal{B}_{\mu,\epsilon} \mathsf{f}_1^{-}, \mathsf{f}_1^{-}) & (\mathcal{B}_{\mu,\epsilon} \mathsf{f}_0^{+}, \mathsf{f}_1^{-}) & (\mathcal{B}_{\mu,\epsilon} \mathsf{f}_0^{-}, \mathsf{f}_1^{-}) \\
        (\mathcal{B}_{\mu,\epsilon} \mathsf{f}_1^{+}, \mathsf{f}_0^{+}) & (\mathcal{B}_{\mu,\epsilon} \mathsf{f}_1^{-}, \mathsf{f}_0^{+}) & (\mathcal{B}_{\mu,\epsilon} \mathsf{f}_0^{+}, \mathsf{f}_0^{+}) & (\mathcal{B}_{\mu,\epsilon} \mathsf{f}_0^{-}, \mathsf{f}_0^{+}) \\
        (\mathcal{B}_{\mu,\epsilon} \mathsf{f}_1^{+}, \mathsf{f}_0^{-}) & (\mathcal{B}_{\mu,\epsilon} \mathsf{f}_1^{-}, \mathsf{f}_0^{-}) & (\mathcal{B}_{\mu,\epsilon} \mathsf{f}_0^{+}, \mathsf{f}_0^{-}) & (\mathcal{B}_{\mu,\epsilon} \mathsf{f}_0^{-}, \mathsf{f}_0^{-})
    \end{array}
    \right).
\end{align}
The entries of the matrix $\mathsf{B}_{\mu,\epsilon}$ are alternatively real or purely imaginary: for any $\sigma = \pm, k = 0, 1$,
\begin{align} \label{B are alternatively real or imaginary}
    (\mathcal{B}_{\mu,\epsilon} \mathsf{f}_k^{\sigma}, \mathsf{f}_{k'}^{\sigma}) \text{ is real}, \quad
    (\mathcal{B}_{\mu,\epsilon} \mathsf{f}_k^{\sigma}, \mathsf{f}_{k'}^{-\sigma}) \text{ is purely imaginary}.
\end{align}  
\end{lemma} 
\begin{proof}
    Recalling \eqref{mathscr L mu e} and \eqref{expasion of f by using symplectic basis}, for $\sigma=\pm$, $k=0,1$, we obtain
    \begin{equation*} 
        \begin{aligned}
            \mathscr{L}_{\mu,\epsilon}\mathsf{f}^{\sigma}_k&=-(\mathcal{J}\mathscr{L}_{\mu,\epsilon}\mathsf{f}^{\sigma}_k,\mathsf{f}^-_1)\mathsf{f}^+_1+(\mathcal{J}\mathscr{L}_{\mu,\epsilon}\mathsf{f}^{\sigma}_k,\mathsf{f}^+_1)\mathsf{f}^-_1-(\mathcal{J}\mathscr{L}_{\mu,\epsilon}\mathsf{f}^{\sigma}_k,\mathsf{f}^-_0)\mathsf{f}^+_0+(\mathcal{J}\mathscr{L}_{\mu,\epsilon}\mathsf{f}^{\sigma}_k,\mathsf{f}^+_0)\mathsf{f}^-_0\\
            &=(\mathcal{B}_{\mu,\epsilon}\mathsf{f}^{\sigma}_k,\mathsf{f}^-_1)\mathsf{f}^+_1-(\mathcal{B}_{\mu,\epsilon}\mathsf{f}^{\sigma}_k,\mathsf{f}^+_1)\mathsf{f}^-_1+(\mathcal{B}_{\mu,\epsilon}\mathsf{f}^{\sigma}_k,\mathsf{f}^-_0)\mathsf{f}^+_0-(\mathcal{B}_{\mu,\epsilon}\mathsf{f}^{\sigma}_k,\mathsf{f}^+_0)\mathsf{f}^-_0.
        \end{aligned}
    \end{equation*}
This verifies that the matrix representation of $\mathscr{L}_{\mu,\epsilon}$ with respect to $\mathsf{F}$ is $\mathsf{J}_4\mathsf{B}_{\mu,\epsilon}$. Also, the matrix $\mathsf{B}_{\mu,\e}$ is self-adjoint because $\mathcal{B}_{\mu,\e}$ is self-adjoint. Next, recall from \eqref{complex product} and \eqref{reversible} that the inner product satisfies
\begin{align} \label{(f,g)=(rho f, rho g)}
    (\mathsf{f},\mathsf{g})=\overline{(\overline{\rho}\mathsf{f},\overline{\rho}\mathsf{g})}.
\end{align}
Then, since $\mathcal{B}_{\mu,\e}$ is both self-adjoint and reversibility-preserving \eqref{B rho=rho B} and \eqref{basis is reversible}, we compute:
\begin{align*}
    (\mathcal{B}_{\mu,\epsilon} \mathsf{f}_k^{\sigma}, \mathsf{f}_{k'}^{\sigma'})=\overline{(\overline{\rho}\mathcal{B}_{\mu,\e} \mathsf{f}^{\sigma}_k,\overline{\rho} \mathsf{f}^{\sigma'}_{k'})}=\overline{(\mathcal{B}_{\mu,\e}\overline{\rho} \mathsf{f}^{\sigma}_k,\overline{\rho} \mathsf{f}^{\sigma'}_{k'})}=\sigma\sigma'\overline{(\mathcal{B}_{\mu,\e}\mathsf{f}^{\sigma}_k,\mathsf{f}^{\sigma'}_{k'})},
\end{align*}
which proves \eqref{B are alternatively real or imaginary}.
\end{proof}

We conclude this section recalling some notation. A $2n \times 2n$, $n = 1, 2$, matrix of the form $\mathsf{L} = \mathsf{J}_{2n} \mathsf{B}$ is \textit{Hamiltonian} if $\mathsf{B}$ is a self-adjoint matrix, i.e. $\mathsf{B} = \mathsf{B}^*$. It is \textit{reversible} if $\mathsf{B}$ is reversibility-preserving, i.e.
\begin{align} \label{reversibility preserving}
\rho_{2n} \circ \mathsf{B} = \mathsf{B} \circ \rho_{2n},
\end{align}
where
\[
\rho_4 := \begin{pmatrix} \rho_2 & 0 \\ 0 & \rho_2 \end{pmatrix}, \quad
\rho_2 := \begin{pmatrix} \mathfrak{c} & 0 \\ 0 & -\mathfrak{c} \end{pmatrix}, 
\]
and $\mathfrak{c} : z \mapsto \bar{z}$ is the conjugation of the complex plane.
Equivalently, $\rho_{2n} \circ \mathsf{L} = - \mathsf{L} \circ \rho_{2n}.$

The transformations preserving the Hamiltonian structure are called \textit{symplectic}, and satisfy
\begin{align} \label{YstarJ4Y=J4}
Y^* \mathsf{J}_4 Y = \mathsf{J}_4.
\end{align}
If $Y$ is symplectic then $Y^*$ and $Y^{-1}$ are symplectic as well. A Hamiltonian matrix $\mathsf{L} = \mathsf{J}_4 \mathsf{B}$, with $\mathsf{B} = \mathsf{B}^*$, is conjugated through a symplectic matrix $Y$ in a new Hamiltonian matrix
\begin{align} \label{L1=J4B}
    \mathsf{L}_1=Y^{-1}\mathsf{L}Y=Y^{-1}\mathsf{J}_4(Y^{-1})^*Y^*\mathsf{B}Y=\mathsf{J}_4 \mathsf{B}_1\,\,\textit{where}\,\,\mathsf{B}_1:=Y^*\mathsf{B}Y=\mathsf{B}^*_1.
\end{align}
A $4\times 4$ matrix $\mathsf{B} = (\mathsf{B}_{ij})_{i,j=1,\ldots,4}$ is reversibility-preserving if and only if its entries are alternatively real and purely imaginary, namely $\mathsf{B}_{ij}$ is real when $i + j$ is even and purely imaginary otherwise, as in \eqref{B are alternatively real or imaginary}. A $4\times 4$ complex matrix $\mathsf{L} = (\mathsf{L}_{ij})_{i,j=1,\ldots,4}$ is reversible if and only if $\mathsf{L}_{ij}$ is purely imaginary when $i + j$ is even and real otherwise.

We finally mention that the flow of a Hamiltonian reversibility-preserving matrix is symplectic and reversibility-preserving (see \cite[Lemma 3.8]{BMV1}).

\section{Matrix Representation of $\mathscr{L}_{\mu,\e}$ on $\mathcal{V}_{\mu,\e}$} Using the transformation operator $U_{\mu,\e}$ in \eqref{U transformation operators}, we construct the basis of $\mathcal{V}_{\mu,\e}$
\begin{equation} \label{F basis set and f}
    \begin{aligned}
        \mathcal{F}&:=\{f^+_1(\mu,\e),f^-_1(\mu,\e),f^+_0(\mu,\e),f^-_0(\mu,\e)\}, \qquad 
        f^\sigma_k(\mu,\e):=U_{\mu,\e} f^{\sigma}_k,~~\sigma=\pm,~k=0,1,
    \end{aligned}
\end{equation}
where 
\begin{align} \label{eigenfunc of mathcall L00 2}
    f^+_1=\vet{\frac{1}{\sqrt{\ck}}\cos(x)}{\sqrt{\ck}\sin(x)}, ~~f^-_1=\vet{-\frac{1}{\sqrt{\ck}}\sin(x)}{\sqrt{\ck}\cos(x)}, ~~f^+_0=\vet{1}{0},~~f^-_0=\vet{0}{1},
\end{align}
form a basis of $\mathcal{V}_{0,0}=\mathrm{Rg}(P_{0,0})$, cf. \eqref{eigenfunc of mathcall L00}-\eqref{eigenfunc f+0}. Note that the real valued vectors $\{f^+_1,f^-_1,f^+_0,f^-_0\}$ form a symplectic and reversible basis for $\mathcal{V}_{0,0}$, according to Definition \ref{Symplectic and reversible basis}. Then, by Lemma \ref{kato thm} and Lemma \ref{properties of U and P} we deduce that:
\begin{lemma} \label{F is symplectic and reversible}
    The basis $\mathcal{F}$ of $\mathcal{V}_{\mu,\e}$ defined in \eqref{F basis set and f}, is symplectic and reversible, i.e. satisfies \eqref{basis is symplectic} and \eqref{basis is reversible}. Each map $(\mu,\e) \mapsto f^\sigma_k(\mu,\e)$ is analytic as a map $B(\mu_0)\times B(\e_0)\rightarrow H^2(\mathbb{T})\times H^1(\mathbb{T})$.
\end{lemma}
\begin{proof}
By Lemma \ref{properties of U and P}-(ii), the operators $U_{\mu,\epsilon}$ are symplectic and reversibility preserving, i.e. 
$U_{\mu,\epsilon}^* \mathcal{J} U_{\mu,\epsilon} = \mathcal{J}$ and $U_{\mu,\epsilon}\circ\overline{\rho}=\overline{\rho}\circ U_{\mu,\epsilon}$. Recall \eqref{F basis set and f}, i.e. $f_k^\sigma(\mu,\epsilon):=U_{\mu,\epsilon} f_k^\sigma, \sigma=\pm, k=0,1.$
Then for any $k,k',\sigma,\sigma'$,
\begin{align*}
(\mathcal{J} f_k^\sigma(\mu,\epsilon), f_{k'}^{\sigma'}(\mu,\epsilon))
= (\mathcal{J} f_k^\sigma, f_{k'}^{\sigma'}),
\end{align*}
so the symplectic relations \eqref{basis is symplectic} are preserved. Moreover,
\begin{align*}
\overline{\rho} f_k^\sigma(\mu,\epsilon)
= U_{\mu,\epsilon}\circ \overline{\rho} f_k^\sigma
= \sigma f_k^\sigma(\mu,\epsilon),
\end{align*}
which shows that the reversibility conditions \eqref{basis is reversible} hold as well. 
Finally, the analyticity of $f_k^\sigma(\mu,\epsilon)$ follows from the analyticity of 
$U_{\mu,\epsilon}$ (Lemma \ref{kato thm}). 
\end{proof}

We then expand the vectors $f^\sigma_k(\mu,\e)$ in $(\mu,\e)$. We denote by $\mathit{even}_0(x)$ a real, even, $2\pi$-periodic function with zero space average. In the sequel $\cO(\mu^m\e^n)\vet{\mathit{even}(x)}{\mathit{odd}(x)}$ denotes an analytic map in $(\mu,\e)$ with values in $Y=H^2(\mathbb{T},\mathbb{C})\times H^1(\mathbb{T},\mathbb{C})$, whose first component is $\mathit{even}(x)$ and the second one $\mathit{odd}(x)$; we have a similar meaning for $\cO(\mu^m\e^n)\vet{\mathit{odd}(x)}{\mathit{even}(x)}$, etc $\ldots$.

\begin{lemma} [Expansion of the basis $\mathcal{F}$] \label{expansion of the basis F} For small values of $(\mu,\e)$ the basis $\mathcal{F}$ in \eqref{F basis set and f} has the expansion

\begin{equation} \label{43 f+1}
\begin{aligned} 
f_1^+(\mu,\epsilon) &= 
\vet{\frac{1}{\sqrt{\ck}}\cos(x)}{\sqrt{\ck}\sin(x)}
+ \im\mu \frac{1-\kappa}{4\ck^2} 
\vet{\frac{1}{\sqrt{\ck}}\sin(x)}{\sqrt{\ck}\cos(x)}
+ \epsilon
\begin{bmatrix}
\frac{2-\kappa}{(1-2\kappa)\sqrt{\ck}} \cos(2x) \\
\frac{\ck^2\sqrt{\ck}}{1-2\kappa} \sin(2x)
\end{bmatrix} \\
& \quad + \mathcal{O}(\mu^2)
\begin{bmatrix}
\mathit{even}_0(x) + \im \mathit{odd}(x) \\
\mathit{odd}(x) + \im \mathit{even}_0(x)
\end{bmatrix}
+ \mathcal{O}(\epsilon^2)
\begin{bmatrix}
\mathit{even}_0(x) \\
\mathit{odd}(x)
\end{bmatrix} + \im \mu \epsilon
\begin{bmatrix}
\mathit{odd}(x) \\
\mathit{even}(x)
\end{bmatrix}
+ \mathcal{O}(\mu^2 \epsilon, \mu \epsilon^2)\,,
\end{aligned}
\end{equation}

\begin{equation} \label{44 f-1}
\begin{aligned}
f_1^-(\mu,\epsilon) &= 
\vet{-\frac{1}{\sqrt{\ck}}\sin(x)}{\sqrt{\ck}\cos(x)}
+ \im\mu \frac{1-\kappa}{4\ck^2} 
\vet{\frac{1}{\sqrt{\ck}}\cos(x)}{-\sqrt{\ck}\sin(x)}
+ \epsilon
\begin{bmatrix}
-\frac{2-\kappa}{(1-2\kappa)\sqrt{\ck}} \sin(2x) \\
\frac{\ck^2\sqrt{\ck}}{1-2\kappa} \cos(2x)
\end{bmatrix} \\
& \quad + \mathcal{O}(\mu^2)
\begin{bmatrix}
\mathit{odd}(x) + \im \mathit{even}_0(x) \\
\mathit{even}_0(x) + \im \mathit{odd}(x)
\end{bmatrix}
+ \mathcal{O}(\epsilon^2)
\begin{bmatrix}
\mathit{odd}(x) \\
\mathit{even}(x)
\end{bmatrix}   + \im \mu \epsilon
\begin{bmatrix}
\mathit{even}(x) \\
\mathit{odd}(x)
\end{bmatrix}
+ \mathcal{O}(\mu^2 \epsilon, \mu \epsilon^2)\,,
\end{aligned}
\end{equation}

\begin{equation} \label{45 f+0}
\begin{aligned}
f_0^+(\mu,\epsilon) &=
\begin{bmatrix}
1 \\
0
\end{bmatrix}
+  \e\vet{\cos(x)}{-\ck\sin(x)}
+ \textcolor{black}{\frac{\mu \epsilon}{4}
\begin{bmatrix}
\ck^{-2}\cos(x) \\
-\ck^{-1}\sin(x)
\end{bmatrix}}  + \im \mu \epsilon
\begin{bmatrix}
\mathit{odd}(x) \\
\mathit{even}_0(x)
\end{bmatrix} \\
&\quad 
+\mathcal{O}(\epsilon^2)
\begin{bmatrix}
\mathit{even}_0(x) \\
\mathit{odd}(x)
\end{bmatrix}+ \mathcal{O}(\mu^2 \epsilon, \mu \epsilon^2)\,,
\end{aligned}
\end{equation}

\begin{equation} \label{46 f-0}
\begin{aligned}
f_0^-(\mu,\epsilon) &=
\begin{bmatrix}
0 \\
1
\end{bmatrix}
+ \frac{1}{2}\mu \epsilon\left(
\vet{\frac{1}{\ck}\sin(x)}{\cos(x)}
+ \im
\begin{bmatrix}
\mathit{even}_0(x) \\
\mathit{odd}(x)
\end{bmatrix}\right)
+ \mathcal{O}(\mu^2 \epsilon, \mu \epsilon^2)\,.
\end{aligned}
\end{equation}
For $\mu = 0$, the basis $\{ f_j^\pm(0,\epsilon),\, \epsilon = 0 \}$ is real and
\begin{equation} \label{48 f-0=01}
\begin{aligned}
f_1^+(0,\epsilon) = 
\begin{bmatrix}
\mathit{even}_0(x) \\
\mathit{odd}(x)
\end{bmatrix}, \quad
f_1^-(0,\epsilon) = 
\begin{bmatrix}
\mathit{odd}(x) \\
\mathit{even}(x)
\end{bmatrix},\quad  
f_0^+(0,\epsilon) =\vet{1}{0}+
\begin{bmatrix}
\mathit{even}_0(x) \\
\mathit{odd}(x)
\end{bmatrix}, \quad
f_0^-(0,\epsilon) = 
\begin{bmatrix}
0 \\
1
\end{bmatrix}.
\end{aligned}
\end{equation}
\end{lemma}
\begin{proof}
    The long computations are given in Appendix \ref{secA1}.
\end{proof}
\paragraph{Second basis of $\mathcal{V}_{\mu,\e}$.}
We now construct from the basis $\mathcal{F}$ in Lemma \ref{expansion of the basis F} another symplectic and reversible basis of $\mathcal{V}_{\mu,\e}$, with an additional property.Note that
the second component of the vector $f_1^-(0,\epsilon)$ is an even function whose space average is not necessarily zero, cf. \eqref{48 f-0=01}. Thus we introduce the new symplectic and reversible basis of $\mathcal{V}_{\mu,\e}$
\begin{equation} \label{G basis set and g}
    \begin{aligned}
        \mathcal{G}&:=\left\{g^+_1(\mu,\e),g^-_1(\mu,\e),g^+_0(\mu,\e),g^-_0(\mu,\e)\right\},\
    \end{aligned}
\end{equation}
defined by 
\begin{align}
    & g_1^+(\mu,\epsilon):=f_1^+(\mu,\epsilon)\,,\qquad g_1^-(\mu,\epsilon):=f_1^-(\mu,\epsilon)-n(\mu,\epsilon)f_0^-(\mu,\epsilon)\,,\\
           & g_0^+(\mu,\epsilon):=f_0^+(\mu,\epsilon)+n(\mu,\epsilon)f_1^+(\mu,\epsilon)\,,\qquad g_0^-(\mu,\epsilon):=f_0^-(\mu,\epsilon)\,,
\end{align}
with 
\begin{equation}
    n(\mu,\epsilon):=\frac{(f_1^-(\mu,\epsilon),f_0^-(\mu,\epsilon))}{\|f_0^-(\mu,\epsilon)\|^2}\,.
\end{equation}
Note that $n(\mu,\epsilon)$ is real, because, in view of \eqref{(f,g)=(rho f, rho g)} and Lemma \ref{F is symplectic and reversible},
 \begin{align} \label{n mu e}
        n(\mu,\epsilon)=\frac{\overline{(\overline{\rho}f_1^-(\mu,\epsilon),\overline{\rho}f_0^-(\mu,\epsilon))}}{\|f_0^-(\mu,\epsilon)\|^2}=\frac{\overline{(f_1^-(\mu,\epsilon),f_0^-(\mu,\epsilon))}}{\|f_0^-(\mu,\epsilon)\|^2}=\overline{ n(\mu,\epsilon)}\,.
    \end{align}
This new basis has the property that $g_1^-(0,\epsilon)$has zero average, see \eqref{int g1-=0}. We shall exploit this feature crucially in Lemma \ref{Expansion of B flat}.

\begin{lemma}The basis $\mathcal{G}$ in \eqref{G basis set and g} is symplectic and reversible, i.e. satisfies \eqref{basis is symplectic} and \eqref{basis is reversible}. Each map $(\mu,\e) \mapsto g^\sigma_k(\mu,\e)$ is analytic as a map $B(\mu_0)\times B(\e_0)\rightarrow H^2(\mathbb{T})\times H^1(\mathbb{T})$.
First, we provide the following expansion in $(\mu,\e)$ of the basis $\mathcal{G}$.
\end{lemma}
\begin{proof}
    The proof is analogous to the one of \cite[Lemma 4.3]{BMV1}.
\end{proof}
\begin{lemma} [New expansion of the basis $\mathcal{G}$] \label{new expansion of the basis G} For small values of $(\mu,\e)$ the basis $\mathcal{G}$ in \eqref{G basis set and g} has the expansion

\begin{equation} \label{43 g+1}
\begin{aligned} 
g_1^+(\mu,\epsilon) &= 
\vet{\frac{1}{\sqrt{\ck}}\cos(x)}{\sqrt{\ck}\sin(x)}
+ \im\mu \frac{1-\kappa}{4\ck^2} 
\vet{\frac{1}{\sqrt{\ck}}\sin(x)}{\sqrt{\ck}\cos(x)}
+ \epsilon
\begin{bmatrix}
\frac{2-\kappa}{(1-2\kappa)\sqrt{\ck}} \cos(2x) \\
\frac{\ck^2\sqrt{\ck}}{1-2\kappa} \sin(2x)
\end{bmatrix} \\
& \quad + \mathcal{O}(\epsilon^2)
\begin{bmatrix}
\mathit{even}_0(x) \\
\mathit{odd}(x)
\end{bmatrix}+ \mathcal{O}(\mu^2)
\begin{bmatrix}
\mathit{even}_0(x) + \im \mathit{odd}(x) \\
\mathit{odd}(x) + \im \mathit{even}_0(x)
\end{bmatrix}+ \im \mu \epsilon
\begin{bmatrix}
\mathit{odd}(x) \\
\mathit{even}(x)
\end{bmatrix}
+ \mathcal{O}(\mu^2 \epsilon, \mu \epsilon^2)\,,
\end{aligned}
\end{equation}

\begin{equation} \label{44 g-1}
\begin{aligned}
g_1^-(\mu,\epsilon) &= 
\vet{-\frac{1}{\sqrt{\ck}}\sin(x)}{\sqrt{\ck}\cos(x)}
+ \im\mu\frac{1-\kappa}{4\ck^2} 
\vet{\frac{1}{\sqrt{\ck}}\cos(x)}{-\sqrt{\ck}\sin(x)}
+ \epsilon
\begin{bmatrix}
-\frac{2-\kappa}{(1-2\kappa)\sqrt{\ck}} \sin(2x) \\
\frac{\ck^2\sqrt{\ck}}{1-2\kappa} \cos(2x)
\end{bmatrix}+ \mathcal{O}(\epsilon^2)
\begin{bmatrix}
\mathit{odd}(x) \\
\mathit{even}_0(x)
\end{bmatrix} \\
& \quad + \mathcal{O}(\mu^2)
\begin{bmatrix}
\mathit{odd}(x) + \im \mathit{even}_0(x) \\
\mathit{even}_0(x) + \im \mathit{odd}(x)
\end{bmatrix}
   + \red{ \mu \epsilon\left(\frac{-\kappa}{4\ck\sqrt{\ck}}\vet{0}{1}+
\im\begin{bmatrix}
\mathit{even}(x) \\
\mathit{odd}(x)
\end{bmatrix}\right)}
+ \mathcal{O}(\mu^2 \epsilon, \mu \epsilon^2)\,,
\end{aligned}
\end{equation}

\textcolor{black}{\begin{equation} \label{45 g+0}
\begin{aligned}
g_0^+(\mu,\epsilon)
&=
\begin{bmatrix}
1\\
0
\end{bmatrix}
+\epsilon
\begin{bmatrix}
\cos(x)\\
-\ck\sin(x)
\end{bmatrix}
+\mathcal{O}(\epsilon^2)
\begin{bmatrix}
\mathit{even}_0(x)\\
\mathit{odd}(x)
\end{bmatrix}
\\
&\quad
+\mu\epsilon
\left(
\begin{bmatrix}
\frac14\cos(x)\\[1mm]
\frac{\kappa-1}{4\ck}\sin(x)
\end{bmatrix}
+\im
\begin{bmatrix}
\mathit{odd}(x)\\
\mathit{even}_0(x)
\end{bmatrix}
\right)
+\mathcal{O}(\mu^2\epsilon,\mu\epsilon^2).
\end{aligned}
\end{equation}}

\begin{equation} \label{46 g-0}
\begin{aligned}
g_0^-(\mu,\epsilon) &=
\begin{bmatrix}
0 \\
1
\end{bmatrix}
+ \frac{1}{2}\mu \epsilon\left(
\vet{\frac{1}{\ck}\sin(x)}{\cos(x)}
+ \im
\begin{bmatrix}
\mathit{even}_0(x) \\
\mathit{odd}(x)
\end{bmatrix}\right)
+ \mathcal{O}(\mu^2 \epsilon, \mu \epsilon^2)\,.
\end{aligned}
\end{equation}
In particular, at $\mu = 0$, the basis $\{ g^+_1(0,\epsilon), g^-_1(0,\epsilon), g^+_0(0,\epsilon), g^-_0(0,\epsilon)\}$ is real and
\begin{equation} \label{48 g-0=01}
\begin{aligned}
g_1^+(0,\epsilon) = 
\begin{bmatrix}
\mathit{even}_0(x) \\
\mathit{odd}(x)
\end{bmatrix}, \quad
g_1^-(0,\epsilon) = 
\begin{bmatrix}
\mathit{odd}(x) \\
\mathit{even}_0(x)
\end{bmatrix},  \quad
g_0^+(0,\epsilon) =\vet{1}{0}+
\begin{bmatrix}
\mathit{even}_0(x) \\
\mathit{odd}(x)
\end{bmatrix}, \quad
g_0^-(0,\epsilon) = 
\begin{bmatrix}
0 \\
1
\end{bmatrix}.
\end{aligned}
\end{equation}
and, for any $\epsilon$, 
\begin{align}\label{int g1-=0}
       \int_\mathbb{T} g^-_1(0,\epsilon) dx=0\,.
    \end{align}
\end{lemma}

\begin{proof}
   First note that, by \eqref{48 f-0=01}  $f^-_0(0,\epsilon)=\vet{0}{1}$, and thus $g^-_1(0,\epsilon)$ reduces to
\begin{align}
       g^-_1(0,\epsilon)=f^-_1(0,\epsilon)-\left(f^-_1(0,\epsilon),\vet{0}{1}\right)\vet{0}{1}
    \end{align}
which satisfies \eqref{int g1-=0}, recalling also that the first component of $f^-_1(0,\epsilon)$ is odd.
In order to prove \eqref{43 g+1}-\eqref{46 g-0} we note that $n(\mu,\epsilon)$ in \eqref{n mu e} is real, and satisfies, by \eqref{44 f-1}, \eqref{46 f-0},
\begin{equation}
\begin{aligned}
       n(\mu,\epsilon)&=\frac{1}{1+r(\mu^2\epsilon,\mu\epsilon^2)}\left[\left(\vet{-\frac{1}{\sqrt{\ck}}\sin(x)}{\sqrt{\ck}\cos(x)},\frac{1}{2}\mu\epsilon \vet{\frac{1}{\ck}\sin(x)}{\cos(x)}\right)
+ r(\e^2,\mu^2 \epsilon, \mu \epsilon^2)\right]\\
&=\mu \epsilon \frac{\kappa}{4\ck\sqrt{\ck}}+r(\e^2,\mu^2 \epsilon, \mu \epsilon^2).
 \end{aligned}
 \end{equation}
Hence, in view of \eqref{43 f+1}-\eqref{46 f-0}, the vectors $g^{\sigma}_k(\mu,\epsilon)$ satisfy the expansion \eqref{43 g+1}-\eqref{46 g-0}. Finally at $\mu=0$ the vectors $g^{\sigma}_k(0,\epsilon)$, $\sigma=\pm$, $k=0,1$ are real being real linear combinations of real vectors.
\end{proof}

We now state the main result of this section.

\begin{lemma} \label{B decomposition}
    The matrix that represents the Hamiltonian and reversible operator $\mathscr{L}_{\mu,\epsilon} : \mathcal{V}_{\mu,\epsilon} \to \mathcal{V}_{\mu,\epsilon}$ in the symplectic and reversible basis $\mathcal{G}$ of $\mathcal{V}_{\mu,\epsilon}$ defined in \eqref{G basis set and g}, is a Hamiltonian matrix $\mathsf{L}_{\mu,\epsilon} = \mathsf{J}_4 \mathsf{B}_{\mu,\epsilon}$, where $\mathsf{B}_{\mu,\epsilon}$ is a self-adjoint and reversibility preserving (i.e., satisfying \eqref{B are alternatively real or imaginary}) $4 \times 4$ matrix of the form
\begin{align} \label{B mu epsilon}
    \mathsf{B}_{\mu,\epsilon} = 
\begin{pmatrix}
E & F \\
F^* & G
\end{pmatrix}, \quad E = E^*, \quad G = G^*,
\end{align}
where $E, F, G$ are the $2 \times 2$ matrices
\begin{equation} \label{E}
\begin{aligned} 
E &:= 
\begin{pmatrix}
\mathsf{e}_{11} \epsilon^2(1+r'(\e,\mu\e^2))- \mathsf{e}_{22}\frac{\mu^2}{8}(1+r''(\e,\mu))  &
\im\left( \frac{1}{2}\mathsf{e}_{12} \mu + r_2(\mu \epsilon^2,\mu^2 \epsilon, \mu^3)\right) \\
- \im\left( \frac{1}{2}\mathsf{e}_{12} \mu + r_2(\mu \epsilon^2,\mu^2 \epsilon, \mu^3)\right) &
-\mathsf{e}_{22}\frac{\mu^2}{8}(1+r_5(\e,\mu))
\end{pmatrix}, 
\end{aligned}
\end{equation}
\begin{equation} \label{F 2nd}
\begin{aligned} 
F &:= 
\begin{pmatrix}
r_3(\epsilon^3, \mu \epsilon^2,\mu^2 \epsilon, \mu^3) &
\im \sqrt{\ck}\,\mu \epsilon  + \im\, r_4(\mu \epsilon^2,\mu^2 \epsilon, \mu^3) \\
\im\, r_6(\mu \epsilon,\mu^3) &
0
\end{pmatrix}, 
\end{aligned}
\end{equation}

\begin{equation} \label{G 2nd}
\begin{aligned} 
G &:= 
\begin{pmatrix}
1 +\kappa\mu^2+ r_8(\epsilon^3, \mu \epsilon^2,\mu^2 \epsilon, \mu^3) &
\im r_9(\mu \epsilon^2,\mu^2 \epsilon, \mu^3) \\
- \im r_9(\mu \epsilon^2,\mu^2 \epsilon, \mu^3) &
\mu +r_{10}(\mu^2 \epsilon, \mu^3)
\end{pmatrix}\,,
\end{aligned}
\end{equation}
where $\mathsf{e}_{11}$, $\mathsf{e}_{12}$, and $\mathsf{e}_{22}$ are defined in \eqref{e11e22e12}.


\end{lemma}

We decompose $\mathcal{B}_{\mu,\epsilon}$ in \eqref{mathcal B mu e} as
\[
\mathcal{B}_{\mu,\epsilon} = \mathcal{B}_\epsilon +\mathcal{B}^{\flat} + \mathcal{B}^\sharp+\mathcal{B}^s,
\]
where $\mathcal{B}_\epsilon$, $\mathcal{B}^{\flat}$, $\mathcal{B}^\sharp$, $\mathcal{B}^s$ are the self-adjoint and reversibility preserving operators:
\begin{align} \label{B epsilon}
\mathcal{B}_\epsilon := \mathcal{B}_{0,\epsilon} := 
\begin{bmatrix}
1 + a_\epsilon(x)-\kappa \Sigma_{0,\epsilon} & -(\ck + p_\epsilon(x))\partial_x \\
\partial_x \circ (\ck + p_\epsilon(x)) & |D| 
\end{bmatrix}, 
\end{align}

\begin{align} \label{B b}
\mathcal{B}^{\flat} := 
\begin{bmatrix}
0 & 0 \\
0 & \mu\left(\mathrm{sgn}(D)+\Pi_0\right)
\end{bmatrix}, 
\end{align}

\begin{align} \label{B sharp}
\mathcal{B}^\sharp :=  
\begin{bmatrix}
0 & -\im\, \mu p_\epsilon \\
\im\, \mu p_\epsilon & 0
\end{bmatrix}, 
\qquad 
\mathcal{B}^s :=  
\begin{bmatrix}
-\kappa \left(\mu\dot{\Sigma}_{0,\epsilon}+\frac{1}{2}\mu^2\ddot{\Sigma}_{0,\epsilon}\right) & 0 \\
0 & 0
\end{bmatrix}, 
\end{align}
where $\ttf$ in \eqref{def:ttf}, $a_\e(x)$ and $p_\e(x)$ in \eqref{SN1} and 
\begin{align*}
    \dot{\Sigma}_{0,\epsilon}:=\frac{\im}{1+\mathfrak{p}_x}\left(g_\e(x)\circ\pa_x+\pa_x\circ g_\e(x)\right)\frac{1}{1+\mathfrak{p}_x}, \quad \ddot{\Sigma}_{0,\epsilon}:=-\frac{2g_{\e}(x)}{(1+\mathfrak{p}_x)^2}.
\end{align*}
We note that $\mathcal{B}^{\flat}$ and $\mathcal{B}^{\sharp}$ are linear in $\mu$.

\begin{lemma} [Expansion of $\mathsf{B}_\epsilon$] \label{Expansion of B epsilon}
The self-adjoint and reversibility preserving matrix 
$\mathsf{B}_\epsilon := \mathsf{B}_\epsilon(\mu)$ associated, as in \eqref{Bmuepsilon matrix representation}, 
with the self-adjoint and reversibility preserving operator 
$\mathcal{B}_\epsilon$ defined in \eqref{B epsilon}, with respect to the basis 
$\mathcal{G}$ of $\mathcal{V}_{\mu,\epsilon}$ in \eqref{G basis set and g}, expands as
\begin{equation} \label{Be90}
\begin{aligned}
\mathsf{B}_\epsilon =
\left(
\begin{NiceArray}{cc|ccc}[code-for-first-col=\scriptstyle]
\mathsf{e}_{11} \epsilon^2 + \frac{(1-\kappa)^2}{8\ck^3} \mu^2 + r_1(\epsilon^3, \mu \epsilon^4)
  & \im\, r_2(\mu \epsilon^3)
  &  r_3(\epsilon^3, \mu \epsilon^2)
  & \im\, r_4(\mu \epsilon^3)
  &  \\
- \im\, r_2(\mu \epsilon^3)
  & \frac{(1-\kappa)^2}{8\ck^3} \mu^2
  & \im\, r_6(\mu \epsilon)
  & 0
  &  \\
\hline
 r_3(\epsilon^3, \mu \epsilon^2)
  & - \im\, r_6(\mu \epsilon)
  & 1 + r_8(\epsilon^3, \mu \epsilon^2)
  & \im\, r_9(\mu \epsilon^2)
  & \\
- \im\, r_4(\mu \epsilon^3)
  & 0
  & - \im\, r_9(\mu \epsilon^2)
  & 0
  &  \\
\end{NiceArray}
\right)
\end{aligned}+\cO(\mu^2\e,\mu^3),
\end{equation}
where $\mathsf{e}_{11}$ is defined respectively in \eqref{e11e22e12}.
\end{lemma}

\begin{proof}
    We expand the matrix $\mathsf{B}_\epsilon(\mu)$ as

\begin{equation} \label{expansion of B_e in mu}
\mathsf{B}_\epsilon(\mu) = \mathsf{B}_\epsilon(0) + \mu(\partial_\mu \mathsf{B}_\epsilon)(0) + \frac{\mu^2}{2} (\partial_\mu^2 \mathsf{B}_0)(0) + \mathcal{O}(\mu^2 \epsilon, \mu^3).
\end{equation}

\textbf{The matrix $\mathsf{B}_\epsilon(0)$.}
The main result of this long paragraph is to prove that the matrix $\mathsf{B}_\epsilon(0)$ has the expansion \eqref{Be(0)}. We start recalling that the self-adjoint $\mathsf{B}_\e(0)$ has the form 
\begin{align} \label{B epsilon 0}
    \mathsf{B}_{\epsilon}(0) = 
\left(
\begin{array}{cccc}
\left( \mathcal{B}_{\epsilon} g_1^{+}(\epsilon),\ g_{1}^{+}(\epsilon) \right) & 
\left( \mathcal{B}_{\epsilon} g_1^{-}(\epsilon),\ g_{1}^{+}(\epsilon) \right) & 
\left( \mathcal{B}_{\epsilon} g_0^{+}(\epsilon),\ g_{1}^{+}(\epsilon) \right) & 
\left( \mathcal{B}_{\epsilon} g_0^{-}(\epsilon),\ g_{1}^{+}(\epsilon) \right) \\
\left( \mathcal{B}_{\epsilon} g_1^{+}(\epsilon),\ g_{1}^{-}(\epsilon) \right) & 
\left( \mathcal{B}_{\epsilon} g_1^{-}(\epsilon),\ g_{1}^{-}(\epsilon) \right) & 
\left( \mathcal{B}_{\epsilon} g_0^{+}(\epsilon),\ g_{1}^{-}(\epsilon) \right) & 
\left( \mathcal{B}_{\epsilon} g_0^{-}(\epsilon),\ g_{1}^{-}(\epsilon) \right) \\
\left( \mathcal{B}_{\epsilon} g_1^{+}(\epsilon),\ g_{0}^{+}(\epsilon) \right) & 
\left( \mathcal{B}_{\epsilon} g_1^{-}(\epsilon),\ g_{0}^{+}(\epsilon) \right) & 
\left( \mathcal{B}_{\epsilon} g_0^{+}(\epsilon),\ g_{0}^{+}(\epsilon) \right) & 
\left( \mathcal{B}_{\epsilon} g_0^{-}(\epsilon),\ g_{0}^{+}(\epsilon) \right) \\
\left( \mathcal{B}_{\epsilon} g_1^{+}(\epsilon),\ g_{0}^{-}(\epsilon) \right) & 
\left( \mathcal{B}_{\epsilon} g_1^{-}(\epsilon),\ g_{0}^{-}(\epsilon) \right) & 
\left( \mathcal{B}_{\epsilon} g_0^{+}(\epsilon),\ g_{0}^{-}(\epsilon) \right) & 
\left( \mathcal{B}_{\epsilon} g_0^{-}(\epsilon),\ g_{0}^{-}(\epsilon) \right)
\end{array}
\right),
\end{align}
where $g^{\sigma}_k(\e):=g^{\sigma}_k(0,\e)$, for $\sigma=\pm$, $k=0,1$. The matrix $\mathsf{B}_\epsilon(0)$ is real, because the operator $\mathcal{B}_\epsilon$ is real and the basis $\{g_k^{\pm}(0, \epsilon)\}_{k=0,1}$ is real.
On the other hand, by \eqref{B are alternatively real or imaginary}, its matrix elements $(\mathsf{B}_\epsilon(0))_{i,j}$ vanish for $i + j$ odd. In addition, $g_0^{-}(0,\epsilon) = 
\begin{bmatrix}
0 \\ 1
\end{bmatrix}
$ by \eqref{48 f-0=01}, and, by \eqref{B epsilon}, we get $\mathcal{B}_\epsilon g_0^{-}(0,\epsilon) = 0$, for any $\epsilon$.
We deduce that the self-adjoint matrix $\mathsf{B}_\epsilon(0)$ in \eqref{B epsilon 0} has the form
\begin{equation} \label{B_e 0}
\mathsf{B}_\epsilon(0) =\left(
\begin{NiceArray}{cc|ccc}[code-for-first-col=\scriptstyle]
E_{11}(0,\epsilon)
  & 0
  & F_{11}(0,\epsilon)
  & 0
  &  \\
0
  & E_{22}(0,\epsilon)
  & 0
  & 0
  &  \\
\hline
F_{11}(0,\epsilon)
  & 0
  & G_{11}(0,\epsilon)
  & 0
  & \\
0
  & 0
  & 0
  & 0
  &  \\
\end{NiceArray}
\right),
\end{equation}
where \( E_{11}(0,\epsilon), E_{22}(0,\epsilon), G_{11}(0,\epsilon), F_{11}(0,\epsilon) \) are real.
We claim that \( E_{22}(0,\epsilon) \equiv 0 \) for any \( \epsilon \). As a first step, following \cite{BMV1}, we prove that

\begin{align} \label{E22=0 or E11=0=F11}
    \text{either } E_{22}(0,\epsilon) \equiv 0, \quad \text{or} \quad E_{11}(0,\epsilon) \equiv 0 \equiv F_{11}(0,\epsilon).
\end{align}
Indeed, by \eqref{237}, the operator \( \mathscr{L}_{0,\epsilon} \equiv \mathcal{L}_{0,\epsilon} \) possesses, for any sufficiently small \( \epsilon \neq 0 \), the eigenvalue 0 with a four dimensional generalized kernel 
\[
\mathcal{W}_{\epsilon} := \text{span}\{ U_1, \tilde{U}_2, U_3, U_4 \},
\]
spanned by \( \epsilon \)-dependent vectors. By Lemma \ref{kato thm} it results that \( \mathcal{W}_\epsilon = \mathcal{V}_{0,\epsilon} = \text{Rg}(P_{0,\epsilon}) \) and by \eqref{237} we have \( \mathscr{L}_{0,\epsilon}^2 = 0 \) on \( \mathcal{V}_{0,\epsilon} \). Thus the matrix

\begin{equation} \label{Le0}
\mathsf{L}_\epsilon(0) := \mathsf{J}_4 \mathsf{B}_\epsilon(0) = 
\left(
\begin{NiceArray}{cc|ccc}[code-for-first-col=\scriptstyle]
0
  & E_{22}(0,\epsilon)
  & 0
  & 0
  &  \\
-E_{11}(0,\epsilon)
  & 0
  & -F_{11}(0,\epsilon)
  & 0
  &  \\
\hline
0
  & 0
  & 0
  & 0
  & \\
-F_{11}(0,\epsilon)
  & 0
  & -G_{11}(0,\epsilon)
  & 0
  &  \\
\end{NiceArray}
\right),
\end{equation}
which represents \( \mathscr{L}_{0,\epsilon} : \mathcal{V}_{0,\epsilon} \to \mathcal{V}_{0,\epsilon} \), satisfies \( \mathsf{L}_\epsilon^2(0) = 0 \), namely

\begin{equation}
\mathsf{L}_\epsilon^2(0) =  
\left(
\begin{NiceArray}{cc|ccc}[code-for-first-col=\scriptstyle]
-(E_{11}E_{22})(0,\epsilon)
  & 0
  & -(E_{22}F_{11})(0,\e)
  & 0
  &  \\
0
  & -(E_{11}E_{22})(0,\epsilon)
  & 0
  & 0
  &  \\
\hline
0
  & 0
  & 0
  & 0
  & \\
0
  & -(E_{22}F_{11})(0,\epsilon)
  & 0
  & 0
  &  \\
\end{NiceArray}
\right)=0,
\end{equation}
which implies \eqref{E22=0 or E11=0=F11}. We now prove that the matrix \( \mathsf{B}_\epsilon(0) \) defined in \eqref{B_e 0} expands as

\begin{equation} \label{Be(0)}
\mathsf{B}_\epsilon(0) = 
\left(
\begin{NiceArray}{cc|ccc}[code-for-first-col=\scriptstyle]
\mathsf{e}_{11}\e^2+r(\e^3)
  & 0
  & r(\e^3)
  & 0
  &  \\
0
  & 0
  & 0
  & 0
  &  \\
\hline
r(\e^3)
  & 0
  & 1+r(\e^3)
  & 0
  & \\
0
  & 0
  & 0
  & 0
  &  \\
\end{NiceArray}
\right)\,,
\end{equation}
where $\mathsf{e}_{11}$ is in \eqref{E110epsilon}. We expand the operator $\mathcal{B}_{\e}$ in \eqref{B epsilon} as 
\begin{equation} \label{Be B0 B1 B2}
\begin{aligned}
    &\mathcal{B}_{\e}=\mathcal{B}_{0}+\e \mathcal{B}_{1}+\e^2 \mathcal{B}_{2}+\cO(\e^3), \quad \mathcal{B}_{0}:=\begin{bmatrix}
1-\kappa\pa^2_{x}&  -\ck\pa_x\\
        \ck\pa_x & ~~|D|
    \end{bmatrix},\\
    &\mathcal{B}_{1}:=\begin{bmatrix}
a_1(x)-\kappa\Sigma_1&  -p_1(x)\pa_x\\
       \pa_x\circ p_1(x) & ~~0
    \end{bmatrix}, \quad \mathcal{B}_{2}:=\begin{bmatrix}
a_2(x)-\kappa\Sigma_2&  -p_2(x)\pa_x\\
        \pa_x \circ p_2(x) & ~~ 0
    \end{bmatrix},
\end{aligned}
\end{equation}
where the remainder term $\cO(\e^3)\in\mathcal{L}(Y,X)$, the functions $a_1$, $p_1$, $a_2$, $p_2$,  are given in \eqref{SN1}-\eqref{aino2fd}, the operators $\Sigma_1$, $\Sigma_2$ are given in \eqref{Sigma epsilon}-\eqref{h2 h02}.

Our goal here is to determine the expansion of $E_{11}(0,\e)=\mathsf{e}_{11}\e^2+r(\e^3)$. By \eqref{43 g+1} we rewrite the real function $g^+_1(0,\e)$ as
\begin{equation} \label{expansion fp1 0 epsilon}
    \begin{aligned}
        & g^+_1(0,\e)=f^+_1+\e g^+_{1_1}+\e^2 g^+_{1_2}+\cO(\e^3),\\
        & f^+_1=\vet{\frac{1}{\sqrt{\ck}}\cos(x)}{\sqrt{\ck}\sin(x)}, \quad g^+_{1_1}:=\vet{\frac{2-\kappa}{(1-2\kappa)\sqrt{\ck}} \cos(2x)}{\frac{\ck^2\sqrt{\ck}}{1-2\kappa}\sin(2x)}, \quad g^+_{1_2}:=\vet{\mathit{even}_0(x)}{\mathit{odd}(x)},
    \end{aligned}
\end{equation}
where both $g^+_{1_2}$ and $\cO(\e^3)$ are vectors in $H^2(\mathbb{T},\mathbb{C})\times H^1(\mathbb{T},\mathbb{C})$. Since $\mathcal{B}_0 f^+_1=-\mathcal{J}\mathscr{L}_{0,0} f^+_1=0$, and both $\mathcal{B}_0$, $\mathcal{B}_1$ are self-adjoint real operators, it results
\begin{equation}
\begin{aligned}
E_{11}(0, \epsilon) &= \left( \mathcal{B}_\epsilon g_1^+(0,\epsilon), g_1^+(0,\epsilon) \right) \\
&= \epsilon \left( \mathcal{B}_1 f_1^+, f_1^+ \right) + \epsilon^2 \left[ \left( \mathcal{B}_2 f_1^+, f_1^+ \right) + 2 \left( \mathcal{B}_1 f_1^+, g^+_{1_1} \right) + \left( \mathcal{B}_0 g^+_{1_1}, g^+_{1_1} \right) \right] + \mathcal{O}(\epsilon^3).
\end{aligned}
\end{equation}
By \eqref{Be B0 B1 B2} one has
\begin{equation} \label{B1fp1 B2fp1 B0fp11}
\mathcal{B}_1 f_1^+ = 
\begin{bmatrix}
-\frac{3\kappa}{\sqrt{\ck}}\cos(2x) \\
2\sqrt{\ck}\sin(2x)
\end{bmatrix},
\quad 
\mathcal{B}_2 f_1^+ = 
\begin{bmatrix}
A^{[1]}_2 \cos(x) + A^{[3]}_2 \cos(3x) \\
B^{[1]}_2 \sin(x) + B^{[3]}_2 \sin(3x)
\end{bmatrix}, \quad
\mathcal{B}_0 g_{1_1}^+ = 
\begin{bmatrix}
\frac{3\kappa}{\sqrt{\ck}}\cos(2x) \\
-2\sqrt{\ck}\sin(2x)
\end{bmatrix},
\end{equation}
with 
\begin{equation} \label{ABABAB}
\begin{aligned}
A^{[1]}_2 :=\frac{118\kappa ^2+11\kappa -8}{16\left(2\kappa -1\right) \sqrt{\ck}}, \quad A^{[3]}_2 :=\frac{3\kappa \left(26\kappa +17\right)}{8\left(2\kappa -1\right)\sqrt{\ck}},\quad B^{[1]}_2 :=\frac{46\kappa ^2+47\kappa -8}{16\left(1-2\kappa \right)\ck\sqrt{\ck}},\quad B^{[3]}_2:=\frac{3\ck^2\sqrt{\ck}}{1-2\kappa}. 
\end{aligned}
\end{equation}
By \eqref{expansion fp1 0 epsilon}-\eqref{ABABAB}, a straightforward calculation reveals that
\begin{equation} \label{E110epsilon}
\begin{aligned}
E_{11}(0, \epsilon) &= \mathsf{e}_{11} \epsilon^2 + r(\epsilon^3)=\frac{2\kappa^2+\kappa+8}{8(1-2\kappa)\ck}\e^2+ r(\epsilon^3)\,. 
\end{aligned}
\end{equation}
Since $\mathsf{e}_{11}$ is not zero for $\kappa \not \in \fR_2$, the second alternative in \eqref{E22=0 or E11=0=F11} is ruled out, implying $E_{22}(0, \epsilon) \equiv 0$.
Next let us determine the expansion of $G_{11}(0,\e)$. By \eqref{45 g+0} we split the real-valued function $g_0^+(0, \epsilon)$ as
\begin{equation} \label{fp0 expansion}
\begin{aligned}
&g_0^+(0, \epsilon) = f_0^+ + \epsilon g_{0_1}^+ + \epsilon^2g_{0_2}^+ + \mathcal{O}(\epsilon^3), \\  
&f_0^+ = \begin{bmatrix} 1 \\ 0 \end{bmatrix},\quad
g_{0_1}^+ := 
\begin{bmatrix}
 \cos(x) \\
- \ck \sin(x)
\end{bmatrix}, \quad
g_{0_2}^+ := 
\begin{bmatrix}
\mathit{even}_0(x) \\
\mathit{odd}(x)
\end{bmatrix}. 
\end{aligned}
\end{equation}
Since, by \eqref{eigenfunc of mathcall L00}, \eqref{eigenfunc f+0}, and \eqref{Be B0 B1 B2}, $\mathcal{B}_0 f_0^+ = f_0^+$, using that $\mathcal{B}_0$, $\mathcal{B}_1$ are self-adjoint real operators, and $(f_0^+,f_0^+) = 1$, $(f_0^+, g_{0_1}^+)=0$, and $(f_0^+, g_{0_2}^+)=0$, we have
\begin{align*}
G_{11}(0, \epsilon) 
&= (\mathcal{B}_\epsilon g_0^+(0, \epsilon), g_0^+(0, \epsilon)) \\
&= 1 + \epsilon(\mathcal{B}_1 f_0^+, f_{0}^+) +\epsilon^2\left[(\mathcal{B}_0 g_{0_1}^+, g_{0_1}^+)+2(\mathcal{B}_1 f_0^+, g_{0_1}^+)+(\mathcal{B}_2 f_0^+, f_{0}^+)\right]+
r(\epsilon^3).
\end{align*}
By \eqref{Be B0 B1 B2}, \eqref{fp0 expansion}, and \eqref{SN1}, \eqref{pino1fd}, \eqref{aino1fd}, and \eqref{Sigma epsilon}-\eqref{h2 h02} one has
\begin{equation} \label{B1fp0fp0 B0fp01fp0}
    \begin{aligned}
        \left(\mathcal{B}_1 f_0^+,f^+_0\right)=0,\quad (\mathcal{B}_0 g_{0_1}^+, g_{0_1}^+)=\left(\mathcal{B}_2 f_0^+,f^+_0\right)=-(\mathcal{B}_1 f_0^+, g_{0_1}^+)=2\ck^2.
    \end{aligned}
\end{equation}
By \eqref{B1fp0fp0 B0fp01fp0}, we deduce $G_{11}(0, \epsilon) = 1 + r(\epsilon^3)$.

Next we provide the expansion of  $F_{11}(0, \epsilon)$.  
By \eqref{Be B0 B1 B2}, \eqref{expansion fp1 0 epsilon}, \eqref{fp0 expansion}, using that $\mathcal{B}_0$, $\mathcal{B}_1$ are self-adjoint and real, and $\mathcal{B}_0 f_1^+ = 0$, $\mathcal{B}_0 f_0^+ = f_0^+$, we obtain
\begin{align*}
F_{11}(0, \epsilon) 
&= \epsilon \left[ \left( \mathcal{B}_1 f_1^+, f_0^+ \right) + \left(  \mathcal{B}_0g_{1_1}^+, f_0^+ \right) \right] \\
&\quad + \epsilon^2 \left[ \left( \mathcal{B}_2 f_1^+, f_0^+ \right) + \left( \mathcal{B}_1 f_1^+, g_{0_1}^+ \right) + \left( \mathcal{B}_1 f_0^+, g_{1_1}^+ \right)  + \left( \mathcal{B}_0g_{1_2}^+, f_0^+ \right) + \left( \mathcal{B}_0 g_{1_1}^+, g_{0_1}^+ \right) \right] + r(\epsilon^3).
\end{align*}
A straightforward calculation reveals that all the scalar product vanish, yielding
\begin{align} \label{F110e}
    F_{11}(0, \epsilon) =  r(\epsilon^3).
\end{align}
The expansion \eqref{Be(0)} is proved.

The next step is to compute the terms linear in $\mu$ of $\mathsf{B}_\epsilon(\mu)$. We have 
\begin{align} \label{X+X}
    \partial_\mu \mathsf{B}_\epsilon(0) = X + X^* \quad \text{where} \quad 
X := \left( \left( \mathcal{B}_\epsilon g_k^\sigma(0,\epsilon), \, (\partial_\mu g_{k'}^{\sigma'})(0,\epsilon) \right) \right)_{\sigma,\sigma'=\pm;\,k,k'=0,1}. 
\end{align}
and now prove that 
\begin{equation} \label{X}
X = 
\left(
\begin{NiceArray}{cc|ccc}[code-for-first-col=\scriptstyle]
\cO(\e^4)
  & 0
  & \cO(\e^2)
  & 0
  &  \\
\cO(\e^3)
  & 0
  & \cO(\e)
  & 0
  &  \\
\hline
\cO(\e^3)
  & 0
  & \cO(\e^2)
  & 0
  & \\
\cO(\e^3)
  & 0
  & \cO(\e^2)
  & 0
  &  \\
\end{NiceArray}
\right).
\end{equation}
Indeed consider the matrix $\mathsf{L}_\epsilon(0)$ in \eqref{Le0}, where $E_{22}(0, \epsilon) = 0$, and recall that it represents the action of the operator $\mathscr{L}_{0,\epsilon} : \mathcal{V}_{0,\epsilon} \to \mathcal{V}_{0,\epsilon}$ in the basis $\{g_k^\sigma(0,\epsilon)\}$;  we deduce that
\[
\mathscr{L}_{0,\epsilon} g_1^-(0,\epsilon) = 0,\quad \mathscr{L}_{0,\epsilon} g_0^-(0,\epsilon) = 0.
\]
Thus also $\mathcal{B}_\epsilon g_1^-(0,\epsilon) =\mathcal{B}_\epsilon g_0^-(0,\epsilon) = 0$, and the second and fourth columns of the matrix $X$ in \eqref{X} are zero. Now we compute $\pa_\mu g_k^\sigma(0,\e)$.  
In view of \eqref{43 g+1}-\eqref{46 g-0} and by denoting with a dot the derivative w.r.t.\ $\mu$, one has
\begin{equation} \label{dotfp1 dotfp0 dotfn1 dotfn0}
    \begin{aligned}
\dot{g}_1^+(0,\epsilon) &= \im\frac{1-\kappa}{4\ck^2}  
\begin{bmatrix}
\frac{1}{\sqrt{\ck}}\sin(x) \\
\sqrt{\ck} \cos(x)
\end{bmatrix}
+ \im\epsilon 
\begin{bmatrix}
\mathit{odd}(x) \\
\mathit{even}(x)
\end{bmatrix}
+ \mathcal{O}(\epsilon^2),  \\
\dot{g}_1^-(0,\epsilon) &= \im\frac{1-\kappa}{4\ck^2}  
\begin{bmatrix}
\frac{1}{\sqrt{\ck}} \cos(x)\\
-\sqrt{\ck}\sin(x) 
\end{bmatrix}
+ \red{\epsilon\frac{-\kappa}{4\ck\sqrt{\ck}}\begin{bmatrix}
0 \\
1
\end{bmatrix}}+
\im\epsilon 
\begin{bmatrix}
\mathit{even}(x) \\
\mathit{odd}(x)
\end{bmatrix}
+ \mathcal{O}(\epsilon^2), \\
\dot{g}_0^+(0,\epsilon) &= \red{\epsilon
\begin{bmatrix}
\frac14\cos(x)\\[1mm]
\frac{\kappa-1}{4\ck}\sin(x)
\end{bmatrix}}+\im\epsilon 
\begin{bmatrix}
\mathit{odd}(x) \\
\mathit{even}_0(x)
\end{bmatrix}
+ \mathcal{O}(\epsilon^2),\\
\dot{g}_0^-(0,\epsilon) &= \frac{1}{2}\epsilon\begin{bmatrix}
\frac{1}{\ck} \sin(x)\\
\cos(x)
\end{bmatrix}+
\im\epsilon 
\begin{bmatrix}
\mathit{even}_0(x) \\
\mathit{odd}(x)
\end{bmatrix}
+ \mathcal{O}(\epsilon^2). 
\end{aligned}
\end{equation}
In view of \eqref{pa_t eta psi}, \eqref{43 g+1}-\eqref{46 g-0}, \eqref{Le0}, \eqref{E110epsilon}, \eqref{F110e}, and since $\mathcal{B}_\epsilon g_k^\sigma(0, \epsilon) = -\mathcal{J} \mathscr{L}_\epsilon g_k^\sigma(0, \epsilon)$, we have
\begin{equation} \label{Befp1 Befp0}
   \begin{aligned}
    \mathcal{B}_\epsilon g_1^+(0, \epsilon) &= E_{11}(0,\epsilon) \, \mathcal{J} g_1^-(0,\epsilon) + F_{11}(0,\epsilon) \, \mathcal{J} g_0^-(0, \epsilon) 
    \\
    &=\left(\epsilon^2 \mathsf{e}_{11}+\cO(\e^3)\right)\left(
    \begin{bmatrix}
        \sqrt{\ck} \cos(x) \\
        \frac{1}{\sqrt{\ck}} \sin(x)
    \end{bmatrix}+\e\begin{bmatrix}
        \frac{\ck^2\sqrt{\ck}}{1-2\kappa} \cos(2x) \\
        \frac{2-\kappa}{(1-2\kappa)\sqrt{\ck}} \sin(2x)
    \end{bmatrix}+ \cO(\epsilon^2)\right)+\cO(\e^3)\vet{1}{0}
    ,
    \\
    \mathcal{B}_\epsilon g_0^+(0, \epsilon) &= F_{11}(0,\epsilon) \, \mathcal{J} g_1^-(0,\epsilon) + G_{11}(0,\epsilon) \, \mathcal{J} g_0^- (0, \epsilon)
    = 
    \begin{bmatrix}
        1 \\
        0
    \end{bmatrix}
    +\mathcal{O}(\epsilon^3).
\end{aligned} 
\end{equation}
We deduce \eqref{X} by \eqref{dotfp1 dotfp0 dotfn1 dotfn0} and \eqref{Befp1 Befp0}.

Next we compute the terms quadratic in $\mu$. By denoting with a double dot the double derivative w.r.t. $\mu$, we arrive at
\begin{align} \label{ddotB00}
\partial_\mu^2 \mathsf{B}_0(0) = \left( \mathcal{B}_0 g_k^\sigma(0,0), \ddot{g}_{k'}^{\sigma'}(0,0) \right)
+ \left( \ddot{g}_k^\sigma(0,0), \mathcal{B}_0 g_{k'}^{\sigma'}(0,0) \right)
+ 2 \left( \mathcal{B}_0 \dot{g}_k^\sigma(0,0), \dot{g}_{k'}^{\sigma'}(0,0) \right)
= : Y + Y^* + 2Z. 
\end{align}
We claim that $Y = 0$. Indeed, its first, second and fourth column are zero, since
 $\mathcal{B}_0 g^\sigma_k(0,0)=-\mathcal{J}\mathscr{L}_{0,0} g^\sigma_k(0,0)=0$ for $g_k^\sigma(0,0) \in \{ g_1^+(0,0), g_1^-(0,0), g_0^-(0,0) \}$. The third column is also zero by noting that $\mathcal{B}_0 g_0^+ = g_0^+$ and
 \begin{align*}
 \ddot{g}_1^+(0,0) =
 \begin{bmatrix}
 \mathit{even}_0(x) + \im\mathit{odd}(x) \\
 \mathit{odd}(x) + \im\mathit{even}_0(x)
 \end{bmatrix}, \quad
 \ddot{g}_1^-(0,0) =
 \begin{bmatrix}
 \mathit{odd}(x) + \im\mathit{even}_0(x) \\
 \mathit{even}_0(x) + \im\mathit{odd}(x)
 \end{bmatrix},\quad \ddot{g}_0^+(0,0) = \ddot{g}_0^-(0,0) = 0.
 \end{align*}
 We claim that
\begin{equation} \label{Z}
Z = \left( \mathcal{B}_0 \dot{g}_k^\sigma(0,0), \dot{g}_{k'}^{\sigma'}(0,0) \right)_{\sigma,\sigma'=\pm;\,k,k'=0,1}
= \left(
\begin{NiceArray}{cc|ccc}[code-for-first-col=\scriptstyle]
\frac{(1-\kappa)^2}{8\ck^3}
  & 0
  & 0
  & 0
  &  \\
0
  &\frac{(1-\kappa)^2}{8\ck^3}
  & 0
  & 0
  &  \\
\hline
0
  & 0
  & 0
  & 0
  & \\
0
  & 0
  & 0
  & 0 
  &  \\
\end{NiceArray}
\right),
\end{equation}
Indeed, by \eqref{dotfp1 dotfp0 dotfn1 dotfn0}, we have $\dot{g}_0^+(0,0) = \dot{g}_0^-(0,0) = 0$. Therefore the last two columns of $Z$,
and by self-adjointness the last two rows, are zero. Also, using the orthogonality, we have
 \begin{align}
     \left( \mathcal{B}_0 \dot{g}_1^+(0,0), \dot{g}_1^-(0,0) \right)=0
 \end{align}
 By \eqref{Be B0 B1 B2}, \eqref{dotfp1 dotfp0 dotfn1 dotfn0}, we obtain the matrix \eqref{Z} with
$
\left( \mathcal{B}_0 \dot{g}_1^+(0,0), \dot{g}_1^+(0,0) \right)
= \left( \mathcal{B}_0 \dot{g}_1^-(0,0), \dot{g}_1^-(0,0) \right)
= \frac{(1-\kappa)^2}{8\ck^3}.
$
In conclusion, \eqref{expansion of B_e in mu}, \eqref{X+X}, \eqref{X}, \eqref{ddotB00}, the fact that $Y = 0$ and \eqref{Z} imply \eqref{Be90}, using also the self-adjointness of $\mathcal{B}_{\e}$ and \eqref{B are alternatively real or imaginary}. 
\end{proof}

\begin{lemma}
    [Expansion of $\mathsf{B}^{\flat}$] \label{Expansion of B flat} The self-adjoint and reversibility-preserving matrix $\mathsf{B}^{\flat}$ associated, as in \eqref{Bmuepsilon matrix representation}, to the self-adjoint and reversibility-preserving operator $\mathcal{B}^{\flat}$, defined in \eqref{B b}, with respect to the basis $\mathcal{G}$ of $\mathcal{V}_{\mu, \epsilon}$ in \eqref{G basis set and g}, admits the expansion
\begin{equation} \label{mathsfBb}
\mathsf{B}^{\flat} = 
\left(
\begin{NiceArray}{cc|ccc}[code-for-first-col=\scriptstyle]
-\frac{1-\kappa}{4\ck}\mu^2
  & \im\,(\frac{\ck}{2}\mu+r_2(\mu\e^2))
  & 0
  & 0
  &  \\
- \im\,(\frac{\ck}{2}\mu+r_2(\mu\e^2))
  &-\frac{1-\kappa}{4\ck}\mu^2
  & \im\, r_6(\mu\e)
  & 0
  &  \\
\hline
0
  & -\im\, r_6(\mu\e)
  & 0
  & 0
  & \\
0
  & 0
  & 0
  & \mu 
  &  \\
\end{NiceArray}
\right)
+ \mathcal{O}(\mu^2 \epsilon, \mu^3).
\end{equation}
\end{lemma}
\begin{proof}
We compute the matrix 
\begin{align*}
\left( \left( \mathcal{B}^{\flat}g_k^\sigma(\mu,\epsilon), \, g_{k'}^{\sigma'}(\mu,\epsilon) \right) \right)_{\sigma,\sigma'=\pm;\,k,k'=0,1}
\end{align*}
with respect to the basis $\mathcal{G}$. Since $\mathcal{B}^{\flat}$ is self-adjoint and reversibility-preserving, it is enough to determine the entries in the upper triangular part and then use the corresponding symmetry relations. We begin with the last column and row. By \eqref{46 g-0} and the definition \eqref{B b}, we obtain
$
\mathcal{B}^{\flat} g_0^{-}(\mu,\epsilon)
=
\begin{bmatrix}
0\\
\mu
\end{bmatrix}
+\mathcal{O}(\mu^2\epsilon).
$
Using \eqref{43 g+1}--\eqref{46 g-0}, we infer that
\[
\begin{aligned}
\bigl( \mathcal{B}^{\flat} g_0^{-}(\mu,\epsilon), g_1^{+}(\mu,\epsilon) \bigr)
&= \mathcal{O}(\mu^2\epsilon),\quad \bigl( \mathcal{B}^{\flat} g_0^{-}(\mu,\epsilon), g_1^{-}(\mu,\epsilon) \bigr)= \mathcal{O}(\mu^2\epsilon),\\
\bigl( \mathcal{B}^{\flat} g_0^{-}(\mu,\epsilon), g_0^{+}(\mu,\epsilon) \bigr)
&= \mathcal{O}(\mu^2\epsilon),\quad \bigl( \mathcal{B}^{\flat} g_0^{-}(\mu,\epsilon), g_0^{-}(\mu,\epsilon) \bigr)
= \mu + \mathcal{O}(\mu^2\epsilon).
\end{aligned}
\]
This gives the expansion of the fourth column and, by self-adjointness, of the fourth row as well.

Next we determine the upper-left $3\times 3$ block. By \eqref{B b} together with \eqref{43 g+1}--\eqref{45 g+0}, we have
\begin{align*}
\mathcal{B}^{\flat} g_1^{+}(\mu,\epsilon)
&=
-\im \mu
\begin{bmatrix}
0\\
\sqrt{\ck}\cos(x)
\end{bmatrix}
-\mu^2\frac{1-\kappa}{4\ck^2}
\begin{bmatrix}
0\\
\sqrt{\ck}\sin(x)
\end{bmatrix}
-\im \mu\epsilon \frac{\ck^2\sqrt{\ck}}{1-2\kappa}
\begin{bmatrix}
0\\
\cos(2x)
\end{bmatrix} \\
&\qquad
+\im \mathcal{O}(\mu\epsilon^2)
\begin{bmatrix}
0\\
\mathit{even}_0(x)
\end{bmatrix}
+\mathcal{O}(\mu^2\epsilon,\mu^3),
\\[0.3em]
\mathcal{B}^{\flat} g_1^{-}(\mu,\epsilon)
&=\im \mu
\begin{bmatrix}
0\\
\sqrt{\ck}\sin(x)
\end{bmatrix}
-\mu^2\frac{1-\kappa}{4\ck^2}
\begin{bmatrix}
0\\
\sqrt{\ck}\cos(x)
\end{bmatrix}
+\im \mu\epsilon \frac{\ck^2\sqrt{\ck}}{1-2\kappa}
\begin{bmatrix}
0\\
\sin(2x)
\end{bmatrix} \\
&\qquad
+\im \mathcal{O}(\mu\epsilon^2)
\begin{bmatrix}
0\\
\mathit{odd}(x)
\end{bmatrix}
+\mathcal{O}(\mu^2\epsilon,\mu^3),
\\
\mathcal{B}^{\flat} g_0^{+}(\mu,\epsilon)
&=
\im \mu\epsilon 
\begin{bmatrix}
0\\
\ck\cos(x)
\end{bmatrix}
+\im \mathcal{O}(\mu\epsilon^2)
\begin{bmatrix}
0\\
\mathit{even}_0(x)
\end{bmatrix}
+\mathcal{O}(\mu^2\epsilon).
\end{align*}
Pairing these expansions with the basis vectors in $\mathcal{G}$, and using again \eqref{43 g+1}--\eqref{46 g-0}, we obtain
\[
\begin{aligned}
\bigl( \mathcal{B}^{\flat} g_1^{+}(\mu,\epsilon), g_1^{+}(\mu,\epsilon) \bigr)
&= -\mu^2\frac{1-\kappa}{4\ck} + \mathcal{O}(\mu^2\epsilon,\mu^3),\\
\bigl( \mathcal{B}^{\flat} g_1^{+}(\mu,\epsilon), g_1^{-}(\mu,\epsilon) \bigr)
&= -\im \mu\frac{\ck}{2} + \im \mathcal{O}(\mu\epsilon^2) + \mathcal{O}(\mu^2\epsilon,\mu^3),\\
\bigl( \mathcal{B}^{\flat} g_1^{+}(\mu,\epsilon), g_0^{+}(\mu,\epsilon) \bigr)
&= \mathcal{O}(\mu^2\epsilon,\mu^3),\\
\bigl( \mathcal{B}^{\flat} g_0^{+}(\mu,\epsilon), g_1^{-}(\mu,\epsilon) \bigr)
&= \im\mu\epsilon \frac{\ck\sqrt{\ck}}{2}
   + \im \mathcal{O}(\mu\epsilon^2)
   + \mathcal{O}(\mu^2\epsilon,\mu^3)\\
&= \im\,\mathcal{O}(\mu\epsilon)+\mathcal{O}(\mu^2\epsilon,\mu^3),\\
\bigl( \mathcal{B}^{\flat} g_0^{+}(\mu,\epsilon), g_0^{+}(\mu,\epsilon) \bigr)
&= \mathcal{O}(\mu^2\epsilon),\\
\bigl( \mathcal{B}^{\flat} g_1^{-}(\mu,\epsilon), g_1^{-}(\mu,\epsilon) \bigr)
&= -\mu^2\frac{1-\kappa}{4\ck} + \mathcal{O}(\mu^2\epsilon,\mu^3).
\end{aligned}
\]
In particular, the off-diagonal entry coupling $g_1^{-}(\mu,\e)$ and $g_0^{+}(\mu,\e)$ is purely imaginary of size $\mathcal{O}(\mu\epsilon)$, which we denote by $\im\, r_6(\mu\epsilon)$. Collecting the above expansions, and using self-adjointness together with reversibility to recover the remaining entries, we obtain exactly \eqref{mathsfBb}.
\end{proof}

Next, we consider $\mathcal{B}^{\sharp}$.
\begin{lemma} [Expansion of $\mathsf{B}^{\sharp}$] \label{Expansion of B sharp}
    The self-adjoint and reversibility-preserving matrix $\mathsf{B}^{\sharp}$ associated, as in \eqref{Bmuepsilon matrix representation}, to the self-adjoint and reversibility-preserving operators $\mathcal{B}^{\sharp}$, defined in \eqref{B sharp}, with respect to the basis $\mathcal{G}$ of $\mathcal{V}_{\mu,\epsilon}$ in \eqref{G basis set and g}, admits the expansion
\begin{align} \label{mathsf B sharp}
\mathsf{B}^{\sharp} =
\left(
\begin{NiceArray}{cc|ccc}[code-for-first-col=\scriptstyle]
0
  & \im\, r_2(\mu\e^2)
  & 0
  & \im\sqrt{\ck}\,\mu\e+\im\, r_4(\mu\e^2)
  &  \\
-\im\, r_2(\mu\e^2)
  & 0
  & -\im\, r_6(\mu\e)
  & 0
  &  \\
\hline
0
  & \im\, r_6(\mu\e)
  & 0
  & -\im\, r_9(\mu\e^2)
  & \\
-\im\sqrt{\ck}\, \mu \e -\im \,r_4(\mu\e^2)
  & 0
  & \im\, r_9(\mu\e^2)
  & 0
  &  \\
\end{NiceArray}
\right)
+ \mathcal{O}(\mu^2 \epsilon).
\end{align}
\end{lemma}
\begin{proof}
    Since $\mathcal{B}^{\sharp} = -\mathrm{i} \mu p_{\epsilon} \mathcal{J}$ and $p_{\epsilon} = \mathcal{O}(\epsilon)$ by \eqref{SN1}, we have the expansion
 \begin{equation}
 \left( \mathcal{B}^{\sharp} g_k^\sigma (\mu, \epsilon),\, g_{k'}^{\sigma'}(\mu, \epsilon) \right)_{\sigma,\sigma'=\pm;\,k,k'=0,1}
 = \left( \mathcal{B}^{\sharp} g_k^\sigma (0, \epsilon),\, g_{k'}^{\sigma'}(0, \epsilon) \right)_{\sigma,\sigma'=\pm;\,k,k'=0,1}
 + \mathcal{O}(\mu^2 \epsilon).
 \end{equation}
The matrix entries $\left( \mathcal{B}^{\sharp} g_k^\sigma (0, \epsilon),\, g_{k'}^{\sigma}(0, \epsilon) \right),\,\sigma = \pm,\,k, k' = 0, 1,$ are zero, because they are simultaneously real by \eqref{B are alternatively real or imaginary}, and purely imaginary, being the operator $\mathcal{B}^{\sharp}$ purely imaginary and the basis $\{ g_k^{\pm}(0, \epsilon) \}_{k=0,1}$ real. Hence $\mathsf{B}^{\sharp}$ has the form
 \begin{equation}
 \mathsf{B}^{\sharp} =
 \left(
 \begin{NiceArray}{cc|ccc}[code-for-first-col=\scriptstyle]
 0 & \mathrm{i} a^{\sharp} & 0 & \mathrm{i} b^{\sharp} \\
 -\mathrm{i} a^{\sharp} & 0 & -\mathrm{i} c^{\sharp} & 0 \\
 \hline
 0 & \mathrm{i} c^{\sharp} & 0 & \mathrm{i} d^{\sharp} \\
 -\mathrm{i} b^{\sharp} & 0 & -\mathrm{i} d^{\sharp} & 0
 \end{NiceArray}
 \right)
 + \mathcal{O}(\mu^2 \epsilon),
 \end{equation}
 where
 \begin{equation}
 \begin{aligned}
 & \left( \mathcal{B}^{\sharp} g_1^{-}(0, \epsilon),\, g_1^{+}(0, \epsilon) \right) := \mathrm{i} a^{\sharp}, \quad \left( \mathcal{B}^{\sharp} g_1^{-}(0, \epsilon),\, g_0^{+}(0, \epsilon) \right) := \mathrm{i} c^{\sharp}, \\
 & \left( \mathcal{B}^{\sharp} g_0^{-}(0, \epsilon),\, g_1^{+}(0, \epsilon) \right) := \mathrm{i} b^{\sharp}, \quad \left( \mathcal{B}^{\sharp} g_0^{-}(0, \epsilon),\, g_0^{+}(0, \epsilon) \right) := \mathrm{i} d^{\sharp},
 \end{aligned}
 \end{equation}
 and $a^{\sharp}, b^{\sharp}, c^{\sharp}, d^{\sharp}$ are real numbers. As $\mathcal{B}^{\sharp} = \mathcal{O}(\mu \epsilon)$ in $\mathcal{L}(Y,X)$, we deduce that $c^{\sharp} = r(\mu \epsilon)$. To determine the leading order of $a^{\sharp}$, $b^{\sharp}$ and $d^{\sharp}$, by \eqref{pino1fd} and \eqref{B sharp} we use the expansion 
 \begin{equation} \label{B sharp B1 sharp}
 \mathcal{B}^{\sharp} = \mathrm{i} \mu \epsilon \mathcal{B}_1^{\sharp} + \mathcal{O}(\mu \epsilon^2),
 \quad \mathcal{B}_1^{\sharp} := 2\ck \cos(x)
 \begin{bmatrix}
 0 & \operatorname{Id} \\
 -\operatorname{Id} & 0
 \end{bmatrix},
 \end{equation}
 with $\mathcal{O}(\mu \epsilon^2) \in \mathcal{L}(Y,X)$. In view of \eqref{43 g+1}-\eqref{46 g-0}, we have $g_1^{\pm}(0, \epsilon) = g_1^{\pm}(0,0) + \mathcal{O}(\epsilon)$, $g_0^{+}(0, \epsilon) = g_0^{+}(0,0) + \mathcal{O}(\epsilon)$, $g_0^{-}(0, \epsilon) = \begin{bmatrix} 0 \\ 1 \end{bmatrix}$. By \eqref{B sharp B1 sharp} we obtain
 \[
 \mathcal{B}_1^{\sharp} g_1^{-}(0,0) =
 \begin{bmatrix}
  \ck\sqrt{\ck}\left(1 + \cos(2x)\right) \\
 \sqrt{\ck}\sin(2x)
 \end{bmatrix},
 \quad
 \mathcal{B}_1^{\sharp} g_0^{-}(0,0) =
 \begin{bmatrix}
 2 \ck \cos(x) \\
 0
 \end{bmatrix},
 \]
 and then
 \begin{align*}
 a^{\sharp} &= \mu \epsilon \left( \mathcal{B}_1^{\sharp} g_1^{-}(0,0),\, g_1^{+}(0,0) \right) + r(\mu \epsilon^2) = r(\mu \epsilon^2), \\
 b^{\sharp} &= \mu \epsilon \left( \mathcal{B}_1^{\sharp} g_0^{-}(0,0),\, g_1^{+}(0,0) \right) + r(\mu \epsilon^2) = \mu \epsilon \sqrt{\ck}+ r(\mu \epsilon^2), \\
 d^{\sharp} &= \mu \epsilon \left( \mathcal{B}_1^{\sharp} g_0^{-}(0,0),\, g_0^{+}(0,0) \right) + r(\mu \epsilon^2) = r(\mu \epsilon^2).
 \end{align*}
 This proves \eqref{mathsf B sharp}. 
 \end{proof}

Finally, we consider the term  $\mathcal{B}^{s}$ that takes into the account the capillary term.
\begin{lemma} [Expansion of $\mathsf{B}^{s}$] \label{Expansion of B s}
    The self-adjoint and reversibility-preserving matrix $\mathsf{B}^{s}$ associated, as in \eqref{Bmuepsilon matrix representation}, to the self-adjoint and reversibility-preserving operators $\mathcal{B}^{s}$, defined in \eqref{B sharp}, with respect to the basis $\mathcal{G}$ of $\mathcal{V}_{\mu,\epsilon}$ in \eqref{G basis set and g}, admits the expansion
\begin{equation}\label{mathsf B s}
\begin{aligned} 
\mathsf{B}^{s} &=
\left(
\begin{NiceArray}{cc|ccc}[code-for-first-col=\scriptstyle]
\frac{\kappa}{\ck^3}\mu^2
  & \im\,\frac{\kappa}{\ck}\mu+\im \,r_2(\mu\e^2)
  & 0
  & \im \,r_4(\mu\e^2)
  &  \\
-\im\,\frac{\kappa}{\ck}\mu-\im \,r_2(\mu\e^2)
  & \frac{\kappa}{\ck^3}\mu^2
  & \im \,r_6(\mu\e)
  & 0
  &  \\
\hline
0
  & -\im \,r_6(\mu\e)
  & \kappa\mu^2
  & \im \,r_9(\mu\e^2)
  & \\
-\im \,r_4(\mu\e^2)
  & 0
  & -\im \,r_9(\mu\e^2)
  & 0
  &  \\
\end{NiceArray}
\right)+ \mathcal{O}(\mu^2 \epsilon,\mu^3)\,.
\end{aligned}
\end{equation}
\end{lemma}
\begin{proof}
    Recalling \eqref{expfe}, \eqref{SN1 g}, and \eqref{exp:g1}, we write the operator $\mathcal{B}^{s}$ in \eqref{B sharp} as
    \begin{align}
        \mathcal{B}^s=-\kappa\mu\mathcal{B}^s_1- \kappa \mu^2\mathcal{B}^s_2,
    \end{align}
    where
    \begin{equation}\label{B s 1}
            \begin{aligned} 
        \mathcal{B}^s_1:=&\frac{\im}{1+\mathfrak{p}_x}\left(g_\e(x)\circ\pa_x+\pa_x\circ g_\e(x)\right)\frac{1}{1+\mathfrak{p}_x} \Lambda_{\tilde{\Pi}}\\
        =&\underbrace{2\im\begin{bmatrix}
 \pa_x & 0 \\
0 & 0
\end{bmatrix}}_{=:\mathcal{B}^s_{11}}+\underbrace{3\im\e\begin{bmatrix}
-\pa_x\circ \cos(x)-\cos(x)\circ\pa_x & 0 \\
0 & 0
\end{bmatrix}+\cO(\e^2)}_{=:\mathcal{B}^s_{12}},
    \end{aligned}
    \end{equation}
and    
    \begin{align} \label{B s 2}
        \mathcal{B}^s_2:=-\frac{g_{\e}(x)}{(1+\mathfrak{p}_x)^2} \Lambda_{\tilde{\Pi}}=-\Lambda_{\tilde{\Pi}}+\cO(\e), \quad \Lambda_{\tilde{\Pi}}:=\begin{bmatrix}
\mathrm{Id} & 0 \\
0 & 0
\end{bmatrix},
    \end{align}
with $\cO(\e),\, \cO(\e^2)\in\mathcal{L}(Y, X)$.

Using \eqref{43 g+1}-\eqref{46 g-0}, we derive the following expansion:
\begin{equation} \label{Lambda f pm 10}
\begin{aligned}
  \mathcal{B}^s_{11} g^+_1(\mu,\e)&= -\frac{2\im}{\sqrt{\ck}} \vet{\sin(x)}{0}
-\mu\frac{1-\kappa}{2\ck^2\sqrt{\ck}}
\vet{\cos(x)}{0}
-\im\e \frac{4(2-\kappa)}{(1-2\kappa)\sqrt{\ck}}
\begin{bmatrix}
 \sin(2x) \\
0
\end{bmatrix} \\
& \quad
+\im\cO(\e^2)\vet{\mathit{odd}(x)}{0}+ \mathcal{O}(\mu^2, \mu \epsilon),\\
 \mathcal{B}^s_{11}g^-_1(\mu,\e)&=-\frac{2\im}{\sqrt{\ck}} \vet{\cos(x)}{0}
+ \mu\frac{1-\kappa}{2\ck^2\sqrt{\ck}}
\vet{\sin(x)}{0}
-\im\e\frac{4(2-\kappa)}{(1-2\kappa)\sqrt{\ck}}
\begin{bmatrix}
 \cos(2x) \\
0
\end{bmatrix} \\
& \quad
+\im\cO(\e^2)\vet{\mathit{even}(x)}{0}+ \mathcal{O}(\mu^2, \mu \epsilon),\\
 \mathcal{B}^s_{11} g^+_0(\mu,\e)&=-2\im \e\vet{\sin(x)}{0}+\cO(\e^2)\vet{\mathit{odd}(x)}{0}+\cO(\mu\e), \quad 
  \mathcal{B}^s_{11} g^-_0(\mu,\e)=\cO(\mu\e).
\end{aligned}
\end{equation}
Using \eqref{Lambda f pm 10} and  the self-adjoint and reversibility-preserving properties of $\mathcal{B}^s_{11}$, the matrix
$-\kappa\mu\mathsf{B}^{s}_{11}$ representing the action of 
the operator $-\kappa\mu\mathcal{B}^s_{11}$ on the basis 
$\mathcal{G}$ of $\mathcal{V}_{\mu,\epsilon}$ in \eqref{G basis set and g}
admits the expansion
\begin{equation}\label{mathsf B s11} 
\begin{aligned} 
-\kappa\mu\mathsf{B}^{s}_{11} &=
\left(
\begin{NiceArray}{cc|ccc}[code-for-first-col=\scriptstyle]
\mu^2\frac{\kappa(1-\kappa)}{2\ck^3}
  & \im\,\mu\frac{\kappa}{\ck}+\im \,r(\mu\e^2)
  & 0
  & 0
  &  \\
-\im\,\mu\frac{\kappa}{\ck}-\im \,r(\mu\e^2)
  &\mu^2\frac{\kappa(1-\kappa)}{2\ck^3}
  & \im\, r(\mu\e)
  & 0
  &  \\
\hline
0
  & -\im\, r(\mu\e)
  & 0
  & 0
  & \\
0
  & 0
  & 0
  & 0
  &  \\
\end{NiceArray}
\right) + \mathcal{O}(\mu^2 \epsilon,\mu^3).
\end{aligned}
\end{equation}
To determine the matrix representation of $-\kappa\mu\mathcal{B}^{s}_{12}$, we use the expansion
\begin{equation}
\left( -\kappa\mu\mathcal{B}^{s}_{12} g_k^\sigma (\mu, \epsilon),\, g_{k'}^{\sigma'}(\mu, \epsilon) \right)
= \left( -\kappa\mu\mathcal{B}^{s}_{12} g_k^\sigma (0, \epsilon),\, g_{k'}^{\sigma'}(0, \epsilon) \right)
+ \mathcal{O}(\mu^2 \epsilon).
\end{equation}
The matrix entries $\left( -\kappa\mu\mathcal{B}^{s}_{12} g_k^\sigma (0, \epsilon),\, g_{k'}^{\sigma}(0, \epsilon) \right),\, \sigma = \pm,\,k, k' = 0, 1,$ are zero, because they are simultaneously real by \eqref{B are alternatively real or imaginary}, and purely imaginary, being the operator $-\kappa\mu\mathcal{B}^{s}_{12}$ purely imaginary and the basis $\{ g_k^{\pm}(0, \epsilon) \}_{k=0,1}$ real. Hence the matrix $-\kappa\mu\mathsf{B}^{s}_{12}$
 representing the action of 
the operator $-\kappa\mu\mathcal{B}^s_{12}$ on the basis 
$\mathcal{G}$ of $\mathcal{V}_{\mu,\epsilon}$
has the form
\begin{equation}
-\kappa\mu\mathsf{B}^{s}_{12} =
\left(
\begin{NiceArray}{cc|ccc}[code-for-first-col=\scriptstyle]
0 & \mathrm{i} a^{s}_{12} & 0 & \mathrm{i} b^{s}_{12} \\
-\mathrm{i} a^{s}_{12} & 0 & -\mathrm{i} c^{s}_{12} & 0 \\
\hline
0 & \mathrm{i} c^{s}_{12} & 0 & \mathrm{i} d^{s}_{12} \\
-\mathrm{i} b^{s}_{12} & 0 & -\mathrm{i} d^{s}_{12} & 0
\end{NiceArray}
\right)
+ \mathcal{O}(\mu^2 \epsilon),
\end{equation}
where
\begin{equation}
\begin{aligned}
& \left( -\kappa\mu\mathcal{B}^{s}_{12} g_1^{-}(0, \epsilon),\, g_1^{+}(0, \epsilon) \right) := \mathrm{i} a^{s}_{12}, \quad \left( -\kappa\mu\mathcal{B}^{s}_{12} g_1^{-}(0, \epsilon),\, g_0^{+}(0, \epsilon) \right) := \mathrm{i} c^{s}_{12}, \\
& \left( -\kappa\mu\mathcal{B}^{s}_{12} g_0^{-}(0, \epsilon),\, g_1^{+}(0, \epsilon) \right) := \mathrm{i} b^{s}_{12}, \quad \left( -\kappa\mu\mathcal{B}^{s}_{12} g_0^{-}(0, \epsilon),\, g_0^{+}(0, \epsilon) \right) := \mathrm{i} d^{s}_{12},
\end{aligned}
\end{equation}
and $a^{s}_{12}, b^{s}_{12}, c^{s}_{12}, d^{s}_{12}$ are real numbers. 
 As $-\kappa\mu\mathcal{B}^{s}_{12} = \mathcal{O}(\mu \epsilon)$ in $\mathcal{L}(Y, X)$, we deduce that $c^{s}_{12} = r(\mu \epsilon)$. Let us compute the expansion of $a^{s}_{12}$, $b^{s}_{12}$ and $d^{s}_{12}$. In view of \eqref{43 g+1}-\eqref{46 g-0}, $g_1^{\pm}(0, \epsilon) = g_1^{\pm}(0,0) + \mathcal{O}(\epsilon)$, $g_0^{+}(0, \epsilon) = g_0^{+}(0,0) + \mathcal{O}(\epsilon)$, $g_0^{-}(0, \epsilon) = \begin{bmatrix} 0 \\ 1 \end{bmatrix}$. By \eqref{B s 1} we have

\[ 
-\kappa\mu\mathcal{B}^{s}_{12} g_1^{-}(0,0) = 
-\im\mu\e\frac{3\kappa}{2\sqrt{\ck}}\begin{bmatrix}
(1+3\cos(2x)) \\
0
\end{bmatrix}+\cO(\mu\e^2),
\quad
-\kappa\mu\mathcal{B}^{s}_{12} g_0^{-}(0,\e) =
\cO(\mu\e^2),
\]
and then
\begin{align*} 
a^{s}_{12} &=  \left( -\kappa\mu\mathcal{B}_{12}^{s} g_1^{-}(0,0),\, g_1^{+}(0,0) \right) + r(\mu \epsilon^2) = r(\mu \epsilon^2), \quad 
b^{s}_{12} = \cO(\mu\e^2),\quad d^{s}_{12}=\cO(\mu\e^2).
\end{align*}

Next we consider the matrix
$-\kappa\mu^2\mathsf{B}^{s}_2$ representing the action of the operator
$ -\kappa\mu^2\mathcal{B}^{s}_2$ on the basis 
$\mathcal{G}$.
Recalling \eqref{B s 2}, we have $-\kappa\mu^2\mathcal{B}^{s}_2 = \kappa \mu^2\Lambda_{\tilde{\Pi}} + \cO(\mu^2\e) $, hence one has 
\begin{equation}
\left( -\kappa\mu^2\mathcal{B}^{s}_2 g_k^\sigma (\mu, \epsilon),\, g_{k'}^{-\sigma'}(\mu, \epsilon) \right)
= \left( \kappa\mu^2\Lambda_{\tilde{\Pi}} g_k^\sigma (0, \epsilon),\, g_{k'}^{\sigma'}(0, \epsilon) \right)
+ \mathcal{O}(\mu^2\e,\mu^3).
\end{equation}
The matrix entries $\left( \kappa\mu^2\Lambda_{\tilde{\Pi}} g_k^\sigma (0, \epsilon),\, g_{k'}^{-\sigma}(0, \epsilon) \right),\, k, k' = 0, 1,\, \sigma = \pm$ are zero, because they are simultaneously purely imaginary by \eqref{B are alternatively real or imaginary}, and real, being the operator $\mathcal{B}^{s}_2$ real and the basis $\{ g_k^{\pm}(0, \epsilon) \}_{k=0,1}$ real. Hence $-\kappa\mu^2\mathcal{B}^{s}_2$, with respect to the basis $\mathcal{G}$ of $\mathcal{V}_{\mu,\epsilon}$ in \eqref{G basis set and g}, admits the expansion 
\begin{equation}
-\kappa\mu^2\mathsf{B}^{s}_2 =
\left(
\begin{NiceArray}{cc|ccc}[code-for-first-col=\scriptstyle]
a^s_2 & 0 & b^s_2 & 0 \\
0 & c^s_2 & 0 & d^s_2 \\
\hline
b^s_2 & 0 & e^s_2 & 0 \\
0 & d^s_2 & 0 & f^s_2
\end{NiceArray}
\right)
+ \mathcal{O}(\mu^2 \epsilon,\mu^3),
\end{equation}
where
\begin{align*}
a^{s}_2 &= \left( \kappa\mu^2\Lambda_{\tilde{\Pi}} g_1^+ (0, \epsilon),\, g_{1}^{+}(0, \epsilon) \right) =\left( \kappa\mu^2\Lambda_{\tilde{\Pi}} g_1^+(0,0) ,\, g_{1}^{+}(0,0) \right)
+r(\mu^2\e)=\mu^2\frac{\kappa}{2\ck}+r(\mu^2\e), \\
b^{s}_2 &:=\left( \kappa\mu^2\Lambda_{\tilde{\Pi}} g_1^+ (0, \epsilon),\, g_{0}^{+}(0, \epsilon) \right) = \left( \kappa\mu^2\Lambda_{\tilde{\Pi}} g_1^+(0,0) ,\, g_{0}^{+}(0,0) \right)+r(\mu^2\e)=r(\mu^2\e),\\
c^{s}_2 &:= \left( \kappa\mu^2\Lambda_{\tilde{\Pi}} g_1^- (0, \epsilon),\, g_{1}^{-}(0, \epsilon) \right) = \left( \kappa\mu^2\Lambda_{\tilde{\Pi}} g_1^-(0,0) ,\, g_{1}^{-}(0,0) \right)+r(\mu^2\e)=\mu^2\frac{\kappa}{2\ck}+r(\mu^2\e), \\
d^{s}_2 &
:=  \left( \kappa\mu^2\Lambda_{\tilde{\Pi}} g_1^- (0, \epsilon),\, g_{0}^{-}(0, \epsilon) \right)
= \left( \kappa\mu^2\Lambda_{\tilde{\Pi}} g_1^-(0,0) ,\, g_{0}^{-}(0,0) \right)+r(\mu^2\e)=r(\mu^2\e),\\
e^{s}_2 &:=\left( \kappa\mu^2\Lambda_{\tilde{\Pi}} g_0^+ (0, \epsilon),\, g_{0}^{+}(0, \epsilon) \right) =\left( \kappa\mu^2\Lambda_{\tilde{\Pi}} g_0^+(0,0),\, g_{0}^{+}(0,0) \right)+r(\mu^2\e)=\kappa\mu^2+r(\mu^2\e),\\
f^{s}_2 &:=\left( \kappa\mu^2\Lambda_{\tilde{\Pi}} g_0^- (0, \epsilon),\, g_{0}^{-}(0, \epsilon) \right) =\left( \kappa\mu^2\Lambda_{\tilde{\Pi}} g_0^-(0,0) ,\, g_{0}^{-}(0,0) \right)+r(\mu^2\e)=r(\mu^2\e).
\end{align*}
This verifies \eqref{mathsf B s}. 
\end{proof}
Lemma \ref{Expansion of B epsilon}-Lemma \ref{Expansion of B s} imply \eqref{B mu epsilon} where the matrix $E$ has the form 
\begin{equation} \label{E0} 
\begin{aligned} 
E &:= \begin{pmatrix}
E_{11} &
\im E_{12} \\
-\im E_{12} &
E_{22}
\end{pmatrix}=
\begin{pmatrix}
\mathsf{e}_{11} \epsilon^2(1+r'(\e,\mu\e^2))- \mathsf{e}_{22}\frac{\mu^2}{8}(1+r''(\e,\mu))  &
\im\left( \frac{1}{2}\mathsf{e}_{12} \mu + r_2(\mu \epsilon^2,\mu^2 \epsilon, \mu^3)\right) \\
- \im\left( \frac{1}{2}\mathsf{e}_{12} \mu + r_2(\mu \epsilon^2,\mu^2 \epsilon, \mu^3)\right) &
-\mathsf{e}_{22}\frac{\mu^2}{8}(1+r_5(\e,\mu))
\end{pmatrix}, 
\end{aligned}
\end{equation}
where $\mathsf{e}_{11}$, $\mathsf{e}_{12}$, and $\mathsf{e}_{22}$ are defined in \eqref{e11e22e12}. Moreover,
\begin{equation} \label{F0} 
\begin{aligned} 
F &:= 
\begin{pmatrix}
F_{11} &
\im F_{12} \\
\im F_{21} &
F_{22}
\end{pmatrix}=
\begin{pmatrix}
r_3(\epsilon^3, \mu \epsilon^2,\mu^2 \epsilon, \mu^3) &
\im \sqrt{\ck}\,\mu \epsilon  + \im\, r_4(\mu \epsilon^2,\mu^2 \epsilon, \mu^3) \\
\im\, r_6(\mu \epsilon,\mu^3) & r_7(\mu^2\e,\mu^3)

\end{pmatrix}, 
\end{aligned}
\end{equation}

\begin{equation}\label{G0} 
\begin{aligned} 
G &:= 
\begin{pmatrix}
G_{11} &
\im G_{12} \\
-\im G_{12} &
G_{22}
\end{pmatrix}=
\begin{pmatrix}
1 +\kappa\mu^2+ r_8(\epsilon^3, \mu \epsilon^2,\mu^2 \epsilon, \mu^3) &
\im r_9(\mu \epsilon^2,\mu^2 \epsilon, \mu^3) \\
- \im r_9(\mu \epsilon^2,\mu^2 \epsilon, \mu^3) &
\mu +r_{10}(\mu^2 \epsilon, \mu^3)
\end{pmatrix}. 
\end{aligned}
\end{equation}

\section{Block-Decoupling}\label{sec:BD}
In this section we block-decouple the $4 \times 4$ Hamiltonian matrix $\mathsf{L}_{\mu,\epsilon} = \mathsf{J}_4 \mathsf{B}_{\mu,\epsilon}$, where $\mathsf{B}_{\mu,\epsilon}$ is obtained in Lemma \ref{B decomposition}.
The purpose of the first step is to make the coupling
vanish at $\mu=0$.
Indeed, by \eqref{F0}, $F_{11}(0,\epsilon)=\mathcal{O}(\epsilon^3)$,
whereas the other entries of $F$ vanish at $\mu=0$.
Since $G_{11}(0,0)=1$, the entry $F_{11}$ can be eliminated
by an analytic symplectic, reversibility-preserving
transformation with parameter $m=-F_{11}/G_{11}$.
The resulting coupling $F^{(1)}$ has every entry divisible
by $\mu$, with an analytic quotient, while the leading-order
structure of the diagonal blocks is preserved.
This divisibility compensates for the possible $1/\mu$
singularity of the Sylvester inverse in the second step
and provides the remainder estimates used in the final
block-decoupling.
\subsection{First Step of Block-Decoupling} 
\begin{lemma}  Conjugating the Hamiltonian and reversible matrix $\mathsf{L}_{\mu,\epsilon} = \mathsf{J}_4 \mathsf{B}_{\mu,\epsilon}$ obtained in Lemma \ref{B decomposition} through the symplectic and reversibility-preserving $4 \times 4$ matrix
\begin{align}
    &\tilde{Y} := \mathrm{Id}_4+m\begin{pmatrix} 0 & -P \\ Q & 0 \end{pmatrix} 
\quad \text{with} \quad 
Q := \begin{pmatrix} 1 & 0 \\ 0 & 0 \end{pmatrix},\quad P := \begin{pmatrix} 0 & 0 \\ 0 & 1 \end{pmatrix},\quad m:=m(\mu,\e):=-\frac{F_{11}(\mu,\e)}{G_{11}(\mu,\e)},
\end{align}
where $m=r(\e^3,\mu\e^2,\mu^2\e,\mu^3)$ is a real number, we obtain the Hamiltonian and reversible matrix
\begin{align} \label{L1mue YLY}
    \mathsf{L}_{\mu,\epsilon}^{(1)} := \tilde{Y}^{-1} \mathsf{L}_{\mu,\epsilon} \tilde{Y} 
= \mathsf{J}_4 \mathsf{B}_{\mu, \epsilon}^{(1)} 
= \begin{pmatrix}
\mathsf{J}_2 E^{(1)} & \mathsf{J}_2 F^{(1)} \\
\mathsf{J}_2 [F^{(1)}]^* & \mathsf{J}_2 G^{(1)}
\end{pmatrix},
\end{align}
where $\mathsf{B}_{\mu, \epsilon}^{(1)}$ is a self-adjoint and reversibility-preserving $4 \times 4$ matrix
\begin{align}
    \mathsf{B}_{\mu, \epsilon}^{(1)} = 
\begin{pmatrix}
E^{(1)} & F^{(1)} \\
[F^{(1)}]^* & G^{(1)}
\end{pmatrix}, 
\quad 
E^{(1)} = [E^{(1)}]^*, \quad G^{(1)} = [G^{(1)}]^*,
\end{align}
where the $2 \times 2$ reversibility-preserving matrices $E^{(1)}$, $F^{(1)}$, and $G^{(1)}$ have the following expansion
\begin{align} \label{E1}
E^{(1)} = 
\begin{pmatrix}
\mathsf{e}_{11} \epsilon^2(1+r'(\e,\mu\e^2))- \mathsf{e}_{22}\frac{\mu^2}{8}(1+r''(\e,\mu))  &
\im\left( \frac{1}{2}\mathsf{e}_{12} \mu + r_2(\mu \epsilon^2,\mu^2 \epsilon, \mu^3)\right) \\
- \im\left( \frac{1}{2}\mathsf{e}_{12} \mu + r_2(\mu \epsilon^2,\mu^2 \epsilon, \mu^3)\right) &
-\mathsf{e}_{22}\frac{\mu^2}{8}(1+r_5(\e,\mu))
\end{pmatrix},
\end{align},
\begin{equation} \label{F1}
\begin{aligned} 
F^{(1)} = 
\begin{pmatrix}
0 &
\im \sqrt{\ck}\mu \e+\im r_4(\mu\e^2,\mu^2\e,\mu^3) \\
\im r_6(\mu\e,\mu^3) & r_7(\mu^2\e,\mu^3)
\end{pmatrix},
\end{aligned}
\end{equation}

\begin{equation} \label{G1}
\begin{aligned} 
G^{(1)} = \begin{pmatrix}
1 +\kappa\mu^2+ r_8(\epsilon^3, \mu \epsilon^2,\mu^2 \epsilon, \mu^3) &
\im r_9(\mu \epsilon^2,\mu^2 \epsilon, \mu^3) \\
- \im r_9(\mu \epsilon^2,\mu^2 \epsilon, \mu^3) &
\mu +r_{10}(\mu^2 \epsilon, \mu^3)
\end{pmatrix},
\end{aligned}
\end{equation}
where $\mathsf{e}_{11}$, $\mathsf{e}_{22}$, and $\mathsf{e}_{12}$ are defined in \eqref{e11e22e12}.

\end{lemma} 
\begin{proof}
    Recalling $\mathsf{B}_{\mu,\e}$ defined in \eqref{B mu epsilon} and using the fact that $\tilde{Y}$ is symplectic (cf. \eqref{YstarJ4Y=J4}) and reversibility preserving (cf. \eqref{reversibility preserving}), one may find that
    \begin{align}
        \mathsf{L}_{\mu,\epsilon}^{(1)} := \tilde{Y}^{-1} \mathsf{L}_{\mu,\epsilon} \tilde{Y}= \tilde{Y}^{-1}\mathsf{J}_4[\tilde{Y}^{-1}]^*\tilde{Y}^{*}\mathsf{B}_{\mu,\e}\tilde{Y}=\mathsf{J}_4 \tilde{Y}^{*}\mathsf{B}_{\mu,\e}\tilde{Y}.
    \end{align}
A straightforward calculation reveals that
\begin{align}
    \mathsf{B}_{\mu, \epsilon}^{(1)}:=\tilde{Y}^{*}\mathsf{B}_{\mu,\e}\tilde{Y}=\begin{pmatrix}
E & F \\
F^* & G
\end{pmatrix}+m\begin{pmatrix}
FQ+QF^* & QG-EP\\
GQ-PE & -PF-F^*P
\end{pmatrix}+m^2 \begin{pmatrix}
QGQ & -QF^*P \\
-PFQ & PEP
\end{pmatrix},
\end{align}
or, equivalently,
\begin{align}
E^{(1)}:=E+\begin{pmatrix}
2mF_{11}+m^2 G_{11} & -\im m F_{21} \\
\im m F_{21}& 0
\end{pmatrix},\quad
G^{(1)}:=G+\begin{pmatrix}
0 & \im m F_{21} \\
-\im m F_{21}& -2m F_{22}+m^2 E_{22}
\end{pmatrix},
\end{align}
and 
\begin{align}
F^{(1)}:=\begin{pmatrix}
0 & \im\left(F_{12}+m G_{12}-m E_{12}+m^2 F_{21} \right) \\
\im  F_{21}& F_{22}-m E_{22}
\end{pmatrix}.
\end{align}
The matrices $E^{(1)}$ and $G^{(1)}$ are self-adjoint, since $m$ is real. In conclusion, recalling \eqref{E0}, \eqref{F0}, and \eqref{G0}, all the formulas in \eqref{E1}, \eqref{F1}, and \eqref{G1} are proved. 
\end{proof}

\subsection{Second Step of Block-Decoupling} 

We state the main result of this section.
\begin{lemma} \label{Step of block-decoupling}
     There exists a $2 \times 2$ reversibility-preserving matrix $\tilde{X}$, analytic in $(\mu, \epsilon)$, of the form
\begin{equation} \label{XX}
    \begin{aligned}
    \tilde{X} :=& \begin{pmatrix}
x_{11} & \im \,x_{12} \\
\im \,x_{21} & x_{22}
\end{pmatrix}=\begin{pmatrix}
r_{11}(\mu^2,\mu \epsilon) & \im\, r_{12}(\mu^3,\mu\epsilon) \\
\im \, r_{21}(\epsilon, \mu^2) &  r_{22}(\mu^3,\mu\e)
\end{pmatrix}, 
\end{aligned}
\end{equation}
where $x_{ij} \in \mathbb{R}, \ i, j = 1, 2$. By conjugating the Hamiltonian and reversible matrix $\mathsf{L}^{(1)}_{\mu, \epsilon}$ defined in \eqref{L1mue YLY}, with the symplectic and reversibility-preserving $4 \times 4$ matrix
\begin{align} \label{S1}
    \exp\left(S^{(1)}\right), \quad \text{where } \quad S^{(1)} := \mathsf{J}_4 \begin{pmatrix}
0 & M \\
M^* & 0
\end{pmatrix}, \quad M := \mathsf{J}_2 \tilde{X}, 
\end{align}
we get the Hamiltonian and reversible matrix
\begin{align} \label{L2mueps}
\mathsf{L}^{(2)}_{\mu, \epsilon} := \exp\left(S^{(1)}\right) \mathsf{L}^{(1)}_{\mu, \epsilon} \exp\left(-S^{(1)}\right) = \mathsf{J}_4 \mathsf{B}^{(2)}_{\mu, \epsilon} = \mathsf{J}_4 \begin{pmatrix}
E^{(2)} & F^{(2)} \\
 [F^{(2)}]^* & G^{(2)}
\end{pmatrix}, 
\end{align}
where the $2\times 2$ self-adjoint and reversibility-preserving matrices $E^{(2)}$ and $G^{(2)}$ have the same expansion of $E^{(1)}$ and $G^{(1)}$, given in \eqref{E}, \eqref{G 2nd}, 
and
\begin{align} \label{F2}
F^{(2)} = \mu^2\begin{pmatrix}
r_3(\epsilon^3,\mu\e^2,\mu^3\e,\mu^5) & \im r_4(\epsilon^3,\mu^2\e^2,\mu^4\e,\mu^6) \\
\im r_6(\epsilon^3,\mu^2\e^2,\mu^3\e,\mu^5) & r_7(\mu\epsilon^3,\mu^2\e^2,\mu^4\e,\mu^6)
\end{pmatrix}. 
\end{align}
\end{lemma}
The rest of the section is devoted to the proof of Lemma \ref{Step of block-decoupling}. For simplicity let $S = S^{(1)}$. The matrix $\exp(S)$ is symplectic and reversibility-preserving because the matrix $S$ in \eqref{S1} is Hamiltonian and reversibility-preserving, cf. \cite[Lemma 3.8]{BMV1}. Note that $S$ is reversibility-preserving, since $\tilde{X}$ has the form \eqref{XX}.

We now expand in Lie series the Hamiltonian and reversible matrix $
\mathsf{L}_{\mu,\epsilon}^{(2)} = \exp(S) \mathsf{L}_{\mu,\epsilon}^{(1)} \exp(-S)$.
We split $\mathsf{L}^{(1)}_{\mu,\epsilon}$ into its $2 \times 2$-diagonal and off-diagonal Hamiltonian and reversible matrices
\begin{equation} \label{L=D+R}
    \begin{aligned}
        &\mathsf{L}^{(1)}_{\mu,\epsilon} = D^{(1)} + R^{(1)},\\
&D^{(1)} := \begin{pmatrix} D_1 & 0 \\ 0 & D_0 \end{pmatrix}= \begin{pmatrix} \mathsf{J}_2 E^{(1)} & 0 \\ 0 & \mathsf{J}_2 G^{(1)} \end{pmatrix}, \quad 
R^{(1)} := \begin{pmatrix} 0 & \mathsf{J}_2 F^{(1)} \\ \mathsf{J}_2 [F^{(1)}]^* & 0 \end{pmatrix}, 
    \end{aligned}
\end{equation}
and perform the Lie expansion
\begin{equation} \label{L2 mu e}
    \begin{aligned}
        \mathsf{L}^{(2)}_{\mu,\epsilon} = \exp(S) \mathsf{L}^{(1)}_{\mu,\epsilon} \exp(-S) &= D^{(1)} + [S, D^{(1)}] + \frac{1}{2} [S, [S, D^{(1)}]] + R^{(1)} + [S, R^{(1)}]\\
        &+ \frac{1}{2} \int_0^1 (1 - \tau)^2 \exp(\tau S) \mathrm{ad}_S^3(D^{(1)}) \exp(-\tau S) \, d\tau\\
        &+ \int_0^1 (1 - \tau) \exp(\tau S) \mathrm{ad}_S^2(R^{(1)}) \exp(-\tau S) \, d\tau.
    \end{aligned}
\end{equation}

We look for a $4 \times 4$ matrix $S$ as in \eqref{S1} (which is Hamiltonian, reversibility-preserving and off-diagonal as the term $R^{(1)}$ we wish to eliminate) that solves the homological equation 
\[
R^{(1)} + [S, D^{(1)}] = 0,
\]
which, recalling \eqref{L=D+R}, reads
\begin{align} \label{before DX-XD=-JF}
\begin{pmatrix}
0 & \mathsf{J}_2 F^{(1)} + \mathsf{J}_2 M D_0 - D_1 \mathsf{J}_2 M \\
\mathsf{J}_2 [F^{(1)}]^* + \mathsf{J}_2 M^* D_1 - D_0 \mathsf{J}_2 M^* & 0
\end{pmatrix} = 0. 
\end{align}

Note that the equation $\mathsf{J}_2 F^{(1)} + \mathsf{J}_2 M D_0 - D_1 \mathsf{J}_2 M = 0$ implies also $\mathsf{J}_2 [F^{(1)}]^* + \mathsf{J}_2 M^* D_1 - D_0 \mathsf{J}_2 M^* = 0$ and vice versa, since $\mathsf{J}^*_2=-\mathsf{J}_2$ and $\mathsf{J}_2(\mathsf{J}_2 F^{(1)} + \mathsf{J}_2 M D_0 - D_1 \mathsf{J}_2 M)=(\mathsf{J}_2 [F^{(1)}]^* + \mathsf{J}_2 M^* D_1 - D_0 \mathsf{J}_2 M^*)^*\mathsf{J}^*_2$. Thus, writing 
\[
M = \mathsf{J}_2 \tilde{X}, \quad \text{or, equivalently,} \quad \tilde{X} = -\mathsf{J}_2 M,
\]
the equation \eqref{before DX-XD=-JF} amounts to solving the Sylvester equation
\begin{align} \label{DX-XD=JF}
D_1 \tilde{X} - \tilde{X} D_0 = -\mathsf{J}_2 F^{(1)}. 
\end{align}
We write the matrices $E^{(1)}$, $F^{(1)}$, $G^{(1)}$ in \eqref{L1mue YLY} as
\begin{align} \label{E1 F1 G1}
    E^{(1)} = \begin{pmatrix} E^{(1)}_{11} & \im E^{(1)}_{12} \\ -\im E^{(1)}_{12} & E^{(1)}_{22} \end{pmatrix}, \quad 
F^{(1)} = \begin{pmatrix} F^{(1)}_{11} & \im F^{(1)}_{12} \\ \im F^{(1)}_{21} & F^{(1)}_{22} \end{pmatrix}, \quad G^{(1)} = \begin{pmatrix} G^{(1)}_{11} & \im G^{(1)}_{12} \\ -\im G^{(1)}_{12} & G^{(1)}_{22} \end{pmatrix},
\end{align}
where the real numbers $E_{ij}^{(1)}, F_{ij}^{(1)}, G_{ij}^{(1)}, i, j = 1, 2,$ have the expansion in \eqref{E1}, \eqref{F1}, \eqref{G1}. Thus, by \eqref{L=D+R}, \eqref{XX} and \eqref{E1 F1 G1}, the equation \eqref{DX-XD=JF} amounts to solving the $4 \times 4$ real linear system

\begin{align} \label{Ax=f}
\underbrace{
\begin{pmatrix}
a & b & c & 0 \\
d & a & 0 & -c \\
e & 0 & a & -b \\
0 & -e & -d & a
\end{pmatrix}
}_{=: \mathsf{A}}
\underbrace{
\begin{pmatrix}
x_{11} \\ x_{12} \\ x_{21} \\ x_{22}
\end{pmatrix}
}_{=: \vec{x}}
=
\underbrace{
\begin{pmatrix}
- F_{21}^{(1)} \\ F_{22}^{(1)} \\ -F_{11}^{(1)} \\ F_{12}^{(1)}
\end{pmatrix}
}_{=: \vec{f}},
\end{align}
where by \eqref{E1}, \eqref{F1}, and \eqref{G1} 
\begin{equation} \label{a b c d e}
\begin{aligned}
a &= G^{(1)}_{12} - E^{(1)}_{12} = - \frac{1}{2}\mathsf{e}_{12}\mu  + r(\mu\epsilon^2, \mu^2 \epsilon, \mu^3), 
\quad b = G^{(1)}_{11} = 1+\kappa\mu^2 + r(\e^3,\mu\e^2,\mu^2 \epsilon, \mu^3), \\
c &= E^{(1)}_{22} =  -\frac{1}{8}\mathsf{e}_{22}\mu^2+  r(\mu^2\epsilon, \mu^3), 
\quad d = G^{(1)}_{22} = \mu  + r(\mu^2 \epsilon, \mu^3), \quad e = E^{(1)}_{11} =r(\mu^2,\e^2).
\end{aligned}
\end{equation}
Using also  \cite[Lemma 5.5]{BMV3}, one has 
\begin{equation} \label{det A}
\begin{aligned}
\det \mathsf{A} 
= (bd - a^2)^2 - 2ce \Big( a^2 + bd - \tfrac{1}{2} ce \Big) = \mu^2\left(1+r(\e^3,\mu)\right),
\end{aligned}    
\end{equation}
which is not zero provided  $\mu \neq 0$ .
The inverse of $\mathsf{A} $ is computed in \cite[formula (5.23)]{BMV3} as 
\begin{align} \label{A inverse}
\mathsf{A}^{-1} = \frac{1}{\det\mathsf{A} }
\begin{pmatrix}
a (a^2 - bd - ce) & b(-a^2 + bd - ce) & -c(a^2 + bd - ce) & -2abc \\[6pt]
d(-a^2 + bd - ce) & a(a^2 - bd - ce) & 2acd & -c(-a^2 - bd + ce) \\[6pt]
-e(a^2 + bd - ce) & 2abe & a(a^2 - bd - ce) & b(a^2 - bd + ce) \\[6pt]
-2ade & -e(-a^2 - bd + ce) & d(a^2 - bd + ce) & a(a^2 - bd - ce)
\end{pmatrix}
\end{align}
and expands as 
\begin{align} \label{mathsf A -1}
\mathsf{A}^{-1} = 
\frac{1}{\mu}\left(
\begin{array}{cccc}
\frac{\mu}{2}\mathsf{e}_{12}(1+r(\mu,\e)) & 1+r(\mu,\e) & \frac{\mu^2}{8}\mathsf{e}_{22}(1+r(\mu,\e))&-\frac{\mu^2}{8}\mathsf{e}_{12}\mathsf{e}_{22}(1+r(\mu,\e)) \\
\mu(1+r(\mu,\e)) & \frac{\mu}{2}\mathsf{e}_{12}(1+r(\mu,\e))  & \textcolor{black}{\frac{\mu^3}{8}\mathsf{e}_{12}}\mathsf{e}_{22}(1+r(\mu,\e)) & \textcolor{black}{-\frac{\mu^2}{8}}\mathsf{e}_{22}(1+r(\mu,\e)) \\
r(\epsilon^2, \mu^2) & r(\epsilon^2, \mu^2) &\frac{\mu}{2}\mathsf{e}_{12}(1+r(\mu,\e)) & -1+r(\mu,\e) \\
 r(\mu\epsilon^2, \mu^3) & r(\epsilon^2, \mu^2) & -\mu(1+r(\mu,\e)) &\frac{\mu}{2}\mathsf{e}_{12}(1+r(\mu,\e)) \\
\end{array}
\right).
\end{align}
Therefore, for any $\mu \neq 0$, there exists a unique solution $\vec{x} = \mathsf{A}^{-1} \vec{f}$ of the linear system \eqref{Ax=f}, namely a unique matrix $\tilde{X}$ which solves the Sylvester equation \eqref{DX-XD=JF}.
Explicitly, 
    by \eqref{Ax=f}, \eqref{mathsf A -1}, \eqref{E1 F1 G1}, \eqref{F1} we obtain, for any $\mu \neq 0$,
    \small
\begin{equation}
\begin{aligned}
\begin{pmatrix}
x_{11} \\ x_{12} \\ x_{21} \\ x_{22}
\end{pmatrix}
&=
\left(
\begin{array}{c}
r(\mu\epsilon,\mu^2) \\[6pt]
r(\mu\epsilon,\mu^3) \\[6pt]
 r(\e,\mu^2) \\[6pt]
 r(\mu\e,\mu^3)
\end{array}
\right),
\end{aligned}
\end{equation}
\normalsize
which proves \eqref{XX}. By \eqref{F1}, each component of $\vec f$
is divisible by $\mu$ with an analytic quotient,
which cancels the factor $1/\mu$ in \eqref{mathsf A -1}. In particular, each $x_{ij}$ admits an analytic extension at $\mu = 0$.  
Note that, for $\mu = 0$, the resulting matrix $\tilde{X}$ still solves \eqref{DX-XD=JF}.

Since the matrix $S$ solves the homological equation $R^{(1)}+[S, D^{(1)}]  = 0$, identity \eqref{L2 mu e} simplifies to
\begin{align} \label{L2 mu epsilon}
\mathsf{L}_{\mu,\epsilon}^{(2)} = D^{(1)} + \frac{1}{2} [ S, R^{(1)} ]
+ \frac{1}{2} \int_0^1 (1 - \tau^2)\, \exp(\tau S) \, \mathrm{ad}_S^2 (R^{(1)}) \, \exp(-\tau S) \, d\tau.
\end{align}
This follows by using the Jacobi identity and the homological relation to cancel lower-order terms and to rewrite
\[
\frac{1}{2} [S, [S, D^{(1)}]] + [S, R^{(1)}] = \frac{1}{2} [S, R^{(1)}],
\quad
\operatorname{ad}_S^3 (D^{(1)}) = -\operatorname{ad}_S^2 (R^{(1)}).
\]
The matrix $\frac{1}{2}[S, R^{(1)}]$ is, by \eqref{S1}, \eqref{L=D+R}, the block-diagonal Hamiltonian and reversible matrix
\begin{align} \label{0.5 S R}
\frac{1}{2} [ S, R^{(1)} ]
=
\begin{pmatrix}
\frac{1}{2} \mathsf{J}_2 ( 
M \mathsf{J}_2 [F^{(1)}]^* - F^{(1)} \mathsf{J}_2 M^* ) & 0 \\
0 & \frac{1}{2} \mathsf{J}_2 ( M^* \mathsf{J}_2 F^{(1)} - [F^{(1)}]^* \mathsf{J}_2 M )
\end{pmatrix}
=
\begin{pmatrix}
\mathsf{J}_2 \widetilde{E} & 0 \\
0 & \mathsf{J}_2 \widetilde{G}
\end{pmatrix},
\end{align}
where, recalling \eqref{S1},
\begin{align} \label{E G tilde}
\widetilde{E} := \mathrm{Sym} ( \mathsf{J}_2 \tilde{X} \mathsf{J}_2 [F^{(1)}]^* )=\begin{pmatrix} \widetilde{E}^{(1)}_{11} & \im \widetilde{E}^{(1)}_{12} \\ -\im \widetilde{E}^{(1)}_{12} & \widetilde{E}^{(1)}_{22} \end{pmatrix}, 
\quad
\widetilde{G} := \mathrm{Sym} ( \tilde{X}^* F^{(1)} )=\begin{pmatrix} \widetilde{G}^{(1)}_{11} & \im \widetilde{G}^{(1)}_{12} \\ -\im \widetilde{G}^{(1)}_{12} & \widetilde{G}^{(1)}_{22} \end{pmatrix},
\end{align}
denoting $\mathrm{Sym}(A) := \frac{1}{2}( A + A^* )$.
\begin{lemma} \label{EG e11 tilde lemma}
    The self-adjoint and reversibility-preserving matrices $\widetilde{E}$, $\widetilde{G}$ in \eqref{E G tilde} have the form
    \begin{equation}\label{tilde EG and tilde e11}
    \begin{aligned} 
\widetilde{E} &= \begin{pmatrix}
\tilde{r}_1 (\mu \epsilon^2, \mu^3 \epsilon,\mu^5) & \im\, \tilde{r}_2 (\mu ^2\epsilon^2,\mu^3\e,\mu^5) \\
- \im \,\tilde{r}_2 (\mu ^2\epsilon^2,\mu^3\e,\mu^5) & \tilde{r}_5(\mu ^2\epsilon^2,\mu^4\e,\mu^5)
\end{pmatrix}, \quad
\widetilde{G} = \begin{pmatrix}
\tilde{r}_8 (\mu \epsilon^2, \mu^3 \epsilon,\mu^5) & \im\, \tilde{r}_9 (\mu^2 \epsilon^2, \mu^3 \epsilon,\mu^5) \\
- \im\, \tilde{r}_9 (\mu^2 \epsilon^2, \mu^3 \epsilon,\mu^5) & \tilde{r}_{10} (\mu^2 \epsilon^2, \mu^4 \epsilon,\mu^6)
\end{pmatrix}, \\
\end{aligned}
 \end{equation}
\end{lemma}
\begin{proof}
    By \eqref{F1}, \eqref{XX}, a straightforward calculation reveals that
\begin{equation*} 
\begin{aligned}
\mathsf{J}_2 \tilde{X} \mathsf{J}_2 [F^{(1)}]^* 
&= \begin{pmatrix}
x_{21} F^{(1)}_{12} - x_{22} F^{(1)}_{11} & \im ( x_{21} F^{(1)}_{22} + x_{22} F^{(1)}_{21} ) \\
\im ( x_{11} F^{(1)}_{12} + x_{12} F^{(1)}_{11} ) & - x_{11} F^{(1)}_{22} + x_{12} F^{(1)}_{21}
\end{pmatrix} 
= \begin{pmatrix}
{r}_1 (\mu \epsilon^2, \mu^3 \epsilon,\mu^5) & \im\, {r}_2 (\mu ^2\epsilon^2,\mu^3\e,\mu^5) \\
- \im \,{r}_2 (\mu ^2\epsilon^2,\mu^3\e,\mu^5) & {r}_5(\mu ^2\epsilon^2,\mu^4\e,\mu^5)
\end{pmatrix},
\end{aligned}
\end{equation*}
and
 \begin{equation*} 
\begin{aligned}
 \tilde{X} ^* F^{(1)}
&= \begin{pmatrix}
x_{11} F^{(1)}_{11} + x_{21} F^{(1)}_{21} & \im ( x_{11} F^{(1)}_{12} - x_{21} F^{(1)}_{22} ) \\
\im ( -x_{12} F^{(1)}_{11} + x_{22} F^{(1)}_{21} ) &  x_{12} F^{(1)}_{12} + x_{22} F^{(1)}_{22}
\end{pmatrix} 
= \begin{pmatrix}
{r}_1 (\mu \epsilon^2, \mu^3 \epsilon,\mu^5) & \im\, {r}_2 (\mu ^2\epsilon^2,\mu^3\e,\mu^5) \\
- \im \,{r}_2 (\mu ^2\epsilon^2,\mu^4\e,\mu^6) & {r}_5(\mu ^2\epsilon^2,\mu^4\e,\mu^6)
\end{pmatrix}.
\end{aligned}
\end{equation*}
By symmetry, we verify \eqref{tilde EG and tilde e11}.
\end{proof}
 
 The last term in \eqref{L2 mu epsilon} is small. 

\begin{lemma} \label{high order terms lemma}
    The $4 \times 4$ Hamiltonian and reversibility matrix
\begin{align} \label{high order terms EG hat}
\frac{1}{2} \int_0^1 (1 - \tau^2) \exp(\tau S) \mathrm{ad}_S^2 (R^{(1)}) \exp(-\tau S) \, d\tau
= \begin{pmatrix}
\mathsf{J}_2 \widehat{E} & \mathsf{J}_2 F^{(2)} \\
\mathsf{J}_2 [F^{(2)}]^* & \mathsf{J}_2 \widehat{G}
\end{pmatrix}.
\end{align}
where the $2 \times 2$ self-adjoint and reversible matrices $\widehat{E}$, $\widehat{G}$ have entries
\begin{align} \label{Eij Gij}
\widehat{E}_{ij} = \widehat{G}_{ij} = \mu^2 r((\e^3,\mu \epsilon^2, \mu^3 \epsilon,\mu^5)), 
\quad i,j = 1,2,
\end{align}
and the $2 \times 2$ reversible matrix $F^{(2)}$ admits an expansion as in \eqref{F2}.
\end{lemma}
\begin{proof}
  Each $\exp(\tau S) \operatorname{ad}_S^2 (R^{(1)}) \exp(-\tau S)$ is Hamiltonian and reversibility-preserving, and formula \eqref{high order terms EG hat} holds. In order to estimate its entries, we first compute $\operatorname{ad}_S^2 (R^{(1)})$. Using the form of $S$ in \eqref{S1} and $[S, R^{(1)}]$ in \eqref{0.5 S R} one gets
$
\operatorname{ad}_S^2 (R^{(1)}) = 
\begin{pmatrix}
0 & J_2 \widetilde{F} \\
J_2 \widetilde{F}^* & 0
\end{pmatrix}$ where  
$\widetilde{F} := 2 \left( M J_2 \widetilde{G} - \widetilde{E} J_2 M \right)$ 
and $\widetilde{E}$, $\widetilde{G}$ are defined in \eqref{E G tilde}. 
\begin{equation*} 
\begin{aligned}
M J_2 \widetilde{G} 
&= \begin{pmatrix}
x_{21} \widetilde{G}_{12} - x_{22} \widetilde{G}_{11} & \im ( x_{21} \widetilde{G}_{22} - x_{22} \widetilde{G}_{12} ) \\
\im ( x_{11} \widetilde{G}_{12} + x_{12} \widetilde{G}_{11} ) & - x_{11} \widetilde{G}_{22} - x_{12}\widetilde{G}_{12}
\end{pmatrix} 
= \mu^2\begin{pmatrix}
{r} (\e^3,\mu \epsilon^2, \mu^3 \epsilon,\mu^5) & \im\, {r} (\e^3,\mu ^2\epsilon^2, \mu^4 \epsilon,\mu^6) \\
 \im \,{r} (\e^3,\mu^2 \epsilon^2, \mu^3 \epsilon,\mu^5) & {r}(\mu\e^3,\mu^2 \epsilon^2, \mu^4 \epsilon,\mu^6)
\end{pmatrix},
\end{aligned}
\end{equation*}
\begin{equation*} 
\begin{aligned}
 \widetilde{E} J_2 M  
&= \begin{pmatrix}
x_{21} \widetilde{E}_{12} - x_{11} \widetilde{E}_{11} & -\im ( x_{12} \widetilde{E}_{11} + x_{22} \widetilde{E}_{12} ) \\
\im ( x_{11} \widetilde{E}_{12} - x_{21} \widetilde{E}_{22} ) & - x_{12} \widetilde{E}_{12} - x_{22}\widetilde{E}_{22}
\end{pmatrix} 
= \mu^2\begin{pmatrix}
{r}(\e^3,\mu \epsilon^2, \mu^3 \epsilon,\mu^5) & \im\, {r} (\e^3,\mu ^2\epsilon^2, \mu^4 \epsilon,\mu^6) \\
 \im \,{r} (\e^3,\mu^2 \epsilon^2, \mu^3 \epsilon,\mu^5) &{r}(\mu\e^3,\mu^2 \epsilon^2, \mu^4 \epsilon,\mu^6)
\end{pmatrix}.
\end{aligned}
\end{equation*}
Thus the matrix $\widetilde{F}$ has an expansion as in \eqref{F2}. Then, for any $\tau \in [0,1]$, the matrix $\exp(\tau S) \operatorname{ad}_S^2 (R^{(1)}) \exp(-\tau S) = \operatorname{ad}_S^2 (R^{(1)})(1 + \mathcal{O}(\mu, \epsilon))$. In particular the matrix $F^{(2)}$ in \eqref{high order terms EG hat} has the same expansion of $\widetilde{F}$, whereas the matrices $\widehat{E}$, $\widehat{G}$ have entries at least as in \eqref{Eij Gij}.
\end{proof}

\begin{proof} [Proof of Lemma \ref{Step of block-decoupling}]
    It follows by \eqref{L2 mu epsilon}-\eqref{0.5 S R}, \eqref{L=D+R} and Lemmata \ref{EG e11 tilde lemma} and \ref{high order terms lemma}.  
The matrix $E^{(2)} := E^{(1)} + \widetilde{E} + \widehat{E}$ has the expansion in \eqref{E}.  
Similarly, $G^{(2)} := G^{(1)} + \widetilde{G} + \widehat{G}$ has the expansion in \eqref{G 2nd}. 
\end{proof}

\subsection{Complete Block-Decoupling and Proof of the Main Result}

We now block-diagonalize the $4 \times 4$ Hamiltonian and reversible matrix $\mathsf{L}^{(2)}_{\mu,\epsilon}$ in \eqref{L2mueps}. First, we split it into its $2 \times 2$-diagonal and off-diagonal Hamiltonian and reversible matrices
\begin{equation} \label{L2mueps2}
\mathsf{L}^{(2)}_{\mu,\epsilon} = D^{(2)} + R^{(2)}, \qquad 
D^{(2)} := 
\begin{pmatrix}
\mathsf{J}_2 E^{(2)} & 0 \\
0 & \mathsf{J}_2 G^{(2)}
\end{pmatrix}, \qquad
R^{(2)} := 
\begin{pmatrix}
0 & \mathsf{J}_2 F^{(2)} \\
\mathsf{J}_2 [F^{(2)}]^* & 0
\end{pmatrix}.
\end{equation}

\begin{lemma} \label{lemma S2}
There exist a $4 \times 4$ reversibility-preserving Hamiltonian matrix $S^{(2)} := S^{(2)}(\mu,\epsilon)$ of the form \eqref{S1}, analytic in $(\mu, \epsilon)$, of size $\mathcal{O}(\mu\epsilon^3,\mu^2\e^2,\mu^4\e,\mu^6)$, and a $4 \times 4$ block-diagonal reversible Hamiltonian matrix $P := P(\mu, \epsilon)$, analytic in $(\mu, \epsilon)$, of size $\mu^2\mathcal{O}(\e^6,\mu\e^5,\mu^2\epsilon^4,\mu^4\e^3,\mu^6\e^2,\mu^8\e,\mu^{10})$ such that
\begin{equation} \label{D2+P}
\exp(S^{(2)})(D^{(2)} + R^{(2)}) \exp(-S^{(2)}) = D^{(2)} + P.
\end{equation}
\end{lemma}

\begin{proof}
For notational simplicity, we denote $S = S^{(2)}$. The equation \eqref{D2+P} is equivalent to the system
\begin{equation} \label{Pi_D Pi_varnothing}
\begin{aligned}
& \Pi_D \left( \exp(S)(D^{(2)} + R^{(2)}) \exp(-S) \right) - D^{(2)} = P  \ , \qquad 
\Pi_\varnothing \left( \exp(S)(D^{(2)} + R^{(2)}) \exp(-S) \right) = 0 ,
\end{aligned}
\end{equation}
where $\Pi_D$ is the projector onto the block-diagonal matrices and $\Pi_\varnothing$ onto the block-off-diagonal ones. The second equation in \eqref{Pi_D Pi_varnothing} is equivalent, by a Lie expansion, and since $[S, R^{(2)}]$ is block-diagonal, to
\begin{equation} \label{R2+SD2+R}
R^{(2)} + [S, D^{(2)}] + \underbrace{\Pi_\varnothing \int_0^1 (1 - \tau) \exp(\tau S) \mathrm{ad}_S^2 (D^{(2)} + R^{(2)}) \exp(-\tau S) d\tau}_{=:\mathcal{R}(S)} = 0.
\end{equation}
The ``nonlinear homological equation'' \eqref{R2+SD2+R},
\begin{equation} \label{SD=-R-R}
[S, D^{(2)}] = -R^{(2)}- \mathcal{R}(S),
\end{equation}
is equivalent to solving the $4 \times 4$ real linear system
\begin{equation} \label{Ax=f f=munu+mug}
\mathsf{A} \vec{x} = \vec{f}(\mu, \epsilon, \vec{x}), \qquad \vec{f}(\mu, \epsilon, \vec{x}) = \mu \vec{\nu}(\mu, \epsilon) + \mu \vec{g}(\mu, \epsilon, \vec{x})
\end{equation}
associated, as in \eqref{Ax=f}, to \eqref{SD=-R-R}. The vector $\mu \vec{\nu}(\mu, \epsilon)$ is associated with $-R^{(2)}$ where $R^{(2)}$ is in \eqref{L2mueps2}. The vector $\mu \vec{g}(\mu, \epsilon, \vec{x})$ is associated with the matrix $-\mathcal{R}(S)$, which is a Hamiltonian and reversible block-off-diagonal matrix (i.e. of the form \eqref{L=D+R}).
The factor $\mu$ is present in $D^{(2)}$ and $R^{(2)}$, see \eqref{E}, \eqref{G 2nd}, \eqref{F2} and the analytic function $\vec{g}(\mu, \epsilon, \vec{x})$ is quadratic in $\vec{x}$ (for the presence of $\mathrm{ad}_S^2$ in $\mathcal{R}(S)$). In view of \eqref{F2} one has
\begin{equation} \label{munu}
\mu \vec{\nu}(\mu, \epsilon) := (-F^{(2)}_{21}, F^{(2)}_{22}, -F^{(2)}_{11}, F^{(2)}_{12})^\top, \qquad F^{(2)}_{ij} = \mu r(\mu\epsilon^3,\mu^2\e^2,\mu^4\e,\mu^6).
\end{equation}
System \eqref{Ax=f f=munu+mug} is equivalent to $\vec{x} = \mathsf{A}^{-1} \vec{f}(\mu, \epsilon, \vec{x})$ and, writing $\mathsf{A}^{-1} = \frac{1}{\mu} \mathcal{B}(\mu, \epsilon)$ (cf. \eqref{mathsf A -1}), to
\[
\vec{x} = \mathcal{B}(\mu, \epsilon) \vec{\nu}(\mu, \epsilon) + \mathcal{B}(\mu, \epsilon) \vec{g}(\mu, \epsilon, \vec{x}).
\]
By the implicit function theorem this equation admits a unique small solution $\vec{x} = \vec{x}(\mu, \epsilon)$, analytic in $(\mu, \epsilon)$, with size $\mathcal{O}(\mu\epsilon^3,\mu^2\e^2,\mu^4\e,\mu^6)$ as $\vec{\nu}$ in \eqref{munu}. Then the first equation of \eqref{Pi_D Pi_varnothing} gives $P = [S, R^{(2)}] + \Pi_D \int_0^1 (1-\tau) \exp(\tau S) \text{ad}_S^2 (D^{(2)} + R^{(2)}) \exp(-\tau S) d\tau$, and its estimate follows from those of $S$ and $R^{(2)}$ (see \eqref{F2}). 
\end{proof}

\begin{proof} [{Proof of Theorems \ref{Complete BF thm} and \ref{thm:main}.}]
By Lemma \ref{lemma S2} and recalling \eqref{mathcal L= ichmu+mathscr L} the operator $\mathcal{L}_{\mu, \epsilon} : \mathcal{V}_{\mu, \epsilon} \to \mathcal{V}_{\mu, \epsilon}$ is represented by the $4 \times 4$ Hamiltonian and reversible matrix
\[
\im \ck \mu + \exp(S^{(2)}) \mathsf{L}^{(2)}_{\mu, \epsilon} \exp(- S^{(2)}) 
= \im \ck \mu + 
\begin{pmatrix}
\mathsf{J}_2 E^{(3)} & 0 \\
0 & \mathsf{J}_2 G^{(3)}
\end{pmatrix}
=: 
\begin{pmatrix}
\mathsf{U} & 0 \\
0 & \mathsf{S}
\end{pmatrix},
\]
where the matrices $E^{(3)}$ and $G^{(3)}$ expand as in \eqref{E}, \eqref{G 2nd}. Consequently, the matrices $\mathsf{U}$ and $\mathsf{S}$ expand as in \eqref{U S}. Theorem \ref{Complete BF thm} is proved. Theorem \ref{thm:main} is a straightforward corollary.
\end{proof}

\appendix
\section{Expansion of the Stokes waves in deep water}\label{sec:App2}

In this Appendix we provide the expansions  \eqref{exp:Sto}-\eqref{expcoef}, \eqref{expfe}, 
\eqref{pino1fd}-\eqref{aino2fd}.
\\[1mm]
\noindent
{\bf Proof  of \eqref{exp:Sto}-\eqref{expcoef}.}
Writing
 \begin{equation}\label{etapsic}
\begin{aligned}
 & \eta_\e(x) = \e \eta_1(x) + \e^2 \eta_2(x) + \cO(\e^3) \, , \\
 &  \psi_\e(x) = \e \psi_1(x) + \e^2 \psi_2(x) + \cO(\e^3) \, ,  
  \end{aligned}
\qquad \quad c_\e = \ck + \e c_1 + \e^2 c_2+ \cO(\e^3) \, ,  
\end{equation}
where  $\eta_i$ is $\mathit{even}(x)$ and $\psi_i$ is $\mathit{odd}(x)$ for $i=1,2$, 
we solve order by order in $ \e $ the equations \eqref{travelingWWstokes},  
that we rewrite  as
\begin{equation}
\label{Sts}
\begin{cases}
-c \, \psi_x +  \eta   + \dfrac{\psi_x^2}{2} - 
\dfrac{\eta_x^2}{2(1+\eta_x^2)} ( c  -  \psi_x )^2-\kappa\left(\dfrac{\eta_x}{(1+\eta_x^2)^{\frac{1}{2}}}\right)_x  = 0 \\
c \, \eta_x+G(\eta)\psi = 0 	\, ,
\end{cases}
\end{equation}
having substituted $G(\eta)\psi $ with $-c \, \eta_x $  in the first equation.
 We expand the Dirichlet-Neumann operator 
$ G(\eta)=  G_0+ G_1(\eta) + G_2(\eta) + \cO(\eta^3)  $
where, according to \cite[formula (2.14)]{CS}, 
\begin{equation} \label{expDiriNeu}
\begin{aligned}
 G_0 & := |D|  \,, \\
 G_1(\eta) & := D \eta D -  G_0\eta G_0 = D \eta D -  |D|\eta  |D|, \\
 G_2(\eta) & := -\frac12  \Big( G_0 
 {\eta}^2 |D|^2 +|D|^2{\eta}^2 G_0 - 2G_0\eta G_0\eta G_0\Big) = -\frac12  \Big(|D| 
 {\eta}^2 |D|^2 +|D|^2{\eta}^2 |D| - 2|D|\eta |D|\eta |D|\Big)\, .
\end{aligned}
\end{equation}
{\bf First order in $ \e $.}  Substituting the expansions in \eqref{etapsic} into \eqref{Sts}, we get the linear system 
\begin{equation}\label{cB0}
 \left\{\begin{matrix} - \ck (\psi_1)_x + \eta_1 -\kappa (\eta_1)_{xx}= 0 \\
 \ck (\eta_1)_x + G_0\psi_1 =0 \, ,  \end{matrix}\right.
 \quad  \text{i.e.} \, \vet{\eta_1}{\psi_1} \in \text{Ker }\cB_0  \text{ with } \cB_0 := \begin{bmatrix} 1-\kappa\pa^2_{x} & -\ck\pa_x \\ \ck\pa_x & G_0  \end{bmatrix},  
\end{equation}
where $\eta_1$ is $\mathit{even}(x)$ and $\psi_1$ is $\mathit{odd}(x)$.
\begin{lemma}\label{lem:B0}
If $\kappa \in \R_+ \setminus \fR$,  the kernel of 
 the linear operator $\cB_0$ in \eqref{cB0} is one dimensional and given by
\begin{equation}\label{chk}
\text{Ker }\cB_0= \text{span}\,\Big\{\vet{\cos(x)}{\ck\sin(x)} \Big\}.
\end{equation}
\end{lemma}

\begin{proof} The action of $\cB_0$ on each subspace spanned by $\footnotesize{\,\Big\{\vet{\cos(kx)}{0}, \vet{0}{\sin(kx)}\Big\}} $, $k\in \N$, is represented by the $2\times 2$ matrix $\footnotesize{ \begin{bmatrix} 1+\kappa k^2 & -\ck k \\ -\ck k & k  \end{bmatrix}}$. Its determinant is given by
 $$ (1+\kappa k^2)k - \ck^2 k^2\stackrel{\eqref{exp:Sto}}{=} \left(\frac{(1+\kappa k^2)}{k} - (1+\kappa)\right) k^2$$
 and, provided $\kappa \in \R_+ \setminus \fR$, vanishes if and only if $k=1$.
Indeed, for $k$ large enough, the determinant goes asymptotically to $+\infty$ so it is uniformly bounded away from zero,  whereas if for some $k \in \N$ it vanishes, it implies that $\kappa \in \fR_k \subset \fR$, absurd.
  For $k=1$ we obtain the kernel of $\cB_0$ given in \eqref{chk}. For $k=0$ it has no kernel since $\psi_1(x)$ is odd.
\end{proof}
We set
$
\eta_1(x) := \cos(x)$, $\psi_1(x) := \ck \sin(x) 
$
in agreement with 
\eqref{exp:Sto}. 
\\[1mm]
{\bf Second order in $ \e $.} 
By \eqref{Sts}, and since 
$ \ck^2 (\eta_1)_x^2 = (G_0\psi_1)^2  $, we get  the linear system
\begin{equation}\label{syslin2}
 \cB_0 \vet{\eta_2}{\psi_2} = \vet{c_1(\psi_1)_x-\frac12 (\psi_1)_x^2 + \frac12 (G_0\psi_1)^2 }{-c_1(\eta_1)_x - G_1(\eta_1)\psi_1} \, , 
\end{equation}
where  $\cB_0$ is the self-adjoint operator in \eqref{cB0}. 
System \eqref{syslin2} admits a solution if and only if its right-hand term is orthogonal to the Kernel  of $\cB_0$ in \eqref{chk}, namely
\begin{equation}\label{orth1}
 \Big(\vet{c_1(\psi_1)_x-\frac12 (\psi_1)_x^2 + \frac12 (G_0\psi_1)^2 }{-c_1(\eta_1)_x - G_1(\eta_1)\psi_1}\;,\;\vet{\cos(x)}{\ck\sin(x)}\Big)=0 \, . 
\end{equation}
In view of the first order expansion \eqref{exp:Sto} and \eqref{expDiriNeu},
  it results $
 [G_0\psi_1](x)=  \ck\sin(x)$, $\big[G_1(\eta_1)\psi_1\big](x) 
 =0$
so that \eqref{orth1} implies
 $c_1=0$, in agreement with \eqref{exp:Sto}. Equation
 \eqref{syslin2} reduces to 
\begin{align}\label{sisto2}
\begin{bmatrix} 1-\kappa\pa_{xx} & -\ck\pa_x \\ \ck\pa_x & G_0  \end{bmatrix}
\vet{\eta_2}{\psi_2} 
  = \vet{-\frac{1+\kappa}{2}\cos(2x) }{ 0}.
\end{align}
Setting $ \eta_2 = \eta_2^{[2]} \cos(2x) $ and  $ \psi_2 =
 \psi_2^{[2]}  \sin (2x) $, system 
\eqref{sisto2} amounts to 
\begin{align*}
\left\{ \begin{matrix} \big( (1+4\kappa)\eta_2^{[2]} -2\ck  \psi_2^{[2]} \big) \cos(2x)  =  -\frac{1+\kappa}{2}\cos(2x)   \\ (-2\ck \eta_2^{[2]}   + 2  \psi_2^{[2]})\sin(2x)  =   0 , \end{matrix}\right.
\end{align*}
which leads to the expansions of $ \eta_2^{[2]} $, $ \psi_2^{[2]} $ 
given in
\eqref{expcoef}.

\noindent
{\bf Third order in $ \e $.} 
It remains to determine 
$ c_2 $ in \eqref{expcoef}.
 We get the linear system 
\begin{equation}\label{syslin3}
 \cB_0 \vet{\eta_3}{\psi_3} = \vet{c_2(\psi_1)_x 
 - (\psi_1)_x (\psi_2)_x - (\eta_1)_x^2 (\psi_1)_x \ck + 
 (\eta_1)_x (\eta_2)_x \ck^2-\frac{3}{2}\kappa (\eta_1)^2_x (\eta_1)_{xx}
 }{-c_2(\eta_1)_x - G_1(\eta_1)\psi_2- G_1(\eta_2)\psi_1 - G_2(\eta_1)\psi_1} \, . 
\end{equation}
System \eqref{syslin3} has 
a solution if and only if the right hand side is orthogonal to the Kernel of
$ \cB_0 $ given in \eqref{chk}. This condition determines uniquely $ c_2 $.
Denoting  $\Pi_1$ the $L^2$-orthogonal projector on span$\, \{\cos(x),\sin(x)\} $, it results 
\begin{align*} 
& c_2 (\psi_1)_x = c_2\ck\cos(x)\, , \quad c_2 (\eta_1)_x = -c_2 \sin(x) \, , \quad \Pi_1[ (\psi_1)_x (\psi_2)_x] = 
\psi_2^{[2]}  \ck \cos(x)\, ,\\ 
& \Pi_1 [\ck (\eta_1)_x^2 (\psi_1)_x ] = \tfrac14 \ck^2\cos(x)\, , \quad \Pi_1[\ck^2 (\eta_1)_x (\eta_2)_x] = \eta_2^{[2]}\ck^2 \cos(x) \, , \quad \Pi_1[\frac{3}{2}\kappa(\eta_1)^2_x(\eta_1)_{xx}] = -\tfrac38\kappa \cos(x) \, ,
\end{align*}
and, in view of \eqref{expDiriNeu}, and \eqref{exp:Sto}, \eqref{expcoef}, 
\begin{align*}
 \Pi_1[ G_1(\eta_1)\psi_2]  = 0 \, ,  \quad 
  \Pi_1[G_2(\eta_1)\psi_1] 
 =   \frac{1}{4}\ck\sin(x) \, , \quad \Pi_1[G_1(\eta_2)\psi_1] =  
  \eta_2^{[2]} \ck\sin(x) \, . 
\end{align*}
Imposing the orthogonality condition gives $c_2$ as in   \eqref{expcoef}.

\noindent{\bf Proof of \eqref{expfe}.} 
We expand  the function $\mathfrak{p}(x)  = \e\mathfrak{p}_1(x) + \e^2 \mathfrak{p}_2(x) + \cO(\e^3)  $ defined by the fixed point equation \eqref{def:ttf}.  
Then 
$ \mathfrak{p}(x)
 = \mathfrak H \big[\e\eta_1 +\e^2\big(\eta_2 + 
 (\eta_1)_x \mathfrak{p}_1 \big)+\cO(\e^3)\big] $, 
and, using that $\mathfrak H \cos(kx) = \sin(kx)$, for any  $ k \in \N$, we obtain 
\begin{align} \label{pfra1}
 \mathfrak{p}_1(x) 
  = \mathfrak H\cos(x)  = \sin(x) \, , \quad
 \mathfrak{p}_2(x) = \mathfrak H(\eta_2+(\eta_1)_x
 \mathfrak{p}_1  ) 
 = \frac{2-\kappa }{2(1-2\kappa)}\sin(2x) \, .
\end{align}
The expansion \eqref{expfe} is proved.
\\[1mm]
{\bf Proof of Lemma \ref{lem:pa.exp}.} 
In view of \eqref{exp:Sto}-\eqref{expcoef}, the expansions of the functions $B$, $V$  in \eqref{espV} and \eqref{espB} are
\begin{align}
 B= : \e B_1(x) + \e^2 B_2(x) + \cO(\e^3) = \e \ck\sin(x) + \e^2 \frac{4\kappa +1}{2(1-2\kappa)}\ck\sin(2x) + \cO(\e^3) \label{espB1}  
  \end{align}
  and
  \begin{align}
 V= : \e V_1(x) + \e^2 V_2(x)  + \cO(\e^3) = \e \ck\cos(x) + \e^2 \Big[\frac{1}{2}\ck  + \frac{4\kappa +1}{2(1-2\kappa)}\ck\cos(2x)\Big]+\cO(\e^3)
  \, .\label{espV1} 
\end{align}
In  view of  \eqref{def:pa}, denoting derivatives w.r.t $x$ with  a prime and suppressing dependence on $x$ when trivial, we have
\begin{align}
\ck+p_\e(x)  &= \left(\ck+\e^2 c_2 - V(x) - V'(x)\mathfrak{p}(x)+\cO(\e^3)\right) \left(1-\mathfrak{p}'(x)+(\mathfrak{p}'(x))^2+\cO(\e^3)\right)\notag \\
&= \ck + \e \underbrace{(- V_1 -\ck \mathfrak{p}_1')}_{=: p_1}+\e^2 \underbrace{\big( c_2 + V_1\mathfrak{p}_1' - V_2 - V_1'\mathfrak{p}_1 -\ck \mathfrak{p}_2' +\ck(\mathfrak{p}_1')^2 \big)}_{=:p_2} + \, \cO(\e^3) \, .\label{pino12imp}
\end{align} 
Similarly by \eqref{def:pa}
\begin{align}
 1+a_\e(&x) : = \frac{1}{1+\mathfrak{p}_x(x)} - (\ck+p_\e(x))B_x(x+\mathfrak{p}(x)) \notag \\
   = &  1+\e\underbrace{\big(-\mathfrak{p}_1'- \ck B_1'\big)}_{=:a_1} +\e^2\underbrace{\big((\mathfrak{p}_1')^2-\mathfrak{p}_2'- \ck B_2'-\ck B_1''\mathfrak{p}_1(x)+ B_1' V_1 + \ck  B_1'\mathfrak{p}_1' \big)}_{=:a_2}+\cO(\e^3)\, . \label{aino12imp}
 \end{align}
By \eqref{espV1}, \eqref{pfra1}, \eqref{exp:Sto},  \eqref{espB1} we 
deduce that the functions $p_1 $, $p_2 $, $a_1 $, $a_2 $ in \eqref{pino12imp} and \eqref{aino12imp} have an expansion as in \eqref{pino1fd}-\eqref{aino2fd}.

In view of \eqref{exp:Sto}-\eqref{expcoef}, the expansion of the function $l(x)$ in \eqref{esp l} is 
\begin{align}
 l(x)&=: 1 + \e^2 l_2(x)  + \cO(\e^3)= 1+ \e^2  \Big[-\frac{3}{4}+\frac{3}{4}\cos(2x)\Big]+\cO(\e^3)
 \label{espl1} \, . 
\end{align}
In  view of  \eqref{def:Sigma g}, denoting derivatives w.r.t $x$ with  a prime and suppressing dependence on $x$ when trivial, we have
\begin{align*}
g_\e(x)  = \left(1-\e\mathfrak{p}_1'+\e^2((\mathfrak{p}_1')^2-\mathfrak{p}_2')+\cO(\e^3)\right) \left(1+\e^2 l_2+\cO(\e^3)\right)= 1 + \e \underbrace{(-\mathfrak{p}_1')}_{=: g_1}+\e^2 \underbrace{\left((\mathfrak{p}_1')^2-\mathfrak{p}_2'+l_2\right)}_{=:g_2} + \, \cO(\e^3) \, .
\end{align*} 
By \eqref{expfe} and \eqref{espl1}, the expansion \eqref{exp:g1} follows.

Finally the operator $\Sigma_\e$ in \eqref{def:Sigma g} expands as
$\Sigma_\e = \pa_{xx} + \e \Sigma_1 + \e^2 \Sigma_2 + \cO(\e^3)$
where
\begin{align*}
    \Sigma_1:=& (g_{1}-2\mathfrak{p}_1')\,\pa_{xx}+(g_1'-2\,\mathfrak{p}_1'')\pa_{x}-\mathfrak{p}_1''',\\ \notag
    \Sigma_2:=&(g_2-2\mathfrak{p}_2'+3(\mathfrak{p}_1')^2-2g_1\mathfrak{p}_1') \pa_{xx}+(g_2'-2\mathfrak{p}_2''-2g_1\mathfrak{p}_1''-2g_1'\mathfrak{p}_1'+6\mathfrak{p}_1''\mathfrak{p}_1')\pa_{x}\\ 
    &+2(\mathfrak{p}_1'')^2-\mathfrak{p}_2'''-\mathfrak{p}_1''g_1'+3\mathfrak{p}_1'''\mathfrak{p}_1'-g_1\mathfrak{p}_1'''.
\end{align*}
Inserting the expansions of $g_1,g_2$ in \eqref{exp:g1} and those of $\mathfrak{p}_1$, $\mathfrak{p}_2$ in \eqref{pfra1} proves the claimed expansions 
in \eqref{Sigma 1} with coefficients in \eqref{d1x}--\eqref{h2 h02}.
\qed

\section{Expansion of the Kato Basis}\label{secA1}

In this appendix we prove Lemma \ref{expansion of the basis F}. We provide the expansion of the basis $f_k^{\pm}(\mu,\epsilon) = U_{\mu,\epsilon} f_k^{\pm}$, $k = 0,1$, in \eqref{F basis set and f}, where $f_k^{\pm}$ defined in \eqref{eigenfunc of mathcall L00 2} belong to the subspace $\mathcal{V}_{0,0} := \operatorname{Rg}(P_{0,0})$. We first Taylor-expand the transformation operators $U_{\mu,\epsilon}$ defined in \eqref{U transformation operators}. We denote $\partial_{\epsilon}$ with a prime and $\partial_{\mu}$ with a dot.
The next lemma follows as \cite[Lemma A.1]{BMV1}
\begin{lemma} \label{U mu epsilon P00}
    The first jets of $U_{\mu,\epsilon} P_{0,0}$ are
\begin{align} \label{A1}
    &U_{0,0} P_{0,0} = P_{0,0}, \quad\quad U'_{0,0} P_{0,0} = P'_{0,0} P_{0,0}, \quad\quad \dot{U}_{0,0} P_{0,0} = \dot{P}_{0,0} P_{0,0},\\ \label{A2}
    &\dot{U}'_{0,0} P_{0,0} = \left( \dot{P}'_{0,0} - \frac{1}{2} P_{0,0} \dot{P}'_{0,0} \right) P_{0,0}, 
\end{align}
where
\begin{align} \label{A3}
    P'_{0,0} &= \frac{1}{2\pi \im} \oint_{\Gamma} (\mathscr{L}_{0,0} - \lambda)^{-1} \mathscr{L}'_{0,0} (\mathscr{L}_{0,0} - \lambda)^{-1} d\lambda\,,  \\ \label{A4}
    \dot{P}_{0,0} &= \frac{1}{2\pi \im} \oint_{\Gamma} (\mathscr{L}_{0,0} - \lambda)^{-1} \dot{\mathscr{L}}_{0,0} (\mathscr{L}_{0,0} - \lambda)^{-1} d\lambda\,, 
\end{align}
and
\begin{align} \label{A5a}
    \dot{P}'_{0,0} &= -\frac{1}{2\pi \im} \oint_{\Gamma} (\mathscr{L}_{0,0} - \lambda)^{-1} \dot{\mathscr{L}}_{0,0} (\mathscr{L}_{0,0} - \lambda)^{-1} \mathscr{L}'_{0,0} (\mathscr{L}_{0,0} - \lambda)^{-1} d\lambda \\ \label{A5b}
    &\quad -\frac{1}{2\pi \im} \oint_{\Gamma} (\mathscr{L}_{0,0} - \lambda)^{-1} \mathscr{L}'_{0,0} (\mathscr{L}_{0,0} - \lambda)^{-1} \dot{\mathscr{L}}_{0,0} (\mathscr{L}_{0,0} - \lambda)^{-1} d\lambda \\ \label{A5c}
    &\quad + \frac{1}{2\pi \im} \oint_{\Gamma} (\mathscr{L}_{0,0} - \lambda)^{-1} \dot{\mathscr{L}}'_{0,0} (\mathscr{L}_{0,0} - \lambda)^{-1} d\lambda.
\end{align}
The operators $\mathscr{L}'_{0,0}$ and $\dot{\mathscr{L}}_{0,0}$ are
\begin{align} \label{A6}
    \mathscr{L}'_{0,0}=\begin{bmatrix}
        \partial_x\circ p_1(x)  &0\\
        -a_1(x)+\kappa \Sigma_1 &p_1(x)\circ \partial_x
    \end{bmatrix},~~\dot{\mathscr{L}}_{0,0}=\begin{bmatrix}
        0 &\mathrm{sgn}(D)+\Pi_0\\
        2\im \kappa \partial_{x} &0
    \end{bmatrix},
\end{align}
where $\mathrm{sgn}(D)$ is defined in \eqref{D+mu}, \eqref{sgn D} 
and $a_1(x)$, $p_1(x)$, $\Sigma_1$ are given in Lemma \ref{lem:pa.exp}.

The operator $\dot{\mathscr{L}}'_{0,0}$ is
\begin{align} \label{A8}
    \dot{\mathscr{L}}'_{0,0}=\begin{bmatrix}
        \im p_1(x) & 0\\
        \im \kappa\left(2 d_1(x) \pa_x + e_1(x)\right)        & \im p_1(x)
    \end{bmatrix}
\end{align}
with $d_1(x), e_1(x)$ in \eqref{d1x}. 
\end{lemma}

By the Lemma \ref{U mu epsilon P00}, we have the Taylor expansion
\begin{equation} \label{the expandsion of f sigma mu}
    \begin{aligned}
    f_k^\sigma (\mu, \epsilon) &= f_k^\sigma + \epsilon P'_{0,0} f_k^\sigma + \mu \dot{P}_{0,0} f_k^\sigma  + \mu \epsilon \big( \dot{P}'_{0,0} - \frac{1}{2} P_{0,0} \dot{P}'_{0,0} \big) f_k^\sigma + \mathcal{O}(\mu^2, \epsilon^2).
\end{aligned}
\end{equation}
In order to compute the vectors $P'_{0,0} f_k^\sigma$ and $\dot{P}_{0,0} f_k^\sigma$ using \eqref{A3} and \eqref{A4}, it is useful to know the action of $(\mathscr{L}_{0,0} - \lambda)^{-1}$ on the vectors
\begin{equation} \label{A10}
    \begin{aligned}
        f^+_k:=\vet{\frac{1}{\sqrt{\ck}}\cos(kx)}{\sqrt{\ck}\sin(kx)},~f^-_k:=\vet{-\frac{1}{\sqrt{\ck}}\sin(kx)}{\sqrt{\ck}\cos(kx)},~f^+_{-k}:=\vet{\frac{1}{\sqrt{\ck}}\cos(kx)}{-\sqrt{\ck}\sin(kx)},~f^-_{-k}:=\vet{\frac{1}{\sqrt{\ck}}\sin(kx)}{\sqrt{\ck}\cos(kx)},~k\in\mathbb{N}.
    \end{aligned}
\end{equation}

\begin{lemma} \label{Lemma B2}
    The space $Y=H^2(\mathbb{T},\mathbb{C})\times H^1(\mathbb{T},\mathbb{C})$ decomposes as $Y = \mathcal{V}_{0,0} \oplus \mathcal{U} \oplus \mathcal{W}_{Y}$, with
\begin{equation*}
    \mathcal{W}_{Y} = \overline{\bigoplus_{k=2}^{\infty} \mathcal{W}_k}^{Y}
\end{equation*}
where the subspaces $\mathcal{V}_{0,0}, \mathcal{U}$, and $\mathcal{W}_k$, defined below, are invariant under $\mathscr{L}_{0,0}$ and the following properties hold:

\begin{itemize}
\item[(i)] $\mathcal{V}_{0,0} = \text{span}\{f_1^+, f_1^-, f_0^+, f_0^-\}$ is the generalized kernel of $\mathscr{L}_{0,0}$. For any $\lambda \neq 0$ the operator $\mathscr{L}_{0,0} - \lambda : \mathcal{V}_{0,0} \to \mathcal{V}_{0,0}$ is invertible and
    \begin{equation} \label{A11-0}
        (\mathscr{L}_{0,0} - \lambda)^{-1} f_1^+ = -\frac{1}{\lambda} f_1^+, \quad (\mathscr{L}_{0,0} - \lambda)^{-1} f_1^- = -\frac{1}{\lambda} f_1^-,
    \end{equation}
    \begin{equation} \label{A11}
        (\mathscr{L}_{0,0} - \lambda)^{-1} f_0^- = -\frac{1}{\lambda} f_0^-,
    \end{equation}
    \begin{equation} \label{A12}
        (\mathscr{L}_{0,0} - \lambda)^{-1} f_0^+ = -\frac{1}{\lambda} f_0^+ + \frac{1}{\lambda^2} f_0^- .
    \end{equation}

    \item[(ii)] $\mathcal{U} := \text{span}\{ f_{-1}^+, f_{-1}^- \}$. For any $\lambda \neq \pm \im \,2\ck$ the operator $\mathscr{L}_{0,0} - \lambda : \mathcal{U} \to \mathcal{U}$ is invertible and
    \begin{equation} \label{A13}
    \begin{aligned}
        (\mathscr{L}_{0,0} - \lambda)^{-1} f_{-1}^+ &= \frac{1}{\lambda^2 + 4\ck^2} \left(-\lambda f_{-1}^+ + 2\ck f_{-1}^-\right),\\
        (\mathscr{L}_{0,0} - \lambda)^{-1} f_{-1}^- &= \frac{1}{\lambda^2 + 4\ck^2} \left(-2\ck f_{-1}^+ - \lambda f_{-1}^-\right).
    \end{aligned}
    \end{equation}

    \item[(iii)] Each subspace $\mathcal{W}_k := \text{span}\{ f_k^+, f_k^-, f_{-k}^+, f_{-k}^- \}$ is invariant under $\mathscr{L}_{0,0}$. Let
    \begin{equation*}
        \mathcal{W}_{L^2} = \overline{\bigoplus_{k=2}^{\infty} \mathcal{W}_k}^{ L^2}.
    \end{equation*}
    For any $|\lambda| < \delta(\kappa)$ small enough, the operator $\mathscr{L}_{0,0} - \lambda : \mathcal{W}_{Y} \to \mathcal{W}_{L^2}$ is invertible and for any $f \in \mathcal{W}_{L^2}$,
    \begin{equation} \label{A14}
        (\mathscr{L}_{0,0} - \lambda)^{-1} f = \left( \ck^2 \partial^2_{x} + |D| (1-\kappa \pa^2_{x}) \right)^{-1} \begin{bmatrix} \ck \partial_x & -|D|  \\ 1-\kappa \pa^2_{x} & \ck \partial_x \end{bmatrix} f
        + \lambda \varphi_f(\lambda, x),
    \end{equation}
    for some analytic function $\lambda \mapsto \varphi_f(\lambda, \cdot) \in Y=H^2(\mathbb{T},\mathbb{C})\times H^1(\mathbb{T},\mathbb{C})$.
\end{itemize}
\end{lemma}

\begin{proof}
    By inspection the spaces $\mathcal{V}_{0,0}, \mathcal{U}$ and $\mathcal{W}_k$ are invariant under $\mathscr{L}_{0,0}$ and, by Fourier series, they decompose $Y=H^2(\mathbb{T},\mathbb{C})\times H^1(\mathbb{T},\mathbb{C})$. Formulas \eqref{A11}-\eqref{A12} follow using that $f_1^+, f_1^-, f_0^-$ are in the kernel of $\mathscr{L}_{0,0}$, and $\mathscr{L}_{0,0} f_0^+ = -f_0^-$. Formula \eqref{A13} follows using that $\mathscr{L}_{0,0} f_{-1}^+ = -2\ck f_{-1}^-$ and $\mathscr{L}_{0,0} f_{-1}^- = 2\ck f_{-1}^+$. Let us prove item $(iii)$. Let $\mathcal{W} := \mathcal{W}_{Y}$. The operator $(\mathscr{L}_{0,0} - \lambda \mathrm{Id})|_{\mathcal{W}}$ is invertible for any 
$$\lambda \notin \{ \pm \im k\ck \pm \im\sqrt{|k|\left(1+\kappa k^2\right)}, k \geq 2, k \in \mathbb{N} \}$$
and 
$$
(\mathscr{L}_{0,0}|_{\mathcal{W}})^{-1} = \left( \ck^2 \partial^2_{x} + |D| (1-\kappa \pa^2_{x}) \right)^{-1} \begin{bmatrix} \ck \partial_x & -|D|  \\ 1-\kappa \pa^2_{x} & \ck \partial_x \end{bmatrix} \Big|_{\mathcal{W}}.
$$
By Neumann series, for any $\lambda$ such that 
$$|\lambda| \| (\mathscr{L}_{0,0}|_{\mathcal{W}})^{-1} \|_{\mathcal{L}(\mathcal{W}, Y)} < 1$$
we have 
$$
(\mathscr{L}_{0,0}|_{\mathcal{W}} - \lambda)^{-1} = (\mathscr{L}_{0,0}|_{\mathcal{W}})^{-1} (\text{Id} - \lambda (\mathscr{L}_{0,0}|_{\mathcal{W}})^{-1})^{-1} 
= (\mathscr{L}_{0,0}|_{\mathcal{W}})^{-1} \sum_{k \geq 0} ((\mathscr{L}_{0,0}|_{\mathcal{W}})^{-1} \lambda)^k.
$$
Formula \eqref{A14} follows with 
$$\varphi_f(\lambda, x) := (\mathscr{L}_{0,0}|_{\mathcal{W}})^{-1} \sum_{k \geq 1} \lambda^{k-1} [(\mathscr{L}_{0,0}|_{\mathcal{W}})^{-1}]^k f.$$ 
\end{proof}

To prove Lemma \ref{expansion of the basis F}, we shall also use the following formulas obtained by \eqref{A6} and \eqref{eigenfunc of mathcall L00 2}:

\begin{equation} \label{A15}
    \begin{aligned}
        &\mathscr{L}'_{0,0} f^+_1=\vet{2\sqrt{\ck}\sin(2x)}{\frac{3\kappa}{\sqrt{\ck}}\cos(2x)}\,,\quad \mathscr{L}'_{0,0} f^-_1=\vet{2\sqrt{\ck}\cos(2x)}{-\frac{3\kappa}{\sqrt{\ck}}\sin(2x)}\,,\quad \mathscr{L}'_{0,0} f^+_0=\vet{2\ck\sin(x)}{2\ck^2\cos(x)}, \quad \mathscr{L}'_{0,0}f^-_0=\vet{0}{0},\\
        &\dot{\mathscr{L}}_{0,0} f^+_1=\vet{-\im\sqrt{\ck}\cos(x)}{-\frac{2\im\kappa}{\sqrt{\ck}}\sin(x)}, \quad \dot{\mathscr{L}}_{0,0} f^-_1=\vet{\im \sqrt{\ck} \sin(x)}{-\frac{2\im\kappa}{\sqrt{\ck}}\cos(x)}\,,\quad \dot{\mathscr{L}}_{0,0} f^+_0=\vet{0}{0}, \quad \dot{\mathscr{L}}_{0,0}f^-_0=\vet{1}{0}=f^+_0.
    \end{aligned}
\end{equation}
We finally calculate $P'_{0,0} f^\sigma_k$ and $\dot{P}_{0,0} f^\sigma_k$.

\begin{lemma} One has
    \begin{equation} \label{A16}
        \begin{aligned}
    P'_{0,0}f^+_1&=\vet{\frac{2-\kappa}{(1-2\kappa)\sqrt{\ck}}\cos(2x)}{\frac{\ck^2\sqrt{\ck}}{1-2\kappa}\sin(2x)},\quad P'_{0,0}f^-_1=\vet{-\frac{2-\kappa}{(1-2\kappa)\sqrt{\ck}}\sin(2x)}{\frac{\ck^2\sqrt{\ck}}{1-2\kappa}\cos(2x)}, \quad 
            P'_{0,0}f^+_0= \sqrt{\ck}f^+_{-1}, \\
            P'_{0,0}f^-_0& =\vet{0}{0}, \quad \dot{P}_{0,0}f^+_0=\vet{0}{0}, \quad \dot{P}_{0,0}f^-_0=\vet{0}{0}, \quad 
            \dot{P}_{0,0}f^+_1=\im\frac{1-\kappa}{4\ck^2}f^-_{-1}, \quad  \dot{P}_{0,0}f^-_1=\im\frac{1-\kappa}{4\ck^2}f^+_{-1},
        \end{aligned}
    \end{equation}
\end{lemma}
\begin{proof}

    We first calculate $P'_{0,0} f^+_1$. By \eqref{A3}, \eqref{A11-0}, \eqref{A11}, and \eqref{A15} we deduce
    \begin{align*}
        P'_{0,0} f^+_1=-\frac{1}{2\pi \im}\oint_{\Gamma} \frac{1}{\lambda} \left(\mathscr{L}_{0,0}-\lambda\right)^{-1} \vet{2\sqrt{\ck}\sin(2x)}{\frac{3\kappa}{\sqrt{\ck}}\cos(2x)} d\lambda.
    \end{align*}
We note that 
\begin{equation*}
    \vet{2\sqrt{\ck}\sin(2x)}{\frac{3\kappa}{\sqrt{\ck}}\cos(2x)}=\left(\frac{3\kappa}{2\ck}-\ck\right)f^-_2+\left(\frac{3\kappa}{2\ck}+\ck\right)f^-_{-2}\in\mathcal{W}  . 
\end{equation*}
Therefore by \eqref{A11} and \eqref{A14} there is an analytic function $\lambda \mapsto \varphi(\lambda,\cdot)\in Y$ so that we obtain
\begin{align*}
    P'_{0,0} f^+_1&=-\frac{1}{2\pi \im} \oint_\Gamma \frac{1}{\lambda}\left(\left( \ck^2 \partial^2_{x} + |D| (1-\kappa \pa^2_{x}) \right)^{-1} 
    \begin{bmatrix} \ck \partial_x & -|D|  \\ 1-\kappa \pa^2_{x} & \ck \partial_x \end{bmatrix} \vet{2\sqrt{\ck}\sin(2x)}{\frac{3\kappa}{\sqrt{\ck}}\cos(2x)}
        +\lambda\varphi(\lambda)\right) d\lambda\\
        &=\frac{1}{2\pi \im} \oint_\Gamma \frac{1}{\lambda}\left(\vet{\frac{2-\kappa}{(1-2\kappa)\sqrt{\ck}}\cos(2x)}{\frac{\ck^2\sqrt{\ck}}{1-2\kappa}\sin(2x)}
        -\lambda\varphi(\lambda)\right) d\lambda.
\end{align*}
Thus, by means of the residue theorem we obtain the first identity in \eqref{A16}. Similarly, one may calculate $P'_{0,0} f^-_1$. By \eqref{A3}, \eqref{A11}, and \eqref{A15}, one has $P'_{0,0}f^-_0=0$. Next, we calculate $P'_{0,0}f^+_0$. By \eqref{A3}, \eqref{A11}, \eqref{A12}, and \eqref{A15} we get 
\[
P'_{0,0} f^+_0 = -\frac{1}{2\pi \im} \oint_{\Gamma} \frac{1}{\lambda} (\mathscr{L}_{0,0} - \lambda)^{-1} \vet{2\ck\sin(x)}{2\ck^2\cos(x)} d\lambda= -\frac{1}{2\pi \im} \oint_{\Gamma} \frac{1}{\lambda} (\mathscr{L}_{0,0} - \lambda)^{-1}2\ck\sqrt{\ck}f^-_{-1}d\lambda.
\]

By \eqref{A13} we get
\[
P'_{0,0} f^+_0 = -\frac{1}{2\pi \im} \oint_{\Gamma} 
2\ck\sqrt{\ck}\left( -\frac{2\ck}{\lambda (\lambda^2 + 4\ck^2)} f^+_{-1} 
- \frac{1}{\lambda^2 + 4\ck^2} f^-_{-1} 
\right) d\lambda.
\]
Applying the residue theorem, we obtain
\[
P'_{0,0} f^+_0 = \sqrt{\ck} f^+_{-1}.
\]
Thus, we obtain the third identity in \eqref{A16}. Now, we calculate $\dot{P}_{0,0} f^+_1$. First, we have 
\begin{align*}
    \dot{P}_{0,0} f^+_1=\frac{1}{2\pi }\oint_\Gamma \frac{1}{\lambda}\left(\mathscr{L}_{0,0}-\lambda\right)^{-1} \vet{\sqrt{\ck} \cos(x)}{\frac{2\kappa}{\sqrt{\ck}}\sin(x)} d\lambda,
\end{align*}
and then, writing
\begin{align*}
   \vet{\sqrt{\ck} \cos(x)}{\frac{2\kappa}{\sqrt{\ck}}\sin(x)}=\frac{\ck^2+2\kappa}{2\ck}f^+_1+\frac{\ck^2-2\kappa}{2\ck}f^+_{-1},
\end{align*}
and using \eqref{A13}, we conclude using again the residue theorem 
\begin{align*}
    \dot{P}_{0,0}f^+_1=\im\frac{1-\kappa}{4\ck^2}  f^-_{-1}.
\end{align*}
Similarly, we have
\begin{align*}
    \dot{P}_{0,0}f^-_1=\im\frac{1-\kappa}{4\ck^2} f^+_{-1}.
\end{align*}
Finally, in view of \eqref{A15}, we have
\begin{equation*}
    \begin{aligned}
        \dot{P}_{0,0}f^+_0=&\frac{1}{2\pi \im}\oint_\Gamma \left(\mathscr{L}_{0,0}-\lambda\right)^{-1}\dot{\mathscr{L}}_{0,0}\left(\frac{1}{\lambda^2}f^-_0-\frac{1}{\lambda}f^+_0\right) d\lambda=0,\\
        \dot{P}_{0,0}f^-_0=&\frac{1}{2\pi \im}\oint_\Gamma \left(\mathscr{L}_{0,0}-\lambda\right)^{-1}\dot{\mathscr{L}}_{0,0}\left(\frac{-1}{\lambda}f^-_0\right) d\lambda=0.
    \end{aligned}
\end{equation*}
In conclusion, all the formulas in \eqref{A16} are proven.
\end{proof}

So far we have obtained the linear terms of the expansions \eqref{43 f+1}, \eqref{44 f-1}, \eqref{45 f+0}, and \eqref{46 f-0}. We now provide further information about the expansion of the basis at $\mu = 0$. 

\begin{lemma} 
     The basis $\{ f_k^{\sigma}(0,\e), \quad k = 0,1, \quad \sigma = \pm \}$ is real. For any $\epsilon$ it results $f_0^-(0,\epsilon) \equiv f_0^-$. The property \eqref{48 f-0=01} holds.
\end{lemma}

\begin{proof}
    The reality of the basis $f_k^\sigma(0,\epsilon)$ is a consequence of Lemma \ref{properties of U and P}-(iii). Since, recalling \eqref{mathscr L mu e} and \eqref{mathcal B mu e}, $\mathscr{L}_{0,\epsilon} f_0^- = 0$ for any $\epsilon$ (cf. \eqref{237}), we deduce $(\mathscr{L}_{0,\epsilon} - \lambda)^{-1} f_0^- = -\frac{1}{\lambda} f_0^-$ and then, using also the residue theorem,
\[
P_{0,\epsilon} f_0^- = -\frac{1}{2\pi \im} \oint_\Gamma (\mathscr{L}_{0,\epsilon} - \lambda)^{-1} f_0^- \, d\lambda = f_0^-.
\]
In particular $P_{0,\epsilon} f_0^- = P_{0,0} f_0^-$, for any $\epsilon$ and we get, by \eqref{U transformation operators}, $f_0^-(0,\epsilon) = U_{0,\epsilon} f_0^- = f_0^-$, for any $\epsilon$.

Let us prove property \eqref{48 f-0=01}. In view of \eqref{Parity structure} and since the basis is real, we know that
\[
f_k^+(0,\epsilon) = \begin{bmatrix}\mathit{even}(x)\\ \mathit{odd}(x)\end{bmatrix}, \quad f_k^-(0,\epsilon) = \begin{bmatrix}\mathit{odd}(x)\\ \mathit{even}(x)\end{bmatrix},
\]
for any $k = 0, 1$. By Lemma \ref{F is symplectic and reversible} the basis $\{f_k^\sigma(0,\epsilon)\}$ is symplectic (cf. \eqref{basis is symplectic}) and, since
\[
\mathcal{J} f_0^-(0,\epsilon) = \mathcal{J} f_0^- = \begin{bmatrix}1\\0\end{bmatrix},
\]
for any $\epsilon$, we get
\[
0 = (\mathcal{J} f_0^-(0,\epsilon),\, f_1^+(0,\epsilon)) = \left(\begin{bmatrix}1\\0\end{bmatrix}, f_1^+(0,\epsilon)\right), \qquad 
1 = (\mathcal{J} f_0^-(0,\epsilon),\, f_0^+(0,\epsilon)) = \left(\begin{bmatrix}1\\0\end{bmatrix}, f_0^+(0,\epsilon)\right).
\]
Thus the first component of both $f_1^+(0,\epsilon)$ and $f_0^+(0,\epsilon) - \begin{bmatrix}1\\0\end{bmatrix}$ has zero average, proving \eqref{48 f-0=01}.
\end{proof}

\begin{lemma} \label{fmu0 eq f0} 
For any small $\mu$, we have $f_0^+(\mu,0) \equiv f_0^+$ and $f_0^-(\mu,0) \equiv f_0^-$. Moreover, the vectors $f_1^+(\mu,0)$ and $f_1^-(\mu,0)$ have both components with zero space average.
\end{lemma}
\begin{proof}
The operator $\mathscr{L}_{\mu,0} = 
\begin{bmatrix}
\ck\partial_x & |D+\mu| \\
-1+\kappa(\pa^2_{x}+2\im \mu\pa_x-\mu^2) & \ck\partial_x
\end{bmatrix}$ 
leaves invariant the subspace $\mathcal{Z} := \text{span}\{ f_0^+, f_0^- \}$ since 
$\mathscr{L}_{\mu,0} f_0^+ = -(1+\kappa\mu^2)f_0^-$ and $\mathscr{L}_{\mu,0} f_0^- = \mu f_0^+$. 
The operator $\mathscr{L}_{\mu,0}|_{\mathcal{Z}}$ has the two eigenvalues $\pm \im \sqrt{(1+\kappa\mu^2)\mu}$, 
which, for small $\mu$, lie inside the loop $\Gamma$ around $0$ in \eqref{Projection P mu e}. Then, by \eqref{spectrum separated by Gamma}, 
we have $\mathcal{Z} \subseteq \mathcal{V}_{\mu,0} = \text{Rg}(P_{\mu,0})$ and 

\[
P_{\mu,0} f_0^\pm = f_0^\pm, \quad 
f_0^\pm (\mu,0) = U_{\mu,0} f_0^\pm = f_0^\pm, 
\quad \text{for any $\mu$ small.}
\]

The basis $\{ f_k^\sigma (\mu,0) \}$ is symplectic (cf. \eqref{basis is symplectic}). Then, since 
$\mathcal{J} f_0^+ = \begin{bmatrix} 0 \\ -1 \end{bmatrix}$ 
and $\mathcal{J} f_0^- = \begin{bmatrix} 1 \\ 0 \end{bmatrix}$, we have

\[
\begin{aligned}
0 &= \big( \mathcal{J} f_0^+ (\mu,0), f_1^\sigma (\mu,0) \big) 
= \Big( \begin{bmatrix} 0 \\ -1 \end{bmatrix}, f_1^\sigma (\mu,0) \Big), \quad 
0 = \big( \mathcal{J} f_0^- (\mu,0), f_1^\sigma (\mu,0) \big) 
= \Big( \begin{bmatrix} 1 \\ 0 \end{bmatrix}, f_1^\sigma (\mu,0) \Big),
\end{aligned}
\]
namely both the components of $f_1^\pm (\mu,0)$ have zero average.
\end{proof}

We finally consider the $\mu \epsilon$ term in the expansion \eqref{the expandsion of f sigma mu}.

\begin{lemma} 
The derivatives $ \partial_\mu \partial_\epsilon f_k^{\sigma}(0,0) = \left( \dot{P}_{0,0}^{\prime} - \frac{1}{2} P_{0,0} \dot{P}_{0,0}^{\prime} \right) f_k^{\sigma}$ satisfy
\begin{equation} \label{A17}
\begin{aligned}
\partial_\mu \partial_\epsilon f_1^+(0,0) &= \im \begin{bmatrix} \mathit{odd}(x) \\ \mathit{even}(x) \end{bmatrix}, & \partial_\mu \partial_\epsilon f_1^-(0,0) &= \im \begin{bmatrix} \mathit{even}(x) \\ \mathit{odd}(x) \end{bmatrix}, \\
\partial_\mu \partial_\epsilon f_0^+(0,0) &=\textcolor{black}{\frac{1}{4}
\begin{bmatrix}
\ck^{-2}\cos(x)\\
-\ck^{-1}\sin(x)
\end{bmatrix}
+\im
\begin{bmatrix}
\mathit{odd}(x)\\
\mathit{even}_0(x)
\end{bmatrix}.} & \partial_\mu \partial_\epsilon f_0^-(0,0) &= \frac{1}{2}\begin{bmatrix} \frac{1}{\ck}\sin(x) \\ \cos(x) \end{bmatrix}+\im \begin{bmatrix} \mathit{even}_0(x) \\ \mathit{odd}(x) \end{bmatrix}.
\end{aligned}
\end{equation}
\end{lemma}

 \begin{proof}
 Recall \eqref{A6} and decompose the Fourier multiplier operator $\dot{\mathscr{L}}_{0,0}$ as
\begin{align*}
    \dot{\mathscr{L}}_{0,0}=\dot{\mathscr{L}}_{0,0}^{I}+\dot{\mathscr{L}}_{0,0}^{II}\,,\quad \dot{\mathscr{L}}_{0,0}^{I}:=\begin{bmatrix}
        0 &\mathrm{sgn}(D)\\
        2\im \kappa \partial_{x} &0
    \end{bmatrix}\,,\quad
    \dot{\mathscr{L}}_{0,0}^{II}:=\begin{bmatrix}
        0 &\Pi_0\\
       0 &0
    \end{bmatrix}.
        \end{align*}
 Moreover, we write 
 \begin{align*}
      \dot{P}_{0,0}^{\prime} = \eqref{A5a}^{I}+\eqref{A5a}^{II}+\eqref{A5b}^{I}+\eqref{A5b}^{II}+\eqref{A5c}
 \end{align*}
 where
\begin{align*}
   \eqref{A5a}^{I} &:= -\frac{1}{2\pi \im} \oint_{\Gamma} (\mathscr{L}_{0,0} - \lambda)^{-1} \dot{\mathscr{L}}_{0,0}^{I} (\mathscr{L}_{0,0} - \lambda)^{-1} \mathscr{L}'_{0,0} (\mathscr{L}_{0,0} - \lambda)^{-1} d\lambda\,,\\
  \eqref{A5a}^{II} &:= -\frac{1}{2\pi \im} \oint_{\Gamma} (\mathscr{L}_{0,0} - \lambda)^{-1} \dot{\mathscr{L}}_{0,0}^{II} (\mathscr{L}_{0,0} - \lambda)^{-1} \mathscr{L}'_{0,0} (\mathscr{L}_{0,0} - \lambda)^{-1} d\lambda\,,\\
  \eqref{A5b}^{I} &:= -\frac{1}{2\pi \im} \oint_{\Gamma} (\mathscr{L}_{0,0} - \lambda)^{-1} \mathscr{L}'_{0,0} (\mathscr{L}_{0,0} - \lambda)^{-1}\dot{\mathscr{L}}_{0,0} ^{I}(\mathscr{L}_{0,0} - \lambda)^{-1}  d\lambda\,,\\
  \eqref{A5b}^{II} &:= -\frac{1}{2\pi \im} \oint_{\Gamma} (\mathscr{L}_{0,0} - \lambda)^{-1} \mathscr{L}'_{0,0} (\mathscr{L}_{0,0} - \lambda)^{-1}\dot{\mathscr{L}}_{0,0} ^{II}(\mathscr{L}_{0,0} - \lambda)^{-1}  d\lambda\,.
        \end{align*}
 Note that is $\eqref{A5a}^{I}$, $\eqref{A5b}^{I}$ and \eqref{A5c} are purely imaginary. Indeed, $\dot{\mathscr{L}}_{0,0}$ is purely imaginary, whereas $\mathscr{L}_{0,0}^{\prime}$ in \eqref{A6} is real, and $\dot{\mathscr{L}}_{0,0}^{\prime}$ in \eqref{A8} is purely imaginary (see Lemma \ref{properties of U and P}-(iii) of \cite{BMV1}). Consequently, when applied to the real vectors $f_k^{\sigma}$, for $k = 0,1$ and $\sigma = \pm$, these operators produce purely imaginary vectors.
 
We first compute $(\partial_\mu \partial_\epsilon f_k^{\sigma}(0,0))$. Using \eqref{A11-0}, \eqref{A11} and \eqref{A15} we obtain
\begin{align*}
   \eqref{A5a}^{II} f_1^+ = \frac{1}{2\pi \im} \oint_{\Gamma} (\mathscr{L}_{0,0} - \lambda)^{-1} \dot{\mathscr{L}}_{0,0}^{II} (\mathscr{L}_{0,0} - \lambda)^{-1} \frac{1}{\lambda} \vet{2\sqrt{\ck}\sin(2x)}{\frac{3\kappa}{\sqrt{\ck}}\cos(2x)} d\lambda=0,
        \end{align*}
because, by Lemma~\ref{Lemma B2}, $ (\mathscr{L}_{0,0} - \lambda)^{-1}\vet{2\sqrt{\ck}\sin(2x)}{\frac{3\kappa}{\sqrt{\ck}}\cos(2x)}\in\mathcal{W} $ and therefore this vector has zero average, implying that it lies in the kernel of $\dot{\mathscr{L}}_{0,0}^{II} $.
Moreover, $\eqref{A5b}^{II} f_1^+=0$, since $\dot{\mathscr{L}}_{0,0} ^{II}(\mathscr{L}_{0,0} - \lambda)^{-1}f_1^+=0$. Combining the above identities, we conclude that $  \dot{P}'_{0,0}f_1^+$
is a purely imaginary vector. Since $P_{0,0}$ is a real operator, it follows that $ \left( \dot{P}'_{0,0} - \frac{1}{2} P_{0,0} \dot{P}'_{0,0} \right)f_1^+$ is also purely imaginary. Finally, \eqref{Parity structure} implies that $\partial_\mu \partial_\epsilon f_1^+(0,0)$ has the structure claimed
in \eqref{A17}. The structure for $\partial_\mu \partial_\epsilon f_1^-(0,0)$ follows by the same argument.

\textcolor{black}{Next, we determine
$(\partial_\mu\partial_\epsilon f_0^+)(0,0)$.
We first show that all contributions except
$\eqref{A5b}^{II}f_0^+$ are purely imaginary and have
the required parity and zero-average properties.} By \eqref{A12} and \eqref{A15}, we have
\begin{align*}
   \eqref{A5a}^{I}f_0^+& = -\frac{1}{2\pi \im} \oint_{\Gamma} (\mathscr{L}_{0,0} - \lambda)^{-1} \dot{\mathscr{L}}_{0,0}^{I} (\mathscr{L}_{0,0} - \lambda)^{-1} \mathscr{L}'_{0,0} \left(\frac{-1}{\lambda}f_0^++\frac{1}{\lambda^2}f_0^-\right) d\lambda\\
   & = -\frac{1}{2\pi \im} \oint_{\Gamma} (\mathscr{L}_{0,0} - \lambda)^{-1} \dot{\mathscr{L}}_{0,0}^{I} (\mathscr{L}_{0,0} - \lambda)^{-1}  \frac{-1}{\lambda}\vet{2\ck\sin(x)}{2\ck^2\cos(x)}d\lambda\,. \end{align*}
Since $ (\mathscr{L}_{0,0} - \lambda)^{-1} $ and $\dot{\mathscr{L}}_{0,0}^{I}$ are Fourier multiplier operators, they preserve the zero-average property of functions. Therefore, $\eqref{A5a}^{I}f_0^+$ has zero average.
Moreover, 
\begin{align*}
   \eqref{A5a}^{II}f_0^+& = 0 \ \ \text{since}\ \  \dot{\mathscr{L}}_{0,0}^{II} (\mathscr{L}_{0,0} - \lambda)^{-1}  \vet{2\ck\sin(x)}{2\ck^2\cos(x)}=0\,.
   \end{align*}
Next, $\eqref{A5b}^{I}f_0^+=0$ because $\dot{\mathscr{L}}_{0,0}^{I}f_0^+=\dot{\mathscr{L}}_{0,0}^{I}f_0^-=0$.
\textcolor{black}{Using also that
$\dot{\mathscr{L}}_{0,0}^{II}f_0^+=0$ and
$\dot{\mathscr{L}}_{0,0}^{II}f_0^-=f_0^+$, we obtain
\begin{align*}
\eqref{A5b}^{II}f_0^+
&=
-\frac{1}{2\pi\im}
\oint_\Gamma
(\mathscr{L}_{0,0}-\lambda)^{-1}
\mathscr{L}'_{0,0}
(\mathscr{L}_{0,0}-\lambda)^{-1}
\dot{\mathscr{L}}_{0,0}^{II}
(\mathscr{L}_{0,0}-\lambda)^{-1}
f_0^+
\,d\lambda
\\
&=
-\frac{1}{2\pi\im}
\oint_\Gamma
(\mathscr{L}_{0,0}-\lambda)^{-1}
\mathscr{L}'_{0,0}
\left(
-\frac{1}{\lambda^3}f_0^+
+\frac{1}{\lambda^4}f_0^-
\right)
\,d\lambda
\\
&=
\frac{2\ck^{3/2}}{2\pi\im}
\oint_\Gamma
\frac{1}{\lambda^3}
(\mathscr{L}_{0,0}-\lambda)^{-1}
f_{-1}^-
\,d\lambda.
\end{align*}}
 \textcolor{black}{By \eqref{A13},
\begin{align*}
\eqref{A5b}^{II}f_0^+
&=
\frac{2\ck^{3/2}}{2\pi\im}
\oint_\Gamma
\frac{-2\ck f_{-1}^+-\lambda f_{-1}^-}
{\lambda^3(\lambda^2+4\ck^2)}
\,d\lambda.
\end{align*}
The coefficient of $f_{-1}^-$ has zero residue at
$\lambda=0$, whereas
\[
-\frac{4\ck^{5/2}}
{\lambda^3(\lambda^2+4\ck^2)}
=
-\frac{\ck^{1/2}}{\lambda^3}
+\frac{1}{4\ck^{3/2}}\frac{1}{\lambda}
+\mathcal{O}(\lambda).
\]
Therefore, by the residue theorem,
\begin{equation} \label{BB20}
\eqref{A5b}^{II}f_0^+
=
\frac{1}{4\ck^{3/2}}f_{-1}^+.
\end{equation}}
Finally, by \eqref{A12} and \eqref{A8} where $p_1(x) = -2\ck \cos(x)\,,d_1=-3\cos(x),e_1=3\sin(x)$, we obtain
\begin{align*}
   \eqref{A5c}f_0^+& = \frac{1}{2\pi \im} \oint_{\Gamma} (\mathscr{L}_{0,0} - \lambda)^{-1} \dot{\mathscr{L}}'_{0,0}\left(\frac{-1}{\lambda}f_0^++\frac{1}{\lambda^2}f_0^-\right) d\lambda\,,
   \end{align*}
which is a vector with zero average. \textcolor{black}{All the remaining contributions to
$\dot P'_{0,0}f_0^+$ are purely imaginary vectors with the
parity and zero-average properties displayed in \eqref{A17}.
The only real contribution is \eqref{BB20}.
Since
\[
f_{-1}^+\in\mathcal{U}\subset\ker(P_{0,0}),
\]
the correction term
$-\frac12P_{0,0}\dot P'_{0,0}f_0^+$
does not alter this resonant component. Consequently,
\[
(\partial_\mu\partial_\epsilon f_0^+)(0,0)
=
\frac{1}{4\ck^{3/2}}f_{-1}^+
+\im
\begin{bmatrix}
\mathit{odd}(x)\\
\mathit{even}_0(x)
\end{bmatrix},
\]
as claimed in \eqref{A17}.} We finally consider $(\partial_\mu \partial_\epsilon f_0^-(0,0))$. By \eqref{A11} and $\mathscr{L}_{0,0}^{\prime} f_0^- = 0$ (cf. \eqref{A15}), it follows that for M=I,II,
\begin{equation}
\eqref{A5a}^{M}f_0^- = -\frac{1}{2\pi \im} \oint_{\Gamma} (\mathscr{L}_{0,0} - \lambda)^{-1} \dot{\mathscr{L}}_{0,0}^{M} (\mathscr{L}_{0,0} - \lambda)^{-1} \mathscr{L}'_{0,0} \frac{-1}{\lambda} f_0^-d\lambda=0\,.
\end{equation}
Next, by \eqref{A11} and $\dot{\mathscr{L}}_{0,0}^{I} f_0^- = 0$, we obtain $\eqref{A5b}^{I}f_0^- = 0$. Moreover, since $\dot{\mathscr{L}}_{0,0}^{II}f_0^-=f_0^+$, we have
\begin{align*}
   \eqref{A5b}^{II}f_0^-& =-\frac{1}{2\pi \im} \oint_{\Gamma} (\mathscr{L}_{0,0} - \lambda)^{-1}\mathscr{L}'_{0,0}  (\mathscr{L}_{0,0} - \lambda)^{-1}  \dot{\mathscr{L}}_{0,0}^{II}\frac{-1}{\lambda}f_0^-d\lambda\\
   & =-\frac{1}{2\pi \im} \oint_{\Gamma} (\mathscr{L}_{0,0} - \lambda)^{-1}\mathscr{L}'_{0,0}  \frac{-1}{\lambda}\left(\frac{-1}{\lambda}f_0^++\frac{1}{\lambda^2}f_0^-\right)d\lambda\\
   & =-\frac{1}{2\pi \im} \oint_{\Gamma} (\mathscr{L}_{0,0} - \lambda)^{-1}\frac{1}{\lambda^2}\vet{2\ck\sin(x)}{2\ck^2\cos(x)} d\lambda= \frac{1}{2}\begin{bmatrix}\frac{1}{\ck}\sin(x) \\ \cos(x) \end{bmatrix}\,.
   \end{align*}
Finally, by \eqref{A11} and \eqref{A8},
\begin{equation}
 \eqref{A5c} f_0^- = -\frac{1}{2\pi \im} \oint_{\Gamma} (\mathscr{L}_{0,0} - \lambda)^{-1} \frac{1}{\lambda} \begin{bmatrix} 0 \\ \im p_1(x)\end{bmatrix} d \lambda\,.
\end{equation} 
This vector has zero average since $(\mathscr{L}_{0,0} - \lambda)^{-1}$ is a Fourier multiplier and therefore preserves the zero-average property.

This completes the proof of Lemma \ref{expansion of the basis F}.
\end{proof}

\footnotesize

\vspace{1em}
\noindent{\bf Acknowledgments:}
 (i) T.-Y. Hsiao is  supported by the European Union ERC CONSOLIDATOR GRANT 2023 GUnDHam, Project Number: 101124921. Views and opinions expressed are however those of the authors only and do not necessarily reflect those of the European Union or the European Research Council. Neither the European Union nor the granting authority can be held responsible for them. (ii) X. Wang is supported by the National Key R$\&$D Program of China, Grant No. 2021YFA1000800. The authors would like to express their sincere gratitude to Prof. Alberto Maspero and Prof. Massimiliano Berti for their helpful guidance and constructive suggestions. In particular, we are grateful to Prof. Berti for his insightful explanations on Hamiltonian PDEs. We also express our sincere appreciation to Prof. Maspero for his patient guidance, invaluable advice, and continuous encouragement throughout this work. The authors thank the Associate Editor and the two anonymous referees
for their careful reading and helpful comments, which improved the manuscript. T.-Y. Hsiao also thanks Dr. Stefano Pasquali for pointing out the typos in the expansion of the basis.

\end{document}